\documentclass[11pt]{article}
\usepackage{color}
\usepackage[colorlinks=true]{hyperref}
\usepackage[margin=1in]{geometry}
\usepackage{amssymb}
\usepackage{amsmath}
\usepackage{amsthm}
\usepackage{mathtools}
\usepackage{epsfig, graphicx}
\usepackage{lipsum}
\usepackage{amsfonts}
\usepackage{graphicx}
\usepackage{epstopdf}
\usepackage[boxed, noline, ruled, linesnumbered]{algorithm2e}
\usepackage{algorithmic}
\usepackage{stmaryrd}
\usepackage{cleveref}
\newtheorem{theorem}{Theorem}[section]
\newtheorem{lemma}{Lemma}[section]

\newtheorem{remark}{Remark}[section]

\makeatletter 

\@addtoreset{equation}{section}
\makeatother 
\allowdisplaybreaks 

\begin{document}
\title{Solving High Dimensional Partial Differential Equations Using  
Tensor Neural Network and A Posteriori Error Estimators\footnote{This work was
supported in part by the National Key Research and Development Program of China
(2019YFA0709601), the National Center for Mathematics and Interdisciplinary Science, CAS.}}
\author{ 
Yifan Wang\footnote{LSEC, NCMIS, Institute
of Computational Mathematics, Academy of Mathematics and Systems
Science, Chinese Academy of Sciences, Beijing 100190,
China,  and School of Mathematical Sciences, University
of Chinese Academy of Sciences, Beijing 100049, China (wangyifan@lsec.cc.ac.cn).}, 
\ \ 
Zhongshuo Lin\footnote{LSEC, NCMIS, Institute
of Computational Mathematics, Academy of Mathematics and Systems
Science, Chinese Academy of Sciences, Beijing 100190,
China,  and School of Mathematical Sciences, University
of Chinese Academy of Sciences, Beijing 100049, China (linzhongshuo@lsec.cc.ac.cn).}, \ \ 
Yangfei Liao\footnote{LSEC, NCMIS, Institute
of Computational Mathematics, Academy of Mathematics and Systems
Science, Chinese Academy of Sciences, Beijing 100190,
China,  and School of Mathematical Sciences, University
of Chinese Academy of Sciences, Beijing 100049, China (liaoyangfei@lsec.cc.ac.cn).}, \ \ 
Haochen Liu\footnote{LSEC, NCMIS, Institute
of Computational Mathematics, Academy of Mathematics and Systems
Science, Chinese Academy of Sciences, Beijing 100190,
China,  and School of Mathematical Sciences, University
of Chinese Academy of Sciences, Beijing 100049, China (liuhaochen@lsec.cc.ac.cn).}\ \ \  and \ \
Hehu Xie\footnote{LSEC, NCMIS, Institute
of Computational Mathematics, Academy of Mathematics and Systems
Science, Chinese Academy of Sciences, Beijing 100190,
China,  and School of Mathematical Sciences, University
of Chinese Academy of Sciences, Beijing 100049, China (hhxie@lsec.cc.ac.cn).}}

\date{}
\maketitle

\begin{abstract}
In this paper, based on the combination of tensor neural network and 
a posteriori error estimator, a novel type of machine learning method is proposed 
to solve high-dimensional boundary value problems with homogeneous and non-homogeneous 
Dirichlet or Neumann type of boundary conditions and eigenvalue problems 
of the second-order elliptic operator. The most important advantage of the tensor neural  
network is that the high dimensional integrations of tensor neural networks 
can be computed with high accuracy and high efficiency. 
Based on this advantage and the theory of a posteriori error estimation, 
the a posteriori error estimator is adopted to design the loss function to 
optimize the network parameters adaptively. The applications of tensor neural network 
and the a posteriori error estimator improve the accuracy of the corresponding 
machine learning method. The theoretical analysis and numerical examples 
are provided to validate the proposed methods.

\end{abstract}

\section{Introduction}
In this paper, we are concerned with the high dimensional problems associated with the 
following second-order elliptic operator
\begin{equation}\label{elliptic_operator}
Lu(x):= -\nabla\cdot(\mathcal A\nabla u(x))+  b(x)u(x),\ \ \ x\in\Omega,
\end{equation} 
where $\Omega:=(a_1, b_1)\times\cdots \times (a_d,b_d)$ is 
a $d$-dimensional domain ($d>3$), $\mathcal A \in \mathbb R^{d\times d}$ is a symmetric 
positive definite matrix and the function $b(x)$ has a positive lower bound.
Here, we will discuss the machine learning type of numerical methods 
for solving boundary value problems with homogeneous and non-homogeneous 
Dirichlet or Neumann type of  boundary conditions,  eigenvalue problems of the operator $L$, 
respectively.

It is well known that solving partial differential equations (PDEs) is one of the 
most essential tasks in modern science and engineering society.  
There have developed successful numerical methods such as finite difference, 
finite element, and spectral method for solving PDEs in three spatial dimensions 
plus the temporal dimension. Meanwhile, many high-dimensional PDEs exist, 
which are nearly impossible to be 
solved using traditional numerical methods, 
such as many-body Schr\"{o}dinger, Boltzmann equations, Fokker-Planck equations, 
and stochastic PDEs (SPDEs). 
Due to its universal approximation property, the fully connected neural network (FNN) 
is the most widely used architecture to build the functions for solving 
high-dimensional PDEs. There are several types of well-known FNN-based methods 
such as deep Ritz \cite{EYu}, deep Galerkin method \cite{DGM}, 
PINN \cite{RaissiPerdikarisKarniadakis}, and weak adversarial networks \cite{WAN}
for solving high-dimensional PDEs by designing different loss functions. 
Among these methods, the loss functions always include computing 
high-dimensional integration for the functions defined by FNN. 
For example,  the loss functions of the deep Ritz method require computing 
the integrations on the high-dimensional domain for the functions constructed 
by FNN. Direct numerical integration for the high-dimensional functions also 
meets the ``curse of dimensionality''. 
Always, the high-dimensional integration is computed using the Monte-Carlo 
method along with some sampling tricks \cite{EYu,HanZhangE}. 
Due to the low convergence rate of the Monte-Carlo method, the solutions obtained 
by the FNN-based numerical methods are challenging to achieve high accuracy and 
stable convergence process. This means that the Monte-Carlo method decreases 
computational work in each forward propagation while decreasing the simulation 
accuracy, efficiency and stability of the FNN-based numerical methods 
for solving high-dimensional PDEs. When solving nonhomogeneous 
boundary value problems, it is difficult to 
choose the numbers of sampling points on the boundary and in the domain. 
Furthermore, for  solving non-homogeneous Dirichlet boundary value problems, besides 
the difficulty of choosing sampling points, it is also very difficult to set the 
hyperparameter to balance the loss from the boundary and interior domain.

Recently, we propose a type of tensor neural network (TNN) and 
the corresponding machine learning method,  
aiming to solve high dimensional problems with high accuracy 
\cite{WangJinXie,WangLiaoXie,WangXie}. 
The most important property of TNN is that the corresponding high-dimensional 
functions can be easily integrated with high accuracy and high efficiency. 
Then, the deduced machine 
learning method can achieve high accuracy in solving high-dimensional problems. 
The reason is that the integration of TNN functions can be separated into one-dimensional 
integrations which can be computed by classical quadratures with high accuracy.  
The TNN-based machine learning method has already been used to solve high-dimensional 
eigenvalue problems and boundary value problems based on the Ritz type of loss functions. 
Furthermore,  in \cite{WangXie},  the multi-eigenpairs can also be computed with 
the machine learning method designed by combining the TNN and Rayleigh-Ritz process. 
The high accuracy of this type of machine learning method depends on the essential 
fact that the high dimensional integrations associated with TNNs 
can be computed with high accuracy. The TNN is also used to solve 
20,000 dimensional Schr\"{o}dinger equation with coupled quantum harmonic 
oscillator potential function \cite{HuShuklaKarniadakisKawaguchi}, high 
dimensional  Fokker-Planck equations \cite{WangHuKawaguchiZhangKarniadakis} 
and high dimensional time-dependent problems \cite{KaoZhaoZhang}.

In the finite element method, the a posteriori error estimate is a 
standard technique for solving the 
PDEs with singularity. With the a posteriori error estimate, we can know the 
accuracy of the concerned approximations. The a posteriori error estimate based on 
the hypercycle technique can give the guaranteed upper bound of the approximations 
of the elliptic type of PDEs \cite{ainsworth1999reliable,ainsworth2011fully, 
vejchodsky2012complementarity}.

In this paper, we design a new machine learning method by combining the 
TNN and the a posteriori error estimate for
solving high-dimensional PDEs associated with 
differential operator (\ref{elliptic_operator}).   
Different from the existed work, the a posteriori error estimate 
is first adopted to build the loss function for the  
TNN-based machine learning method.   
Then, the loss function directly bounds the error of the TNN approximations for 
the high dimensional problems. The essential idea is that we can compute 
the high-dimensional integrations of TNN functions included 
in the a posteriori error estimators with high accuracy.  
Different from common training methods, the training step is decomposed into 
two substeps including the Galerkin step for the coefficient 
and the optimization step for updating the neural networks. 
This separation scheme obviously improves the accuracy of TNN based machine learning method.
The application of the a posteriori error estimator improvers the accuracy of the 
corresponding machine learning method.  We will also find that there 
is no difficulty in choosing sampling points and balancing hyperparameter 
for solving non-homogeneous Dirichlet and Neumann 
boundary value problems with TNN based machine learning method.

An outline of the paper goes as follows. In Section \ref{Section_TNN}, 
we introduce the structure to build TNN and its numerical integration method and 
approximation property. 
Section \ref{Section_Posteriori} is devoted to proposing the a posteriori error estimates 
for the concerned problems in this paper. 
The TNN-based machine learning method based on the a posteriori error 
estimators and optimization process is designed 
in Section \ref{TNN_Machine_Learning}. 
Section \ref{Section_Numerical} gives some numerical examples to validate the accuracy and 
efficiency of the proposed TNN-based machine learning method. 
Some concluding remarks are given in the last section.

\section{Tensor neural network and its quadrature scheme}\label{Section_TNN}
In this section, we introduce the TNN structure and the quadrature scheme for 
the high-dimensional TNN functions.
Also included in this section is a discussion of the approximation properties, 
some techniques to improve the numerical stability, 
the complexity estimate of the high dimensional integrations of the TNN functions.

\subsection{Tensor neural network architecture}\label{Section_TNN_Archictecture}
This subsection is devoted to introducing the TNN structure and some techniques 
to improve the stability of the corresponding machine learning methods.
The approximation properties of TNN functions have been discussed and investigated in \cite{WangJinXie}.
In order to express clearly and facilitate the construction of the TNN method for solving
high-dimensional PDEs, here, we will also introduce some important definitions and properties.

TNN is built by the tensor products of one-dimensional functions which come from
$d$ subnetworks with one-dimensional input and multidimensional output, where
$d$ denotes the spatial dimension of the concerned problems which will be solved by the
machine learning method in this paper. 
For each $i=1,2,\cdots,d$, we use $\Phi_i(x_i;\theta_i)=(\phi_{i,1}(x_i;\theta_i),
\phi_{i,2}(x_i;\theta_i), \cdots,\phi_{i,p}(x_i;\theta_i))$
to denote a subnetwork that maps a set $\Omega_i\subset\mathbb R$ to $\mathbb R^p$,
where $\Omega_i,i=1,\cdots,d,$ can be a bounded interval $(a_i,b_i)$, 
the whole line $(-\infty,+\infty)$
or the half line $(a_i,+\infty)$ \cite{WangXie}.
The number of layers and neurons in each layer, the selections of 
activation functions and other hyperparameters can be different in different subnetworks. 
TNN consists of $p$ correlated rank-one functions,
which are composed of the multiplication of $d$ one-dimensional input functions 
in different directions.
Figure \ref{TNNstructure} shows the corresponding architecture of TNN.

In order to improve the numerical stability, we normalize each $\phi_{i,j}(x_i)$
and use the following normalized TNN structure:
\begin{eqnarray}\label{def_TNN_normed}
\Psi(x;c,\theta)&=&\sum_{j=1}^pc_j\widehat\phi_{1,j}(x_1;\theta_1)
\cdots\widehat\phi_{d,j}(x_d;\theta_d)=\sum_{j=1}^pc_j\prod_{i=1}^d\widehat\phi_{i,j}(x_i;\theta_i),
\end{eqnarray}
where each $c_j$ is a scaling parameter which describes 
the length of each rank-one function, $c=\{c_j\}_{j=1}^{p}$ 
is a set of trainable parameters, $\{c,\theta\}=\{c,\theta_1,\cdots,\theta_d\}$
denotes all parameters of the whole architecture.
For $i=1,\cdots,d,j=1,\cdots,p$, $\widehat\phi_{i,j}(x_i;\theta_i)$ is a 
$L^2$-normalized function as follows:
\begin{eqnarray*}
\widehat\phi_{i,j}(x_i;\theta_i)
=\frac{\phi_{i,j}(x_i;\theta_i)}{\|\phi_{i,j}(x_i;\theta_i)\|_{L^2(\Omega_i)}}.
\end{eqnarray*}
For simplicity of notation, $\phi_{i,j}(x_i;\theta_i)$ 
denotes the normalized function in the following parts.

The TNN architecture (\ref{def_TNN_normed}) and the architecture defined 
in \cite{WangJinXie} are mathematically equivalent, but (\ref{def_TNN_normed}) 
has better numerical stability during the training process.
From Figure \ref{TNNstructure} and numerical tests, we can find the parameters for 
each rank of TNN are correlated by the FNN, which guarantees the stability 
of the TNN-based machine learning methods. This is also an important difference 
from the tensor finite element methods.
\begin{figure}[htb]
\centering
\includegraphics[width=10cm,height=7.5cm]{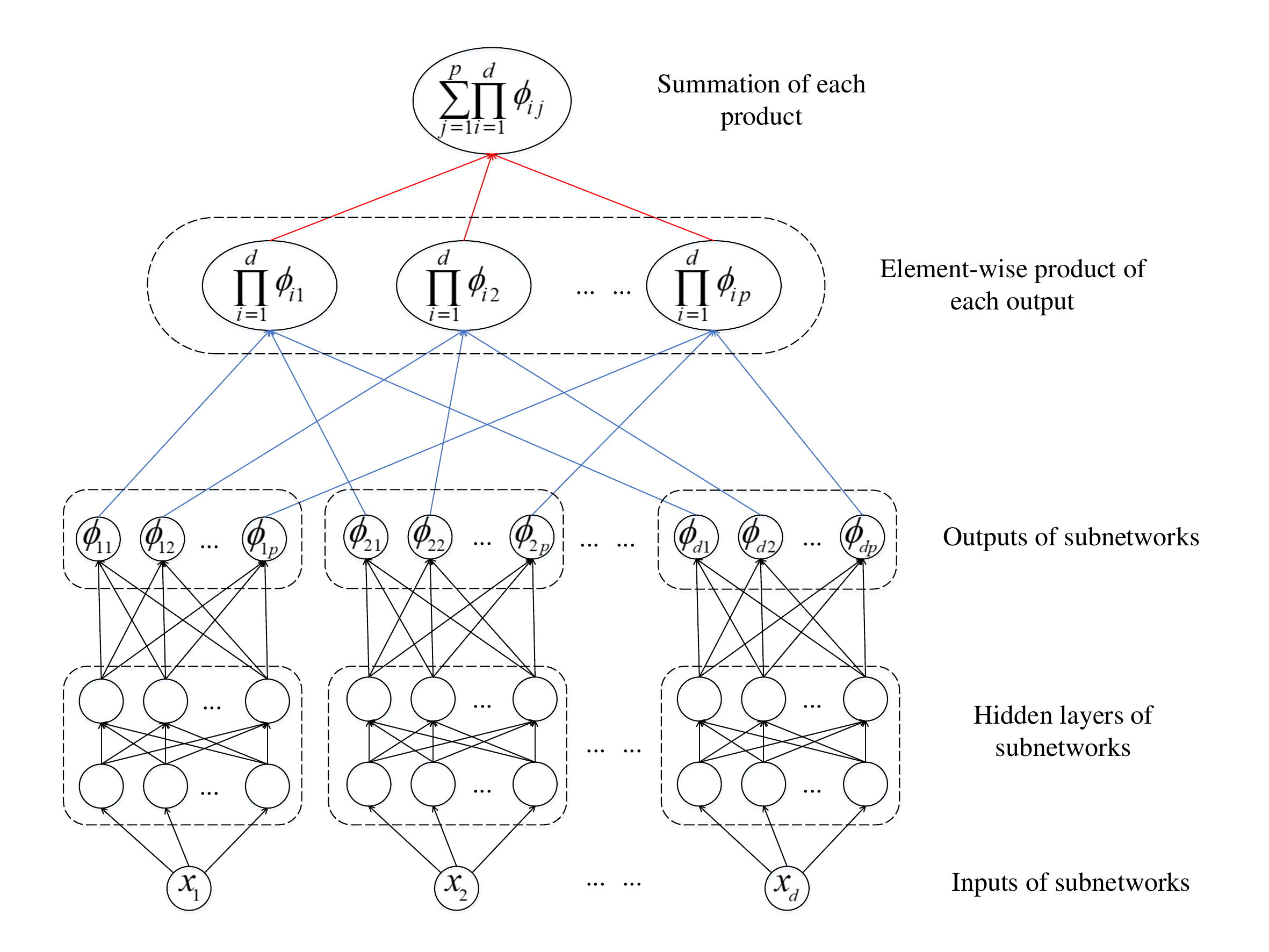}
\caption{Architecture of TNN. Black arrows mean linear transformation
(or affine transformation). Each ending node of blue arrows is obtained by taking the
scalar multiplication of all starting nodes of blue arrows that end in this ending node.
The final output of TNN is derived from the summation of all starting nodes of red 
arrows.}\label{TNNstructure}
\end{figure}

In order to show the reasonableness of TNN, we now introduce the approximation property 
from \cite{WangJinXie}. Since there exists the isomorphism relation between $H^m(\Omega_1\times\cdots\times\Omega_d)$
and the tensor product space $H^m(\Omega_1)\otimes\cdots\otimes H^m(\Omega_d)$,
the process of approximating the function $f(x)\in H^m(\Omega_1\times\cdots\times\Omega_d)$
by the TNN defined as (\ref{def_TNN_normed}) can be regarded as searching for 
a correlated CP decomposition structure
to approximate $f(x)$ in the space $H^m(\Omega_1)\otimes\cdots\otimes H^m(\Omega_d)$
with the rank being not greater than $p$.
The following approximation result to the functions in the space 
$H^m(\Omega_1\times\cdots\times\Omega_d)$ under the sense of $H^m$-norm is 
proved in \cite{WangJinXie}.
\begin{theorem}\cite[Theorem 1]{WangJinXie}\label{theorem_approximation}
Assume that each $\Omega_i$ is an interval in $\mathbb R$ for $i=1, \cdots, d$, $\Omega=\Omega_1\times\cdots\times\Omega_d$,
and the function $f(x)\in H^m(\Omega)$. Then for any tolerance $\varepsilon>0$, there exist a
positive integer $p$ and the corresponding TNN defined by (\ref{def_TNN_normed})
such that the following approximation property holds
\begin{equation*}
\|f(x)-\Psi(x;\theta)\|_{H^m(\Omega)}<\varepsilon.
\end{equation*}
\end{theorem}
Although there is no general result to give the relationship between the hyperparameter $p$ 
and error bound $\varepsilon$,
we also provided an estimate of the rank $p$ under a smoothness assumption. 
For easy understanding, we focus on the periodic setting with 
$I^d=I\times I\times\cdots\times I=[0,2\pi]^d$ and 
the approximations property of TNN to the functions in the linear space which 
is defined with Fourier basis.
Note that similar approximation results of TNN can be extended to the non-periodic functions. 
For each variable $x_i\in[0,2\pi]$, let us define the one-dimensional 
Fourier basis $\{\varphi_{k_i}(x_i):= \frac{1}{\sqrt{2\pi}}e^{{\rm i}k_ix_i},k_i\in\mathbb Z\}$ 
and classify functions via the decay of their Fourier coefficients.
Further denote multi-index $k=(k_1,\cdots,k_d)\in\mathbb Z^d$ and $x=(x_1,\cdots,x_d)\in I^d$. 
Then the $d$-dimensional Fourier basis can be built in the tensor product way 
\begin{eqnarray*}
\varphi_k(x):=\prod_{i=1}^d\varphi_{k_i}(x_i)=(2\pi)^{-d/2}e^{{\rm i}k\cdot x}.
\end{eqnarray*}
We denote $\lambda_{\rm mix}(k):=\prod_{i=1}^d\left(1+|k_i|\right)$ and
$\lambda_{\rm iso}(k):= 1+\sum_{i=1}^d|k_i|$. 
Now, for $-\infty<t,\ell<\infty$, we define the space $H_{\rm mix}^{t,\ell}(I^d)$ as follows 
(cf. \cite{GriebelHamaekers,Knapek})
{\footnotesize 
\begin{eqnarray*}
H_{\rm mix}^{t,\ell}(I^d)=\left\{u(x)=\sum_{k\in\mathbb Z^d}c_k\varphi_k(x):
\|u\|_{H_{\rm mix}^{t,\ell}(I^d)}=\left(\sum_{k\in\mathbb Z^d}\lambda_{\rm mix}(k)^{2t}
\cdot\lambda_{\rm iso}(k)^{2\ell}\cdot|c_k|^2\right)^{1/2}<\infty \right\}. 
\end{eqnarray*}}
Note that the parameter $\ell$ governs the isotropic smoothness, whereas $t$ 
governs the mixed smoothness.
The space $H_{\rm mix}^{t,\ell}(I^d)$ gives a quite flexible framework for 
the study of problems in Sobolev spaces. 
For more information about the space $H_{\rm mix}^{t,\ell}(I^d)$, please refer to \cite{GriebelHamaekers,GriebelKnapek,Knapek}.

Thus, \cite{WangJinXie} gives the following comprehensive error estimate for TNN.
\begin{theorem}\label{theorem_aprrox_rate}
Assume function $f(x)\in H_{\rm mix}^{t,\ell}(I^d)$, $t>0$ and $m>\ell$. 
Then there exists a TNN $\Psi(x;\theta)$ defined by (\ref{def_TNN_normed}) 
such that the following approximation property holds
\begin{eqnarray}
\|f(x)-\Psi(x;\theta)\|_{H^m(I^d)}\leq C(d)\cdot p^{-(\ell-m+t)}\cdot
\|u\|_{H_{\rm mix}^{t,\ell}(I^d)},
\end{eqnarray}
where $C(d)\leq c\cdot d^2\cdot0.97515^d$ with $c$ being independent of $d$. 
And each subnetwork of TNN is a FNN which is built by using $\sin(x)$ as 
the activation function and one hidden layer with $2p$ neurons, see Figure \ref{TNNstructure}. 
\end{theorem}

The TNN-based machine learning method in this paper will 
adaptively select $p$ rank-one functions by the training process. 
From the approximation result in Theorem \ref{theorem_aprrox_rate}, 
when the target function belongs to $H_{\rm mix}^{t,\ell}(\Omega)$,  
there exists a TNN with $p\sim\mathcal O(\varepsilon^{-(m-\ell-t)})$ 
such that the accuracy is $\varepsilon$.
For more details about the approximation properties of TNN, 
please refer to \cite{WangJinXie}.

\subsection{Quadrature scheme for TNN}\label{section_quad}
In this subsection, we introduce the quadrature scheme for computing 
the high-dimensional integrations of the high dimensional TNN functions. 
Due to the low-rank property of TNN, the efficient and accurate quadrature 
scheme can be designed for the TNN-related high-dimensional integrations 
which are included in the loss functions for machine learning methods. 
When calculating the high dimensional integrations in the loss functions,
we only need to calculate
one-dimensional integrations as well as their product.
The method to compute the numerical integrations of polynomial composite
functions of TNN and their derivatives has already been designed in \cite{WangJinXie}.
For more information about the quadrature schemes, computational complexity 
and high accuracy of the numerical integrations of high dimensional TNN functions, 
please refer to \cite{WangJinXie}.

For easy understanding, we give the detailed quadrature scheme to compute the energy 
norm of a TNN function $\Psi$. For simplicity, we assume that $b (x)$ 
in the second order elliptic operator  (\ref{elliptic_operator}) has the following expansion
\begin{eqnarray*}
b(x)=\sum_{j=1}^q\prod_{i=1}^d b_{i,j}(x_i).
\end{eqnarray*}
Then the corresponding energy norm of TNN $\Psi(x)$ can be computed as follows
\begin{eqnarray*}
&&\|\Psi\|_a=(\mathcal A\nabla\Psi,\nabla\Psi)+(b\Psi,\Psi)\\
&\approx&\sum_{j=1}^p\sum_{k=1}^p\sum_{\substack{s,t=1,\cdots,d,\\s\neq t}}\mathcal A_{s,t}\prod_{\substack{i=1,\cdots,d,\\i\neq s,t}}\left(\sum_{n_i=1}^{N_i}w_i^{(n_i)}\phi_{i,j}(x_i^{(n_i)})\phi_{i,j}(x_i^{(n_i)})\right)\\
&&\cdot\left(\sum_{n_s=1}^{N_s}w_s^{(n_s)}\frac{\partial\phi_{s,j}}{\partial x_s}(x_s^{(n_s)})\phi_{s,k}(x_s^{(n_s)})\right)
\cdot\left(\sum_{n_t=1}^{N_t}w_t^{(n_t)}\phi_{t,j}(x_t^{(n_t)})\frac{\partial\phi_{t,k}}{\partial x_t}(x_t^{(n_t)})\right)\\
&&+\sum_{j=1}^p\sum_{k=1}^p\sum_{s=1}^d\mathcal A_{s,s}\prod_{\substack{i=1,\cdots,d,\\i\neq s}}^d\left(\sum_{n_i=1}^{N_i}w_i^{(n_i)}\phi_{i,j}(x_i^{(n_i)})\phi_{i,k}(x_i^{(n_i)})\right)\\
&&\cdot\left(\sum_{x_s=1}^{N_s}w_s^{(n_s)}\frac{\partial\phi_{s,j}}{\partial x_s}(x_s^{(n_s)})
\frac{\partial\phi_{s,k}}{\partial x_s}(x_s^{(n_s)})\right)\\
&&+\sum_{\ell=1}^q\sum_{j=1}^p\sum_{k=1}^p\prod_{i=1}^d
\left(\sum_{i=1}^{N_i}w_i^{(n_i)}b_{i,\ell}(x_i^{(n_i)})\phi_{i,j}(x_i^{(n_i)})
\phi_{i,k}(x_i^{(n_i)})\right).
\end{eqnarray*}
As in the above equation, using the structure of TNN, we can convert the 
$d$-dimensional integration calculation into a series of one-dimensional integration calculations.
The application of TNN can bring a significant reduction of the computational complexity 
for the related numerical integrations.
In this paper, when the TNN-based method is used to solve second-order elliptic problems, 
all integrations can be calculated in a way similar to that for the energy norm.

\begin{remark}
It is worth mentioning that a similar quadrature scheme can also be given for unbounded $\Omega_i$.
In \cite{WangXie}, in addition to the Legendre-Gauss quadrature scheme for a bounded domain, 
we discuss the Hermite-Gauss quadrature scheme for the whole line $\Omega_i=(-\infty,+\infty)$ 
and Laguerre-Gauss quadrature scheme for the half line $\Omega_i=(a_i,+\infty)$.
The computational complexity of these integrations is also polynomial order of $d$.
\end{remark}

\subsection{Choosing space $V_p$ adaptively}\label{section_choose_Vp}
Consider a specific one-dimensional function $f(x)=\sin(nx)$ with $n\in\mathbb Z$, $f$ only 
has a frequency $n$.
When using the Fourier basis to approximate $f$, a common way is to use the first $N$ 
Fourier basis functions to span an $N$-dimensional trial function space, which denotes $F_N$.
Then, find the approximation of $f$ in space $F_N$.
Due to the orthogonal property of the Fourier basis, the dimension of the space $F_N$ is at least $n$.
That is, the basis functions with frequencies not greater than $n$ are all selected into the trial 
function space $F_N$.
Since the expansion of function $f$ does not contain frequencies less than $n$, this result 
wastes the number of dimensions of trial function space.
By contrast, when using TNN to approximate $f(x)$, we can 
construct an one-hidden-layer TNN with $\sin(x)$ as the activation function. 
The critical point is that a single parameter, $p=1$, is sufficient. 
In other words, we only need an one-dimensional subspace $V_p$.
This difference in dimensionality between spaces inspires us to adopt 
the idea of the adaptive method, where finding the best subspace with $p$ terms 
could be a more effective approach.

Fortunately, considering the training process of TNN, we can find that this training process 
is naturally an adaptive process. After the $\ell$-th training step, 
TNN $\Psi(x;\theta^{(\ell)})$ belongs to the following subspace
\begin{eqnarray}\label{def_Vpl}
V_{p}^{(\ell)}:=\operatorname{span}\left\{\varphi_{j}(x ; \theta^{(\ell)}) 
:= \prod_{i=1}^{d} \widehat\phi_{i, j}\left(x_{i} ; \theta_{i}^{(\ell)}\right), j=1, \cdots, p\right\}.
\end{eqnarray}
Then, the parameters of the $\ell+1$-th step are updated according 
to the optimization step with a loss function $\mathcal{L}$. 
If gradient descent is used, this update can be expressed as follows
\begin{eqnarray}\label{Optimization_Step}
\theta^{(\ell+1)} \leftarrow \theta^{(\ell)} - 
\gamma \frac{\partial\mathcal{L}(\Psi(x;\theta^{(\ell)})}{\partial \theta},
\end{eqnarray}
where $\gamma$ denotes the learning rate.
Note that updating parameter $\theta^{(\ell)}$ essentially updates the subspace $V_p^{(\ell+1)}$. 
In other words, in the training process of TNN, an optimal $p$-dimensional 
subspace $V_p$ can be selected adaptively according to the loss function.
The definition of the loss function determines how the subspace is updated and whether the 
algorithm can ultimately achieve high precision.

In Section \ref{Section_Posteriori}, we will introduce the theory of a posteriori error estimation for 
second-order elliptic problems and construct a loss function that is convenient for calculation.

\section{A posteriori error estimates for subspace approximation}\label{Section_Posteriori}
In this section, we introduce the way to compute the a posteriori error estimates 
for the problems associated with the second-order elliptic operator (\ref{elliptic_operator}). 
Here, the a posteriori error estimators are given for the boundary value problems with 
Dirichlet and Neumann boundary conditions, and eigenvalue problem. 
The content of this section comes from \cite{ainsworth1999reliable,ainsworth2011fully, vejchodsky2012complementarity} 
which are developed for solving 
low dimensional PDEs by the traditional finite element methods. 
Based on the special properties of TNN, we can borrow them to build the machine learning method 
for solving high dimensional PDEs. For easy reading and understanding, 
we also give a detailed description here. 

The following formula of integration by parts acts as the key role to derive the a
posteriori error estimates here. 
\begin{lemma}\label{lemma_Green}
Let $\Omega\subset \mathbb R^d$ be a bounded Lipschitz domain with unit outward normal 
$\mathbf n$ to the boundary $\partial\Omega$.  Then the following Green’s formula holds
\begin{eqnarray*}
\int_\Omega v{\rm div}\mathbf yd\Omega+\int_\Omega \mathbf y\cdot\nabla vd\Omega
=\int_{\partial\Omega}v\mathbf y\cdot\mathbf n ds,\ \ \ \forall v\in H^1(\Omega),
\ \forall\mathbf y\in \mathbf W,
\end{eqnarray*}
where $\mathbf W:= \mathbf H({\rm div},\Omega)$.
\end{lemma}
For simplicity of notation, we use $V$ to denote the concerned Sobolev space for 
different type of problems. So it has different definitions for different problems.

\subsection{A posteriori error estimator for homogeneous Dirichlet boundary value problem}\label{section_post_dirichlet} 
For easy understanding, we first introduce the
a posterior error estimate for the second-order elliptic 
problem with homogeneous Dirichlet boundary conditions. 

Here, we are concerned with the following high dimensional boundary value problem:  
Find $u$ such that
\begin{equation}\label{ModelProblem}
\left\{
\begin{array}{rcl}
-\nabla\cdot(\mathcal A\nabla u)+  b u&=&f, \quad {\rm in} \  \Omega,\\
u&=&0, \quad {\rm on}\  \partial\Omega,
\end{array}
\right.
\end{equation}
where 
the right hand side term $f\in L^2(\Omega)$.

The weak form of problem (\ref{ModelProblem}) can be described as: Find $u\in V$ such that
\begin{eqnarray}\label{weak_form_hDirichlet}
a(u,v)=(f,v),\ \ \ \forall v\in V, 
\end{eqnarray}
where $V:=H_0^1(\Omega)$. 
Here and herafter in this paper, the bilinear form $a(\cdot, \cdot)$ is defined as follows  
\begin{eqnarray}\label{Bilinear_Form}
a(u,v)=(\mathcal A\nabla u,\nabla v)+(b u,v). 
\end{eqnarray}
The energy norm $\|\cdot\|_a$ is defined as $\|v\|_a=\sqrt{a(v,v)}$.
Assume we have a $p$-dimentional subspace $V_p$ which satisfies $V_p\subset V$.
The standard Galerkin approximation scheme for the problem (\ref{weak_form_hDirichlet}) is: 
Find $u_p\in V_p$ such that
\begin{eqnarray}\label{weak_form_hDirichlet_Discrete}
a(u_p,v_p)=(f,v_p),\ \ \ \forall v_p\in V_p.
\end{eqnarray}
Then the Galerkin approximation $u_p$ satisfies the following a posteriori error estimation.

\begin{theorem}\cite[Theorem 1]{vejchodsky2012complementarity}\label{Theorem_hDirichlet}
Assume $u\in V$ and $u_p\in V_p \subset V$ are the exact solution of 
(\ref{weak_form_hDirichlet}) and subspace approximation $u_p$ 
defined by (\ref{weak_form_hDirichlet_Discrete}), respectively. 
The following upper bound holds
\begin{equation}\label{Upper_Bound_U}
\|u-u_p\|_a \leq  \eta(u_p,\mathbf y),\ \ \ \forall \mathbf y\in \mathbf W,
\end{equation}
where $\eta(u_p,\mathbf y)$ is defined as follows
\begin{equation}\label{eta_hDirichlet}
\eta(u_p,\mathbf y):=\left(\|b^{-\frac{1}{2}}(f-b u_p+{\rm div}(\mathcal A\mathbf y))\|_0^2
+\|\mathcal A^{\frac{1}{2}}(\mathbf y-\nabla u_p)\|_0^2\right)^{\frac{1}{2}}.
\end{equation}
\end{theorem}
This theorem has been proved in the references. For understanding, we also provide the proof
here since its simplicity.
\begin{proof}
Let us define $v= u-u_p\in V$ in this proof.
Then by Lemma \ref{lemma_Green} and 
Cauchy-Schwarz inequality, the following estimates hold
\begin{eqnarray}\label{Inequality_1}
&&a(u-u_p,v)= (f,v)-(\mathcal A\nabla u_p,\nabla v)-(b u_p,v)\nonumber\\
&=&(f,v)-(\mathcal A\nabla u_p,\nabla v)-(b u_p,v)+(\mathcal A\mathbf y,\nabla v)
+({\rm div}(\mathcal A\mathbf y),v)\nonumber\\
&=&(f-b u_p+{\rm div}(\mathcal A\mathbf y),  v)
+(\mathcal A(\mathbf y-\nabla u_p),\nabla v)\nonumber\\
&\leq& \|b^{-\frac{1}{2}}(f-b u_p+{\rm div}(\mathcal A\mathbf y))\|_0
\|b^{\frac{1}{2}}v\|_0
+\|\mathcal A^{\frac{1}{2}}(\mathbf y-\nabla u_p)\|_0
\|\mathcal A^{\frac{1}{2}}\nabla v\|_0\nonumber\\
&\leq& \left(\|b^{-\frac{1}{2}}(f-b u_p+{\rm div}(\mathcal A\mathbf y))\|_0^2
+\|\mathcal A^{\frac{1}{2}}(\mathbf y-\nabla u_p)\|_0^2\right)^{1/2}\|v\|_a,\ \ \ 
\forall\mathbf y\in \mathbf W.
\end{eqnarray}
This is the desired result (\ref{Upper_Bound_U}) and the proof is complete.
\end{proof}

\subsection{A posteriori error estimator for non-homogeneous Dirichlet 
boundary value problem}\label{section_post_nonhomo_dirichlet} 

In this subsection, we consider the following high dimensional boundary value problem 
with non-homogeneous Dirichlet boundary condition:  Find $u$ such that
\begin{equation}\label{ModelProblem_nonhomoDirichlet}
\left\{
\begin{array}{rcl}
-\nabla\cdot(\mathcal A\nabla u)+  b u&=&f, \quad {\rm in} \  \Omega,\\
u&=&g, \quad {\rm on}\  \partial\Omega,
\end{array}
\right.
\end{equation}
where 
the right hand side term $f\in L^2(\Omega)$, and $g\in H^{\frac{1}{2}}(\partial\Omega)$.

The classical technique to solve this problem is enforcing the constraints on the boundary 
in the search space for the solution.
The variational form of (\ref{ModelProblem_nonhomoDirichlet}) 
can be written as: Find $u\in H_g^1(\Omega):=\{w\in H^1(\Omega):w=g\ {\rm on}\ \partial\Omega\}$ 
such that
\begin{eqnarray}\label{var_form_nonhomo_D}
a(u,v)=F(v),\ \ \ \forall v\in H_0^1(\Omega),
\end{eqnarray}
where  $F(v)=(f,v)$. 
Since $a(\cdot,\cdot)$ is continuous and coercive over $H_0^1(\Omega)$, the equation 
(\ref{var_form_nonhomo_D}) is well-posed.
When using some discretization method to solve (\ref{var_form_nonhomo_D}), 
the non-homogeneous boundary conditions must be imposed on the trial functions, i.e., 
that every trial function must coincide with $g$ on the boundary.
In the finite element method, the enforcement of Dirichlet boundary conditions 
in the trial space is also easily translated to the discrete level.
However, it is always difficult to satisfy the nonhomogeneous Dirichlet boundary 
conditions directly when constructing trial functions with neural networks.

Different from the way of adding penalty term in other common machine learning methods, 
we can follow the way of finite element method to treat the non-homogeneous 
Dirichlet boundary condition. 
First, we do the TNN interpolation for the non-homogeneous Dirichlet boundary condition. 
Then, we solve the deduced homogeneous boundary value problem based on the TNN 
approximation of boundary condition in the first step. 
According to the trace theorem \cite[Section 5.5]{Evans}, 
there exists $u_b\in H^1(\Omega)$ such that $u_b|_{\partial\Omega}=g$. 
Let us define $u_0=u-u_b$ with the exact solution $u$ of (\ref{ModelProblem_nonhomoDirichlet}).  
Then $u_0$ is the solution of the following homogeneous boundary value problem
\begin{equation}\label{ModelProblem_homoDirichlet_u0}
\left\{
\begin{array}{rcl}
-\nabla\cdot(\mathcal A\nabla u_0)+  b u_0&=&f+\nabla\cdot(\mathcal A\nabla u_b)-bu_b, 
\quad {\rm in} \  \Omega,\\
u_0&=&0, \quad {\rm on}\  \partial\Omega.
\end{array}
\right.
\end{equation}
The corresponding variational form for (\ref{ModelProblem_homoDirichlet_u0}) is: 
Find $u_0\in H_0^1(\Omega)$ such that
\begin{eqnarray}\label{var_form_nonhomo_D_u0}
a(u_0,v)=(f,v)-a(u_b,v),\ \ \ \forall v\in H_0^1(\Omega).
\end{eqnarray}

We also define $V:= H_0^1(\Omega)$ and the energy norm $\|v\|_a=\sqrt{a(v,v)}$.
Assume $V_p$ is a given $p$-dimensional subspace and satisfy $V_p\subset V$.
The standard Galerkin approximation scheme for the problem (\ref{var_form_nonhomo_D_u0}) is: 
Find $u_{0,p}\in V_p$ such that
\begin{eqnarray*}
a(u_{0,p},v_p)=(f,v_p)-a(u_b,v_p),\ \ \ \forall v_p\in V_p.
\end{eqnarray*}
Once we have the Galerkin approximation $u_{0,p}$, the approximation of solution $u$ 
can be represented by $u_p=u_{0,p}+u_b$, which satisfies the following a posteriori 
error estimation.

\begin{theorem}\label{Theorem_nonhDirichlet}
Assume $u\in H_g^1(\Omega)$ is the exact solution of problem (\ref{var_form_nonhomo_D}),
$u_0$ and $u_{0,p}$ are the exact solution and subspace approximation of 
problem (\ref{var_form_nonhomo_D_u0}).
Then $u_p=u_{0,p}+u_b$ satisfies the following upper bound
\begin{equation}\label{Upper_Bound_nonhomoD}
\|u-u_p\|_a=\|u_0-u_{0,p}\|_a \leq  \eta(u_p,\mathbf y),\ \ \ \forall \mathbf y\in \mathbf W,
\end{equation}
where $\eta(u_p,\mathbf y)$ is defined as follows
\begin{equation}\label{eta_nonhDirichlet}
\eta(u_p,\mathbf y):=\left(\|b^{-\frac{1}{2}}(f-b u_p+{\rm div}(\mathcal A\mathbf y))\|_0^2
+\|\mathcal A^{\frac{1}{2}}(\mathbf y-\nabla u_p)\|_0^2\right)^{\frac{1}{2}}.
\end{equation}
\end{theorem}
\begin{proof}
Let us define $v= u_0-u_{0,p}\in V$ in this proof.
Then by Lemma \ref{lemma_Green} and 
Cauchy-Schwarz inequality, the following estimates hold
\begin{eqnarray*}
&&a(u_0-u_{0,p},v)=(f,v)-a(u_b,v)-a(u_{0,p},v)=(f,v)-a(u_p,v)\\
&=&(f,v)-(\mathcal A\nabla u_p,\nabla v)-(b u_p,v)\nonumber\\
&=&(f,v)-(\mathcal A\nabla u_p,\nabla v)-(b u_p,v)+(\mathcal A\mathbf y,\nabla v)
+({\rm div}(\mathcal A\mathbf y),v)\nonumber\\
&=&(f-b u_p+{\rm div}(\mathcal A\mathbf y),  v)+(\mathcal A(\mathbf y-\nabla u_p),\nabla v)\nonumber\\
&\leq& \|b^{-\frac{1}{2}}(f-b u_p+{\rm div}(\mathcal A\mathbf y))\|_0\|b^{\frac{1}{2}}v\|_0
+\|\mathcal A^{\frac{1}{2}}(\mathbf y-\nabla u_p)\|_0\|\mathcal A^{\frac{1}{2}}\nabla v\|_0\nonumber\\
&\leq& \left(\|b^{-\frac{1}{2}}(f-b u_p+{\rm div}(\mathcal A\mathbf y))\|_0^2
+\|\mathcal A^{\frac{1}{2}}(\mathbf y-\nabla u_p)\|_0^2\right)^{1/2}\|v\|_a,
\ \ \ \forall\mathbf y\in \mathbf W.
\end{eqnarray*}
Due to $u=u_0+u_b$ and $u_p=u_{0,p}+u_b$, the desired result (\ref{Upper_Bound_nonhomoD}) 
can be obtained, and the proof is complete.
\end{proof}

\subsection{A posteriori error estimator for Neumann boundary value problem}
In this subsection, we come to consider the following second-order elliptic problem 
with Neumann boundary condition: 
Find $u\in H^1(\Omega)$ such that
\begin{equation}\label{PDE_Neumann}
\left\{
\begin{array}{rcl}
-\nabla\cdot(\mathcal A\nabla u)+  b u&=&f, \quad {\rm in} \  \Omega,\\
(\mathcal A\nabla u)\cdot\mathbf n&=&g, \quad {\rm on}\  \partial\Omega,
\end{array}
\right.
\end{equation}
where 
the right hand side term $f\in L^2(\Omega)$ and $g\in L^2(\partial\Omega)$, 
$\mathbf n$ is the outward normal to $\partial\Omega$.

The weak form of problem (\ref{PDE_Neumann}) can be described as: Find $u\in V$ such that
\begin{eqnarray}\label{weak_form_Neumann}
a(u,v)=(f,v)+(g,v)_{L^2(\partial \Omega)},\ \ \ \forall v\in V,
\end{eqnarray}
where $V:= H^1(\Omega)$. 
The energy norm $\|\cdot\|_a$ is defined as $\|v\|_a=\sqrt{a(v,v)}$.
Assume we have a $p$-dimentional subspace $V_p$ which satisfies $V_p\subset V$.
The standard Galerkin approximation scheme for problem (\ref{weak_form_Neumann}) is: 
Find $u_p\in V_p$ such that
\begin{eqnarray}\label{weak_form_Neumann_Discrete}
a(u_p,v_p)=(f,v_p)+(g,v_p)_{L^2(\partial\Omega)},\ \ \ \forall v_p\in V_p.
\end{eqnarray}

Before introducing the a posteriori error estimate for the second-order elliptic problem 
with Neumann boundary condition, we introduce the following Steklov eigenvalue problem: 
Find $(\lambda,u)\in\mathbb R\times V$ such that 
\begin{equation}\label{Steklov_Eigenvalue}
\left\{
\begin{array}{rcl}
-\nabla\cdot(\mathcal A\nabla u) + b u &=&0,  \ \ \ \  {\rm in}\ \Omega,\\
(\mathcal A\nabla u)\cdot\mathbf n &=& \lambda u,\ \ {\rm on}\ \partial\Omega. 
\end{array}
\right.
\end{equation}
The smallest eigenvalue $\lambda_{\rm min}$ of the Steklov eigenvalue 
problem (\ref{Steklov_Eigenvalue}) satisfies the following property
\begin{eqnarray}\label{Inequality_5}
\lambda_{\rm min} = \inf_{v\in V}\frac{\|v\|_a^2}{\|v\|_{0,\partial\Omega}^2}.
\end{eqnarray}
Then, we can obtain the following a posteriori error estimation for the Galerkin approximation $u_p$.
\begin{theorem}
Assume $u\in V$ and $u_p\in V_p \subset V$ are the exact solution of (\ref{weak_form_Neumann}) 
and the corresponding subspace approximation defined by 
(\ref{weak_form_Neumann_Discrete}), respectively. The following upper bound holds
\begin{eqnarray}\label{Upper_Bound_U_Neumann}
\|u-u_p\|_a\leq \left(1+\frac{1}{\lambda_{\rm min}}\right)^{1/2}\eta(u_p,\mathbf y),
\ \ \ \forall \mathbf y\in\mathbf H({\rm div},\Omega),
\end{eqnarray}
where $\lambda_{\rm min}$ denotes the smallest eigenvalue of the 
Steklov eigenvalue problem (\ref{Steklov_Eigenvalue}) 
and $\eta(u_p,\mathbf y)$ is defined as follows
\begin{eqnarray}\label{eta_Neumann}
&&\ \ \ \ \ \ \ \ \ \eta(u_p,\mathbf y)\nonumber\\
&&\ \ \ \ \ \ \ \ \ := \left(\big\|b^{-\frac{1}{2}}
\big(f-b u_p+{\rm div}(\mathcal A\mathbf y)\big)\big\|_0^2
+\big\|\mathcal A^{\frac{1}{2}}(\mathbf y-\nabla u_p)\big\|_0^2
+\big\|g-(\mathcal A\mathbf y)\cdot\mathbf n\big\|_{0,\partial\Omega}^2 \right)^{\frac{1}{2}}.
\end{eqnarray}
\end{theorem}
\begin{proof}
Let us define $v=u-u_p$ in this proof.
Then by Lemma \ref{lemma_Green} and Cauchy-Schwarz inequality, 
for any $\mathbf y\in\mathbf W$, the following estimates hold
\begin{eqnarray*}
&&a(u-u_p,v)=(f,v)+(g,v)_{L^2(\partial\Omega)}-(\mathcal A\nabla u_p,\nabla v)-(b u_p,v)\nonumber\\
&=&(f,v)+(g,v)_{L^2(\partial\Omega)}-(\mathcal A\nabla u_p,\nabla v)-(b u_p,v)\nonumber\\
&&+({\rm div}(\mathcal A\mathbf y),v)+(\mathcal A\mathbf y,\nabla v)
-((\mathcal A\mathbf y)\cdot\mathbf n,v)_{L^2(\partial\Omega)}\nonumber\\
&=&(f-b u_p+{\rm div}(\mathcal A\mathbf y),v)+(\mathcal A(\mathbf y-\nabla u_p),\nabla v)
+(g-(\mathcal A\mathbf y)\cdot\mathbf n,v)_{L^2(\partial\Omega)}\nonumber\\
&\leq&\|b^{-\frac{1}{2}}(f-b u_p+{\rm div}(\mathcal A\mathbf y))\|_0\|b^{\frac{1}{2}}v\|_0
+\|\mathcal A^{\frac{1}{2}}(\mathbf y-\nabla u_p)\|_0\|\mathcal A^{\frac{1}{2}}\nabla v\|_0\nonumber\\
&&+\|g-(\mathcal A\mathbf y)\cdot\mathbf n\|_{0,\partial\Omega}\|v\|_{0,\partial\Omega}\nonumber\\
&\leq&\left(\|b^{-\frac{1}{2}}(f-b u_p+{\rm div}(\mathcal A\mathbf y))\|_0^2
+\|\mathcal A^{\frac{1}{2}}(\mathbf y-\nabla u_p)\|_0^2+\|g-(\mathcal A\mathbf y)\cdot\mathbf n\|_{0,\partial\Omega}^2\right)^{\frac{1}{2}}\nonumber\\
&&\times\left(\|b^{\frac{1}{2}}v\|_0^2+\|\mathcal A^{\frac{1}{2}}\nabla v\|_0^2
+\|v\|_{0,\partial\Omega}^2\right)^{\frac{1}{2}}\nonumber\\
&=&\left(\|b^{-\frac{1}{2}}(f-b u_p+{\rm div}(\mathcal A\mathbf y))\|_0^2
+\|\mathcal A^{\frac{1}{2}}(\mathbf y-\nabla u_p)\|_0^2+\|g-(\mathcal A\mathbf y)\cdot\mathbf n\|_{0,\partial\Omega}^2\right)^{\frac{1}{2}}\nonumber\\
&&\times\left(\|v\|_a^2+\|v\|_{0,\partial\Omega}^2\right)^{\frac{1}{2}}\nonumber\\
&\leq&\left(\|b^{-\frac{1}{2}}(f-b u_p+{\rm div}(\mathcal A\mathbf y))\|_0^2
+\|\mathcal A^{\frac{1}{2}}(\mathbf y-\nabla u_p)\|_0^2+\|g-(\mathcal A\mathbf y)\cdot\mathbf n\|_{0,\partial\Omega}^2\right)^{\frac{1}{2}}\nonumber\\
&&\times \left(1+\frac{1}{\lambda_{\rm min}}\right)^{\frac{1}{2}}\|v\|_a,
\end{eqnarray*}
where the last inequality is due to property (\ref{Inequality_5}).
Then the desired result (\ref{Upper_Bound_U_Neumann}) can be obtained and the proof is complete.
\end{proof}

\subsection{A posteriori error estimator for eigenvalue problem}
In this subsection, we are concerned with the following high-dimensional second-order 
elliptic eigenvalue problem: 
Find $(\lambda,u)\in \mathbb R\times H_0^1(\Omega)$ such that 
\begin{eqnarray}\label{Eigen_Prob}
\left\{
\begin{array}{rcl}
-\nabla\cdot(\mathcal A\nabla u)+  b u&=&\lambda u, \quad {\rm in} \  \Omega,\\
u&=&0, \quad \ \  {\rm on}\  \partial\Omega.
\end{array}
\right.
\end{eqnarray}

The variational form for problem (\ref{Eigen_Prob}) can be described as: 
Find $(\lambda,u)\in\mathbb R\times V$ such that $\|u\|_0=1$ and
\begin{eqnarray}\label{weak_form_eigen}
a(u,v)=\lambda(u,v),\ \ \ \forall v\in V,
\end{eqnarray}
where $V:= H_0^1(\Omega)$. 

The corresponding energy norm $\|\cdot\|_a$ can be defined as $\|v\|_a=\sqrt{a(v,v)}$.
It is well known that the eigenvalue problem (\ref{weak_form_eigen})
has an eigenvalue sequence $\{\lambda_j \}$
(cf. \cite{BabuskaOsborn_Book}):
\begin{eqnarray}\label{Eigenvalues}
0<\lambda_1\leq \lambda_2\leq\cdots\leq\lambda_k\leq\cdots,\ \ \
\lim_{k\rightarrow\infty}\lambda_k=\infty,
\end{eqnarray} 
and associated eigenfunctions
\begin{eqnarray}\label{Eigenfunctions}
u_1, u_2, \cdots, u_k, \cdots,
\end{eqnarray}
where $(u_i,u_j)=\delta_{ij}$ ($\delta_{ij}$ denotes the Kronecker function).
In the sequence $\{\lambda_j\}$, the $\lambda_j$ are repeated according to their
geometric multiplicity.

Assume we have a $p$-dimensional subspace $V_p$ which satisfies $V_p\subset V$.
The standard subspace approximation scheme for the problem (\ref{weak_form_eigen}) is: 
Find $(\lambda_p,u_p)\in\mathbb R\times V_p$ such that $\|u_p\|_0=1$ and
\begin{eqnarray}\label{weak_form_eigen_approx}
a(u_p,v_p)=\lambda_p(u_p,v_p),\ \ \ \forall v_p\in V_p.
\end{eqnarray}
In order to deduce the error estimate for the eigenfunction approximations, 
we define the spectral projection $E_1: V\rightarrow {\rm span}\{u_1\}$ 
as follows
\begin{eqnarray*}
a(E_1w, v) = a(w, v),\ \ \ \ \forall v\in {\rm span}\{u_1\}. 
\end{eqnarray*}
Then, a guaranteed upper bound exists for the energy error of eigenfunction 
approximation $u_p$ and exact solution $u$.
\begin{theorem}\label{theorem_eigen_post}
Let $(\lambda_1,u_1)$ denote the smallest exact eigenpair of 
the eigenvalue problem (\ref{weak_form_eigen}). Assume the discrete eigenpair 
$(\lambda_p,u_p)$ obtained by solving (\ref{weak_form_eigen_approx}) 
is the approximation to the first exact eigenpair $(\lambda_1,u_1)$ 
and the eigenvalue separation $\lambda_2>\lambda_1$ and $\lambda_2>\lambda_p$ holds. 
Then the following guaranteed upper bound holds
\begin{eqnarray}\label{Estimate_Eigenfunction}
\|E_1u_p-u_p\|_a\leq \frac{\lambda_2}{\lambda_2-\lambda_p}\eta(\lambda_p,u_p,\mathbf y),
\ \ \ \forall\mathbf y\in\mathbf W,
\end{eqnarray}
where $\eta(\lambda_p,u_p,\mathbf y)$ is defined as follows
\begin{eqnarray}\label{eta_Eigen}
\ \ \ \eta(\lambda_p,u_p,\mathbf y):=\left(\|b^{-\frac{1}{2}}(\lambda_pu_p-bu_p
+{\rm div}(\mathcal A\mathbf y))\|_0^2 
+\|\mathcal A^{\frac{1}{2}}(\mathbf y-\nabla u_p)\|_0^2\right)^{\frac{1}{2}}.
\end{eqnarray}
\end{theorem} 

\begin{proof}
In order to simplify the notation, let us define $v=E_1u_p-u_p$. 
It is easy to know that $v \perp E_1u_p$ in the sense of $a(\cdot,\cdot)$ 
and $(\cdot,\cdot)$,  and the following inequality holds 
\begin{eqnarray}\label{Inequality_2}
\frac{\|u_p-E_1u_p\|_a^2}{\|u_p-E_1u_p\|_0^2} \geq \lambda_2.
\end{eqnarray}
 With the help of (\ref{Inequality_2}), we have the following estimates
\begin{eqnarray}\label{Inequality_3}
&&a(E_1u_p - u_p,v)=-a(u_p,v)=-(\mathcal A\nabla u_p,\nabla v)-(bu_p,v)\nonumber\\
&&=\lambda_p(u_p,v)-(\mathcal A\nabla u_p,\nabla v)-(bu_p,v)+(\mathcal A\mathbf y,\nabla v)
+({\rm div}(\mathcal A\mathbf y),v)-\lambda_p(u_p,v)\nonumber\\
&&=(\lambda_pu_p-bu_p+{\rm div}(\mathcal A\mathbf y),v)+(\mathcal A(\mathbf y-\nabla u_p),\nabla v)
+\lambda_p (E_1u_p-u_p,v)\nonumber\\
&&\leq\|b^{-\frac{1}{2}}(\lambda_pu_p-bu_p+{\rm div}(\mathcal A\mathbf y))\|_0\|b^{\frac{1}{2}}v\|_0
+\|\mathcal A^{\frac{1}{2}}(\mathbf y-\nabla u_p)\|_0\|\mathcal A^{\frac{1}{2}}\nabla v\|_0\nonumber\\
&&\ \ \ \ \  +\lambda_p\|E_1u_p-u_p\|_0^2\nonumber\\
&&\leq \left(\|b^{-\frac{1}{2}}(\lambda_pu_p-bu_p+{\rm div}(\mathcal A\mathbf y))\|_0^2
+\|\mathcal A^{\frac{1}{2}}(\mathbf y-\nabla u_p)\|_0^2\right)^{\frac{1}{2}}\|v\|_a\nonumber\\
&&\ \ \ \ \ +\frac{\lambda_p}{\lambda_2}\|E_1u_p-u_p\|_a^2.
\end{eqnarray}
From (\ref{Inequality_3}) and  $\lambda_2>\lambda_p$, the following inequality holds
\begin{eqnarray}\label{Inequality_4}
&&\left(1-\frac{\lambda_p}{\lambda_2}\right)\|E_1u_p-u_p\|_a^2\nonumber\\
&&\leq  
\left(\|b^{-\frac{1}{2}}(\lambda_pu_p-bu_p+{\rm div}(\mathcal A\mathbf y))\|_0^2
+\|\mathcal A^{\frac{1}{2}}(\mathbf y-\nabla u_p)\|_0^2\right)^{\frac{1}{2}}\|v\|_a. 
\end{eqnarray}
Then the desired result (\ref{Estimate_Eigenfunction}) can be deduced from (\ref{Inequality_4}), 
and the proof is complete. 
\end{proof}

\begin{remark}
Theorem \ref{theorem_eigen_post}
is only concerned with the error estimates for the eigenfunction 
approximation since the error estimates for the eigenvalue approximation can be easily deduced
from the following error expansion
\begin{eqnarray*}
0\leq \lambda_p-\lambda&=&\frac{a(u_p-u_1,u_p-u_1)}{(u_p,u_p)}
-\lambda\frac{(u_p-u_1,u_p-u_1)}{(u_p,u_p)}\\
&\leq&\frac{\|u_p-u_1\|_a^2}{\|u_p\|_0^2}\leq\|u_p-E_1u_p\|_a^2
\leq \Big(\frac{\lambda_2}{\lambda_2-\lambda_p}\Big)^2 \eta^2(\lambda_p,u_p,\mathbf y).
\end{eqnarray*}
\end{remark}

\subsection{Construction of loss function}
The aim of introducing the a posteriori error estimates in previous subsections is to 
define the loss functions of the machine learning methods for solving the concerned 
second-order elliptic boundary value problems and eigenvalue problems in this paper.  

These a posteriori error estimates provide 
an upper bound of the approximation errors in the energy norm. 
As we know, the numerical methods aim to produce the approximations which have 
the optimal errors in some sense.
This inspires us to use these a posteriori error estimates as the 
loss functions of the TNN-based machine learning method
for solving the concerned problems in this paper.

Take the homogeneous Dirichlet type of boundary value problem as an example. 
Denote the Galerkin approximation on space $V_p$ for problem (\ref{weak_form_hDirichlet}) 
as $u_p=u_p(V_p)$.
Now, consider subspaces $V_p\subset V$ with a fixed dimension $p$.
From Theorem \ref{Theorem_hDirichlet}, as long as $V_p\subset V$, the following inequality holds
\begin{eqnarray}\label{ineq_error_eta}
\inf_{\substack{V_p\subset V\\\dim(V_p)=p}}\|u-u_p(V_p)\|_a 
\leq \inf_{\substack{V_p\subset V\\\dim(V_p)=p}}\inf_{\mathbf y\in\mathbf W}\eta(u_p(V_p),\mathbf y).
\end{eqnarray}
The error of the best $p$-dimensional Galerkin approximation can be controlled by 
solving the optimization problem on the right-hand side of the inequality (\ref{ineq_error_eta}). 
Based on the discussion in Section \ref{section_choose_Vp}, we can define the following loss function 
for the TNN-based machine learning method 
at the $\ell$-th step
\begin{eqnarray*}
\mathcal L=\inf_{\mathbf y\in\mathbf W}\eta(u_p(V_p^{(\ell)}),\mathbf y).
\end{eqnarray*}
In real implementation, in order to make full use of the parameters of TNN and 
improve the computational efficiency, we choose  $\mathbf y=\nabla u_p(V_p^{(\ell)}) \in\mathbf W$. 
Then, we define the loss function as follows
\begin{eqnarray*}
\mathcal L=\eta\left(u_p\left(V_p^{(\ell)}\right),\nabla u_p\left(V_p^{(\ell)}\right)\right)
=\left\|b^{-\frac{1}{2}}\left(f-b u_p(V_p^{(\ell)})+ \nabla\cdot(\mathcal A\nabla u_p(V_p^{(\ell)}))\right)\right\|_0.
\end{eqnarray*}
In the next section, we will show how to calculate the loss function $\mathcal L$ 
and propose a complete TNN-based machine learning algorithm for solving the second-order elliptical PDEs.

\section{TNN-based method for second order elliptic problems}\label{TNN_Machine_Learning}
In this section, we propose the TNN-based machine learning method for solving 
the second-order elliptic problems included in the previous section. 
The TNN-based machine learning method can be regarded as an adaptive mesh-free method. 
This means that the approximating space $V_p$ is chosen adaptively in the learning process.
The corresponding procedure can be broken down into the following steps:
\begin{enumerate}
\item Define the $p$-dimensional space $V_p$ as (\ref{def_Vpl}) by the combination of subnetworks.
\item Obtain Galerkin approximation on the space $V_p$, which is expressed 
in the form of (\ref{def_TNN_normed}).
\item Compute a posteriori error estimator and the descent direction.
\item Update the parameters of TNN
to update the subspace $V_p$.
\end{enumerate}
The way to define $V_p$ at step 1 has already been introduced in detail 
in Sections \ref{Section_TNN} and \ref{section_choose_Vp}. 
In Section \ref{Section_Posteriori}, we construct a posteriori error estimators
that can be used  at step 3 for different problems.  
In the following few subsections, we will  introduce how to use TNN to 
express the Galerkin approximations for different types of PDEs. 

\subsection{Finding Galerkin approximation by TNN}\label{section_find_galerkin}
It is well known that the Galerkin approximation is the optimal solution according 
to the energy norm. Taking the second-order elliptic Dirichlet boundary value 
problem as an example, the following C\'{e}a's lemma 
holds \cite{brenner2007mathematical,ciarlet2002finite}.

\begin{lemma}[C\'{e}a's lemma]
If the bilinear form $a(\cdot,\cdot)$ is defined by (\ref{Bilinear_Form}) in 
Section \ref{section_post_dirichlet} and subspace approximation $u_p$ defined by 
(\ref{weak_form_hDirichlet_Discrete}) for the  problem (\ref{weak_form_hDirichlet}), 
the following estimates hold
\begin{eqnarray*}
\|u-u_p\|_a =  \inf_{v_p\in V_p}\|u-v_p\|_a.
\end{eqnarray*}
\end{lemma}

According to the definition of TNN (\ref{def_TNN_normed}) and the definition 
of the corresponding space $V_p^{(\ell)}$ (\ref{def_Vpl}), 
we can see that the neural network parameters determine the space $V_p^{(\ell)}$, 
while the coefficient parameters $c^{(\ell+1)}$ determine the direction of TNN 
in the space $V_p^{(\ell)}$. When the neural network parameter $\theta^{(\ell)}$ 
is fixed at the  $\ell$-th step, the Galerkin approximation $u_p^{(\ell)}$ can be 
constructed by 
\begin{eqnarray*}
u_p^{(\ell)} = \Psi(x;c^{(\ell+1)},\theta^{(\ell)}),
\end{eqnarray*}
where the linear coefficient parameters $c^{(\ell+1)}$ in  (\ref{def_TNN_normed}) 
is obtained by solving the following system of linear equations
\begin{eqnarray}\label{Galerkin_Problem}
A^{(\ell)}c^{(\ell+1)}=B^{(\ell)},
\end{eqnarray}
where $A^{(\ell)}\in \mathbb R^{p\times p}$ and $B^{(\ell)}\in\mathbb R^{p\times 1}$,
\begin{eqnarray*}
A_{m,n}^{(\ell)}=a(\varphi_n^{(\ell)},\varphi_m^{(\ell)}),\ \ \ 
B_m^{(\ell)}=(f,\varphi_m^{(\ell)}),\ \ \ 1\leq m, n\leq p. 
\end{eqnarray*}
This inspires us to divide the optimization step into two sub-steps.
Firstly, when the neural network parameters $\theta^{(\ell)}$ are fixed, 
the optimal coefficient parameters $c^{(\ell+1)}$ are obtained by solving 
the Galerkin approximation equation (\ref{Galerkin_Problem}). 
Secondly, when the coefficient parameters $c^{(\ell+1)}$ are fixed, 
the neural network parameters $\theta^{(\ell+1)}$
can be updated by the Gradient descent step 
(\ref{Optimization_Step}) with the loss function being defined 
by the a posteriori error estimators. 
In the following subsections, we will give the complete TNN-based algorithm for solving 
the boundary value problems with homogeneous and non-homogeneous 
Dirichlet type and Neumann type boundary conditions, 
and eigenvalue problem of second-order elliptic operator.

\subsection{Algorithm for homogenous Dirichlet boundary value problem}\label{section_alg_hDirichlet}
In order to avoid the penalty of the Dirichlet boundary condition, the 
trial function defined by TNN is required to belong to 
the space $H_0^1(\Omega)$. For this aim, we use the method in \cite{GuWangYang}, 
which is firstly proposed in \cite{LagarisLikasFotiadis,LagarisLikasPapageorgiou}, 
to treat the Dirichlet boundary condition. 
Based on the TNN defined by (\ref{def_TNN_normed}), we multiply each 
$\widehat\phi_{i,j}(x_i;\theta_i)$ by the function $h(x_i):=(x_i-a_i)(b_i-x_i)$ 
to obtain the homogeneous trial function subspace $V_p^{(\ell)}$ as follows
\begin{eqnarray*}
V_{0,p}^{(\ell)}:={\rm span}\left\{\varphi_{j}(x;\theta^{(\ell)})
:=\prod_{i=1}^{d}\left(\widehat\phi_{i,j}\left(x_{i};\theta_{i}^{(\ell)}\right)
(x_i-a_i)(b_i-x_i)\right), 
j=1, \cdots, p\right\}.
\end{eqnarray*}
Obviously, we have $V_{0,p}^{(\ell)}\subset H_0^1(\Omega)$.
It should be pointed out that there are many choices of the limiting function $h(x_i)$ 
as long as they satisfy the condition of being zero on $\partial\Omega_i$ and nonzero in $\Omega_i$.
The details of the TNN-based method for homogenous Dirichlet boundary value problem are defined 
in Algorithm \ref{Algorithm_1}.

\begin{algorithm}[htb!]
\caption{TNN-based method for homogeneous Dirichlet boundary value problem}\label{Algorithm_1}
\begin{enumerate}
\item Initialization step: Build initial TNN $\Psi(x;c^{(0)},\theta^{(0)})$ 
as in (\ref{def_TNN_normed}), set the maximum training steps $M$, 
learning rate $\gamma$ and $\ell=0$.

\item Define the $p$-dimensional space $V_p^{(\ell)}$ as follows
\begin{eqnarray*}
V_{0,p}^{(\ell)}:={\rm span}\left\{\varphi_{j}(x;\theta^{(\ell)}):=\prod_{i=1}^{d}\widehat\phi_{i,j}
\left(x_{i};\theta_{i}^{(\ell)}\right)(x_i-a_i)(b_i-x_i), j=1, \cdots, p\right\}.
\end{eqnarray*}
Assemble the stiffness matrix $A^{(\ell)}$ and the right-hand side term $B^{(\ell)}$ on $V_{p}^{(\ell)}$.
The entries are defined as follows
\begin{eqnarray*}
&&A_{m,n}^{(\ell)}=a(\varphi_n^{(\ell)},\varphi_m^{(\ell)})=(\mathcal A\nabla\varphi_n^{(\ell)},\nabla\varphi_m^{(\ell)})+(b\varphi_n^{(\ell)},\varphi_m^{(\ell)}),
\ 1\leq m,n\leq p, \\
&&B_m^{(\ell)}=(f,\varphi_m^{(\ell)}),\ 1\leq m\leq p.
\end{eqnarray*}

\item Solve the following linear 
 equation to obtain the solution $c\in\mathbb R^{p\times 1}$
\begin{eqnarray*}
A^{(\ell)}c=B^{(\ell)}.
\end{eqnarray*}
Update the coefficient parameter as $c^{(\ell+1)}=c$.
Then the Galerkin approximation on the space $V_{p}^{(\ell)}$ for problem 
(\ref{weak_form_hDirichlet}) is $\Psi(x;c^{(\ell+1)},\theta^{(\ell)})$.

\item Compute the a posteriori error estimator (\ref{eta_hDirichlet}) (loss function)
\begin{eqnarray*}
\mathcal L^{(\ell+1)}(c^{(\ell+1)},\theta^{(\ell)}):=
\eta(\Psi(x;c^{(\ell+1)},\theta^{(\ell)}),\nabla\Psi(x;c^{(\ell+1)},\theta^{(\ell)})).
\end{eqnarray*}

\item Update the neural network parameter $\theta^{(\ell)}$ of TNN as follows 
\begin{eqnarray*}
\theta^{(\ell+1)}=\theta^{(\ell)}-\gamma\frac{\partial\mathcal L^{\ell+1)}}{\partial\theta}(c^{(\ell+1)},\theta^{(\ell)}).
\end{eqnarray*}

\item Set $\ell=\ell+1$ and go to Step 2 for the next step until $\ell=M$.

\end{enumerate}
\end{algorithm}

\subsection{Algorithm for non-homogenous Dirichlet boundary value problem}\label{section_alg_nonhDirichlet}
According to the theory of a posterior error estimation for non-homogeneous 
Dirichlet boundary value problems introduced in Section \ref{section_post_nonhomo_dirichlet}, 
we treat the solution on $\partial\Omega$ and in $\Omega$ separately. 
Similarly to the finite element method, to solve the problem 
(\ref{ModelProblem_nonhomoDirichlet}),  we define two different TNNs, $\Psi_1$ and $\Psi_2$, 
where $\Psi_1$ is used to fit the Dirichlet boundary condition $g$ and $\Psi_2$ 
is used to solve homogeneous Dirichlet boundary value problem (\ref{ModelProblem_homoDirichlet_u0}).

To be specific, we first train $\Psi_1$ with the following loss function
\begin{eqnarray}\label{eq_loss_bd}
\mathcal L_{\rm bd}=\|\Psi_1(x;c_1,\theta_1)-g(x)\|_{0,\partial\Omega},
\end{eqnarray}
and obtain the parameters $c_1^*$ and $\theta_1^*$ after enough epochs.
From the definition of the computational domain 
$\Omega=\Omega_1\times\cdots\times\Omega_d=(a_1,b_1)\times\cdots\times(a_d,b_d)$, 
the boundary $\partial\Omega$ can be partitioned as follows
\begin{eqnarray*}
\partial\Omega=\bigcup_{i=1}^d\left(\Gamma_{i,a}\cup\Gamma_{i,b}\right),
\end{eqnarray*}
where
\begin{eqnarray}\label{eq_Gamma_ab}
\ \ \ \ \Gamma_{i,a}&:=&(a_1,b_1)\times\cdots\times(a_{i-1},b_{i-1})\times\{x_i=a_i\}
\times(a_{i+1},b_{i+1})\times\cdots\times(a_d,b_d),\\
\ \ \ \ \Gamma_{i,b}&:=&(a_1,b_1)\times\cdots\times(a_{i-1},b_{i-1})
\times\{x_i=b_i\}\times(a_{i+1},b_{i+1})\times\cdots\times(a_d,b_d).
\end{eqnarray}
The boundary integrations on $\Gamma_{i,a}$ and $\Gamma_{i,b}$, $i=1,\cdots,d$, 
can be computed separately.
Since $\Gamma_{i,a}$, $\Gamma_{i,b}$, $i=1,\cdots,d$ are
also tensor-type domains, we can use the quadrature scheme introduced 
in Section \ref{section_quad} to calculate the loss function 
(\ref{eq_loss_bd}) with high accuracy and high efficiency, if $g$ has a low-rank structure.

As described in Section \ref{section_alg_hDirichlet}, it is easy to construct the homogeneous
trial function by TNN due to its variable separation property.
Once we obtain $\Psi_1(x;c_1^*,\theta_1^*)$, a similar procedure as Algorithm \ref{Algorithm_1} 
can be used to solve homogeneous Dirichlet boundary problem (\ref{ModelProblem_homoDirichlet_u0}) 
using TNN $\Psi_2$. 
It only needs to change assembling the right-hand side term in step 2 
of Algorithm \ref{Algorithm_1} to be the following form
\begin{eqnarray*}
B_m^{(\ell)}&=&(f,\varphi_m^{(\ell)})-a(\Psi_1(x;c_1^*,\theta_1^*),\varphi_m^{(\ell)}),
\end{eqnarray*}
and the loss function in step 4 of Algorithm \ref{Algorithm_1} is modified 
to be the a posteriori error estimator (\ref{eta_nonhDirichlet}).

\subsection{Algorithm for Neumann boundary value problem}\label{section_alg_Neuamnn}
Similarly, the TNN-based method for the Neumann boundary value problem can be defined 
in Algorithm \ref{Algorithm_2}.
It should be noted that when assembling the right-hand side term in Step 4 of the 
algorithm and computing the a posteriori error estimator in Step 4 of Algorithm 
\ref{Algorithm_2}, we need to compute the integrations 
$(g,\varphi_m^{(\ell)})_{\partial\Omega}$ and $\|g-(\mathcal 
A\nabla\Psi(x;c^{(\ell+1)},\theta^{(\ell)}))\cdot\mathbf n\|_{\partial\Omega}$ 
on the boundary $\partial\Omega$. 
In the same way as the quadrature scheme on $\partial\Omega$ of (\ref{eq_loss_bd}), 
we deal with the boundary integrations involved in Algorithm \ref{Algorithm_2} 
on $\Gamma_{i,0}$ and $\Gamma_{i,1}$, $i=1,\cdots,d$, separately, 
and use the efficient quadrature scheme introduced in Section \ref{section_quad}. 

Unlike the Dirichlet boundary condition, in Algorithm \ref{Algorithm_2}, 
we can treat homogeneous and non-homogeneous Neumann boundary conditions in a uniform way.
We perform the numerical test for a non-homogeneous case in Section \ref{section_test_nonhomoNeumann}.
\begin{algorithm}[htb!]
\caption{TNN-based method for Neumann boundary value problem}\label{Algorithm_2}
\begin{enumerate}
\item Initialization step: Build initial TNN $\Psi(x;c^{(0)},\theta^{(0)})$ as in 
(\ref{def_TNN_normed}), set maximum training steps $M$, learning rate $\gamma$ and $\ell=0$.

\item Define $p$-dimensional space $V_p^{(\ell)}$ as follows
\begin{eqnarray*}
V_{p}^{(\ell)}:=\operatorname{span}\left\{\varphi_{j}(x ; \theta^{(\ell)})
:=\prod_{i=1}^{d}\widehat\phi_{i, j}\left(x_{i} ; \theta_{i}^{(\ell)}\right), j=1, \cdots, p\right\}.
\end{eqnarray*}
Assemble the stiffness matrix $A^{(\ell)}$ and the right-hand side term $B^{(\ell)}$ on $V_{p}^{(\ell)}$.
The entries are defined as follows
\begin{eqnarray*}
A_{m,n}^{(\ell)}&=&a(\varphi_n^{(\ell)},\varphi_m^{(\ell)})=(\mathcal A\nabla\varphi_n^{(\ell)},\nabla\varphi_m^{(\ell)})+(b\varphi_n^{(\ell)},\varphi_m^{(\ell)}),
\ 1\leq m,n\leq p,\\
B_m^{(\ell)}&=&(f,\varphi_m^{(\ell)})-(g,\varphi_m^{(\ell)})_{\partial\Omega},\ 1\leq m\leq p.
\end{eqnarray*}

\item Solve the following linear equation and obtain the solution $c\in\mathbb R^{p\times 1}$
\begin{eqnarray*}
A^{(\ell)}c=B^{(\ell)}.
\end{eqnarray*}
Update the coefficient parameter as $c^{(\ell+1)}=c$.
Then the Galerkin approximation on the space $V_{p}^{(\ell)}$ for 
problem (\ref{weak_form_Neumann}) is $\Psi(x;c^{(\ell+1)},\theta^{(\ell)})$.

\item Compute the a posteriori error estimator (\ref{eta_Neumann}) (loss function)
\begin{eqnarray*}
\mathcal L^{(\ell+1)}(c^{(\ell+1)},\theta^{(\ell)})
:=\eta(\Psi(x;c^{(\ell+1)},\theta^{(\ell)}),\nabla\Psi(x;c^{(\ell+1)},\theta^{(\ell)})).
\end{eqnarray*}

\item Update the neural network parameter $\theta^{(\ell)}$ of TNN as follows 
\begin{eqnarray*}
\theta^{(\ell+1)}=\theta^{(\ell)}-\gamma\frac{\partial\mathcal L^{\ell+1)}}{\partial\theta}(c^{(\ell+1)},\theta^{(\ell)}).
\end{eqnarray*}

\item Set $\ell=\ell+1$ and go to Step 2 for the next step until $\ell=M$.
\end{enumerate}
\end{algorithm}

\subsection{Algorithm for eigenvalue problem}\label{section_alg_eigen}
For the second-order elliptic eigenvalue problem, we are mainly concerned with
obtaining the eigenpair corresponding to the smallest eigenvalue. 
Similarly, the detailed algorithm for the eigenvalue problem can be defined  
in Algorithm \ref{Algorithm_3}, where the homogeneous boundary conditions are treated in the similar way 
as that in Section \ref{section_alg_hDirichlet}.

Since the eigenvalue problem (\ref{Eigen_Prob}) has eigenvalue series (\ref{Eigenvalues}) and 
corresponding eigenfunction  series (\ref{Eigenfunctions}),  in order to get the 
smallest eigenvalue, 
we need to consider how to avoid local minimum cases during the learning process. 
Then for the a posteriori error estimator (\ref{eta_Eigen}) of the eigenvalue problem, 
we have the following equalities
\begin{eqnarray*}
0=\eta(\lambda_1,u_1,\nabla u_1)=\eta(\lambda_2,u_2,\nabla u_2)
=\cdots=\eta(\lambda_k,u_k,\nabla u_k)=\cdots.
\end{eqnarray*}
This means, if the smallest eigenpair in the space $V_p^{(\ell)}$ 
is some $(\lambda_k,u_k)$ with $k>1$,  
Step 5 of Algorithm \ref{Algorithm_3} will not actually update parameter $\theta^{(\ell+1)}$.
As a result, the space $V_p^{(\ell+1)}$ will
stay the same and parameters $c^{(\ell+2)}$ 
will not be updated at the $(\ell+1)$-th learning step. 
This also means that the algorithm converges to 
$(\lambda_k,u_k=\Psi(x;c^{(\ell+1)},\theta^{(\ell)}))$ 
which is not the approximation to the smallest eigenpair.
To avoid this phenomenon, we recommend using the TNN-based method in \cite{WangJinXie} 
as a pre-training method to get an appropriate initial TNN.
Since the eigenvalue problem satisfies $\lambda_1<\lambda_2$ \cite{AndrewsClutterbuck}, 
an appropriate initial TNN can ensure that the optimization process eventually obtain 
an approximation of $(\lambda_1,u_1)$.

\begin{algorithm}[htb!]
\caption{TNN-based method for eigenvalue problem}\label{Algorithm_3}
\begin{enumerate}
\item Initialization step: Build initial TNN $\Psi(x;c^{(0)},\theta^{(0)})$ 
as in (\ref{def_TNN_normed}), set maximum training steps $M$, 
learning rate $\gamma$ and  $\ell=0$.

\item Define $p$-dimensional space $V_p^{(\ell)}$ as follows
\begin{eqnarray*}
V_{0,p}^{(\ell)}:={\rm span}\left\{\varphi_{j}(x;\theta^{(\ell)})
:=\prod_{i=1}^{d}\widehat\phi_{i,j}\left(x_{i};\theta_{i}^{(\ell)}\right)(x_i-a_i)(b_i-x_i), 
j=1, \cdots, p\right\}.
\end{eqnarray*}
Assemble the stiffness matrix $A^{(\ell)}$ and mass matrix $M^{(\ell)}$ on $V_{0,p}^{(\ell)}$ as follows
\begin{eqnarray*}
A_{m,n}^{(\ell)}&=&a(\varphi_n^{(\ell)},\varphi_m^{(\ell)})=(\mathcal A\nabla\varphi_n^{(\ell)},\nabla\varphi_m^{(\ell)})+(b\varphi_n^{(\ell)},\varphi_m^{(\ell)}),
\ \ 1\leq m,n\leq p,\\
M_{m,n}^{(\ell)}&=&(\varphi_n^{(\ell)},\varphi_m^{(\ell)}),\ \ 1\leq m,n\leq p.
\end{eqnarray*}

\item Solve the following eigenvlaue problem to obtain the first 
eigenpair approximation $(\lambda,c)\in\mathbb R\times\mathbb R^{p}$
\begin{eqnarray*}
A^{(\ell)}c=\lambda M^{(\ell)}c.
\end{eqnarray*}
Update the coefficient parameter as $c^{(\ell+1)}=c$.
Then the Galerkin approximation on the space $V_{p}^{(\ell)}$ for eigenvalue problem 
(\ref{weak_form_eigen}) is $(\lambda,\Psi(x;c^{(\ell+1)},\theta^{(\ell)}))$.

\item Compute the a posteriori error estimator (\ref{eta_Eigen}) (loss function)
\begin{eqnarray*}
\mathcal L^{(\ell+1)}(c^{(\ell+1)},\theta^{(\ell)}):=
\eta(\lambda,\Psi(x;c^{(\ell+1)},\theta^{(\ell)}),\nabla\Psi(x;c^{(\ell+1)},\theta^{(\ell)})).
\end{eqnarray*}

\item Update the neural network parameter of TNN as follows 
\begin{eqnarray*}
\theta^{(\ell+1)}=\theta^{(\ell)}-\gamma\frac{\partial\mathcal L^{\ell+1)}}{\partial\theta}(c^{(\ell+1)},\theta^{(\ell)}).
\end{eqnarray*}

\item Set $\ell=\ell+1$ and go to Step 2 for the next step until $\ell=M$.
\end{enumerate}
\end{algorithm}

\section{Numerical examples}\label{Section_Numerical} 
In this section, we provide several examples to validate the efficiency and accuracy of the TNN-based
machine learning methods, Algorithms \ref{Algorithm_1}, \ref{Algorithm_2} and \ref{Algorithm_3}, 
proposed in this paper. All the experiments are done on an NVIDIA Tesla A800 GPU and an Tesla NVIDIA V100 GPU. 

In order to show the convergence behavior and accuracy of the Dirichlet boundary 
value problem and Neumann boundary value problem, we define 
the following errors for the numerical solution $\Psi(x;c^*,\theta^*)$
\begin{eqnarray*}
\widehat e_{L^2}:=\frac{\|u-\Psi(x;c^*,\theta^*)\|_{L^2(\Omega)}}{\|f\|_{L^2(\Omega)}},
\ \ \ \widehat e_{H^1}:=\frac{\left|u-\Psi(x;c^*,\theta^*)\right|_{H^1(\Omega)}}{\left\|f\right\|_{L^2(\Omega)}}.
\end{eqnarray*}
Here $\|\cdot\|_{L^2}$ and $|\cdot|_{H^1}$ denote the $L^2(\Omega)$ norm and the $H^1(\Omega)$ seminorm, respectively.

For the approximate eigenpair approximation $(\lambda^*,\Psi(x;c^*,\theta^*))$ by TNN-based method 
for the eigenvalue problem, we define the $L^2(\Omega)$ projection operator $\mathcal P:H_0^1(\Omega)
\rightarrow {\rm span}\{\Psi(x;c^*,\theta^*)\}$ as follows:
\begin{eqnarray*}
\left\langle\mathcal Pu,v\right\rangle_{L^2}=\left\langle u,v\right\rangle_{L^2}:=\int_\Omega uvdx,
\ \ \ \forall v\in {\rm span}\{\Psi(x;c^*,\theta^*)\}\ \ {\rm for}\ u\in H_0^1(\Omega).
\end{eqnarray*}
And we define the $H^1(\Omega)$ projection operator $\mathcal Q:H_0^1(\Omega)\rightarrow {\rm span}\{\Psi(x;c^*,\theta^*)\}$
as follows:
\begin{eqnarray*}
\left\langle\mathcal Qu,v\right\rangle_{H^1}=\left\langle u,v\right\rangle_{H^1}
:=\int_\Omega\nabla u\cdot\nabla vdx,
\ \ \ \forall v\in {\rm span}\{\Psi(x;c^*,\theta^*)\}\ \ {\rm for}\ u\in H_0^1(\Omega).
\end{eqnarray*}
Then we define the following errors for the approximated eigenvalue $\lambda^*$ and eigenfunction $\Psi(x;\theta^*)$
\begin{eqnarray*}\label{relative_errors}
e_\lambda:=\frac{|\lambda^*-\lambda|}{|\lambda|},\ \ \ e_{L^2}:=\frac{\|u-\mathcal Pu\|_{L^2(\Omega)}}{\|u\|_{L^2(\Omega)}},
\ \ \ e_{H^1}:=\frac{\left|u-\mathcal Qu\right|_{H^1(\Omega)}}{\left|u\right|_{H^1(\Omega)}}.
\end{eqnarray*}
These relative errors are used to test the accuracy of the numerical examples in this section. 

In the implementation of the proposed TNN-based machine learning method,  
the neural networks are trained by Adam optimizer \cite{KingmaAdam} in combination with L-BFGS and the 
automatic differentiation in PyTorch is used to compute the derivatives.

\subsection{High-dimensional elliptic PDEs with homogeneous Dirichlet boundary condition}
We consider the following boundary value problem with the homogeneous Dirichlet boundary condition: 
Find $u$ such that
\begin{eqnarray*}\label{ex_Dirichlet}
\left\{
\begin{aligned}
-\Delta u&=(d+3)\pi^2\sum_{k=1}^d\sin(2\pi x_k)\cdot\prod_{i\neq k}^d\sin(\pi x_i),\ \ \ &x\in&(-1,1)^d,\\
u&=0,\ \ \ &x\in&\partial[-1,1]^d.
\end{aligned}
\right.
\end{eqnarray*}
It is easy to know the exact solution is 
$u(x)=\sum_{k=1}^d\sin(2\pi x_k)\cdot\prod_{i\neq k}^d\sin(\pi x_i)$.

Notice that in this example, we set $b(x) = 0$. We employ Algorithm \ref{Algorithm_1} as well, 
but in Step 4, we use the following loss function instead:
\begin{eqnarray*}
\mathcal L^{(\ell+1)}(c^{(\ell+1)},\theta^{(\ell)}):= \left\|f + \Delta\Psi(x;c^{(\ell+1)},\theta^{(\ell)})\right\|_0.
\end{eqnarray*}
The corresponding a posteriori error estimate can be deduced 
in the same way as Theorem \ref{Theorem_hDirichlet} 
using Lemma \ref{lemma_Green}, Cauchy-Schwarz inequality 
and Poincar\'e inequality. We omit the proof here.

For the cases $d=5,10,20$, we choose the subnetworks of the TNN 
for each dimension to be FNNs. 
Each FNN has three hidden layers and each hidden layer has $100$ neurons.
Then, the activation function is chosen as the sine function in each hidden layer.
The rank parameter for the TNN is set to $p=50$ in each case.
We use the method described in Section \ref{section_alg_hDirichlet} 
to guarantee the homogeneous boundary condition.

In the computation of the integrations in the loss function, 
we divide each $\Omega_i=(-1,1)$ into $200$ subintervals, 
and choose 16 Gauss points in each subinterval.
The Adam optimizer is employed with a learning rate $0.003$ in the first $50000$ epochs
and then the L-BFGS with a learning rate $1$ in the subsequent $10000$ 
steps to produce the final result.
The corresponding numerical results are shown in Table \ref{table_dirichlet}.
And Figure \ref{fig_dirichlet} shows the change of relative errors $e_{L^2}$ and $e_{H^1}$ 
in the training process. 
\begin{table}[htb!]
\caption{Errors of homogeneous Dirichlet boundary value problem for $d=5,10,20$.}\label{table_dirichlet}
\begin{center}
\begin{tabular}{ccccc}
\hline
$d$&   $e_{L^2}$&   $e_{H^1}$\\
\hline
5&      1.089e-07&   1.706e-06\\
10&     1.040e-07&   1.701e-06\\
20&     8.646e-08&   1.812e-06\\
\hline
\end{tabular}
\end{center}
\end{table}

%

\begin{figure}[htb]
\centering
\includegraphics[width=4.25cm,height=4cm]{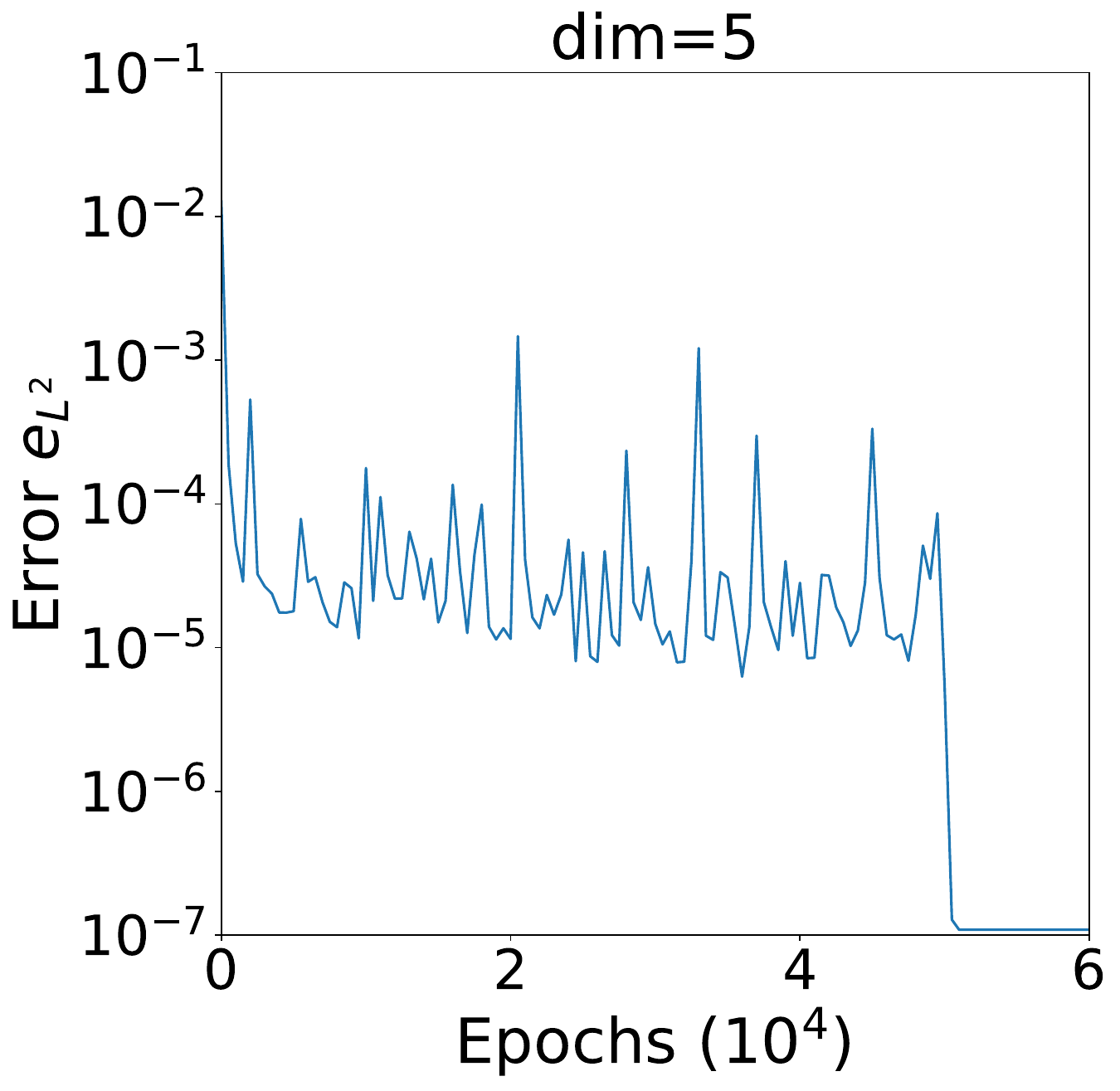}
\includegraphics[width=4.25cm,height=4cm]{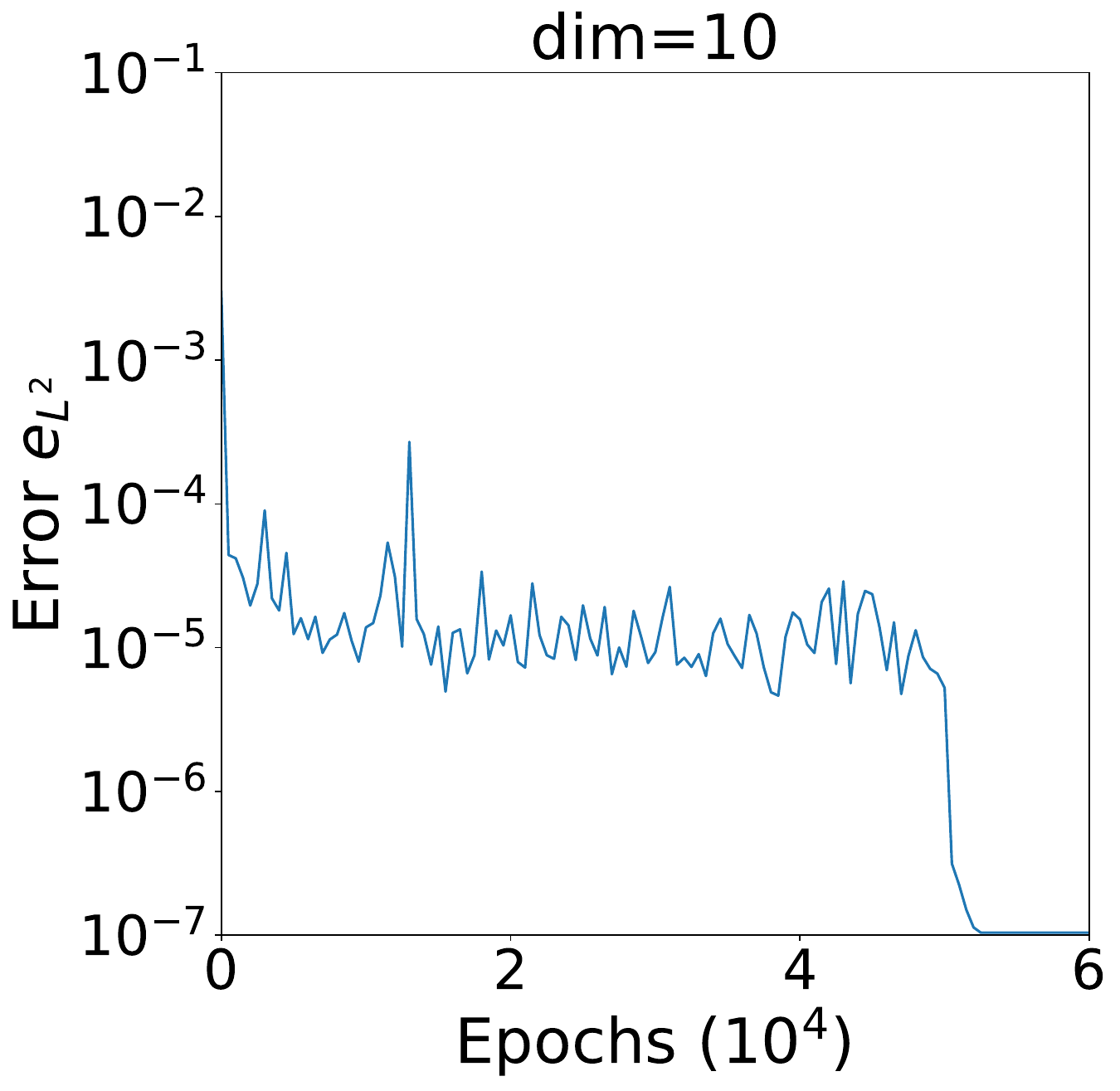}
\includegraphics[width=4.25cm,height=4cm]{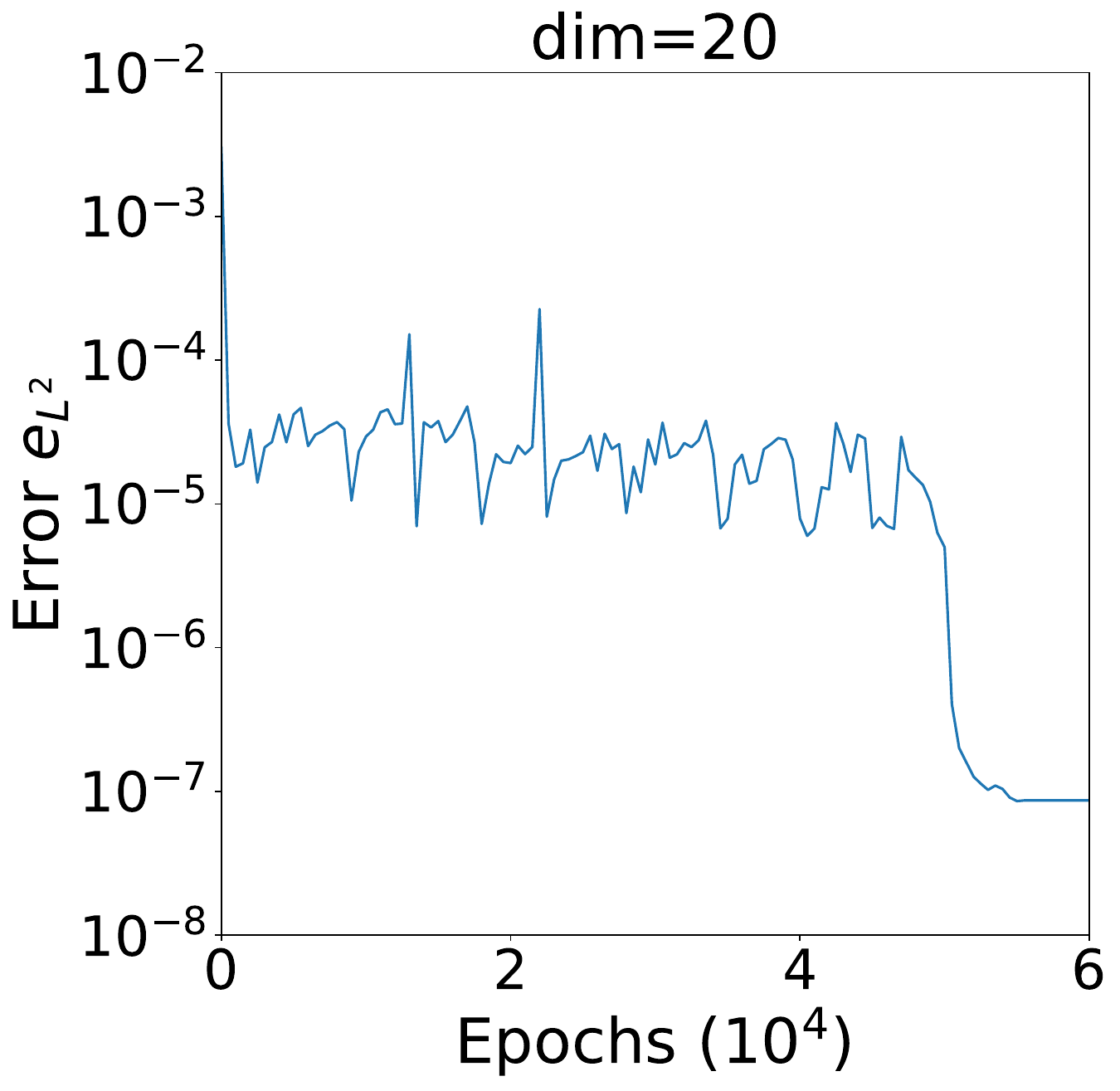}\\
\includegraphics[width=4.25cm,height=4cm]{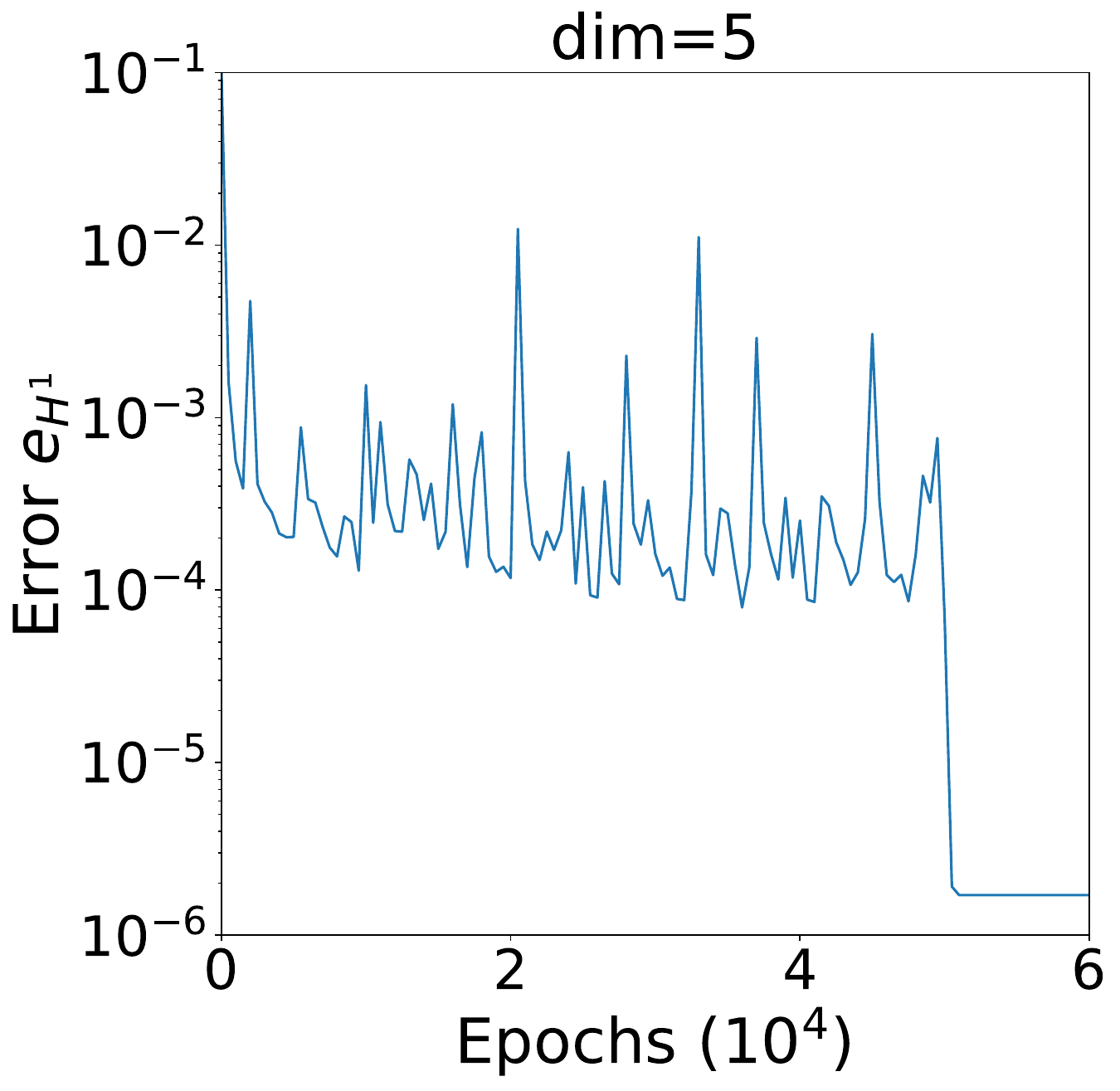}
\includegraphics[width=4.25cm,height=4cm]{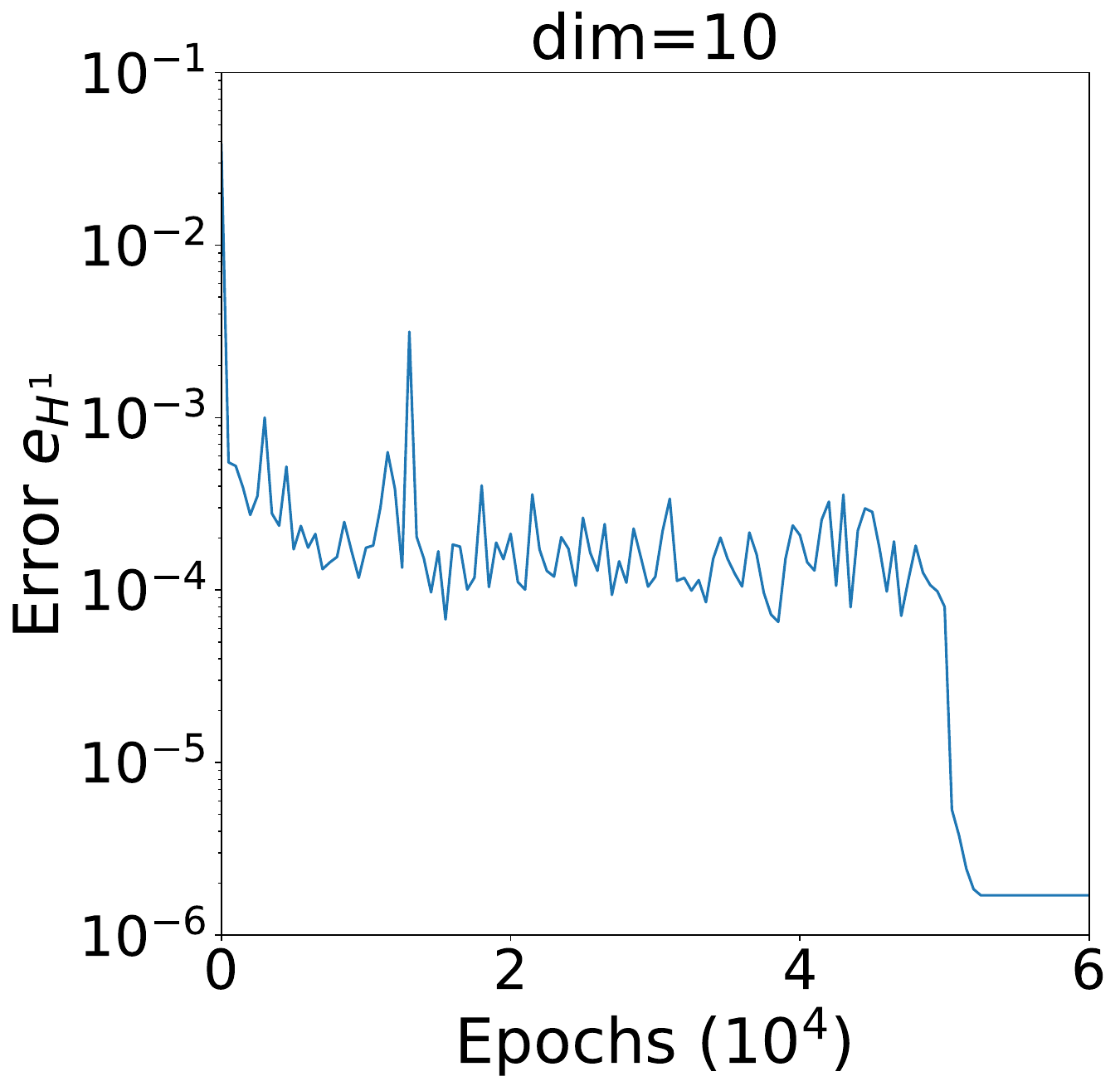}
\includegraphics[width=4.25cm,height=4cm]{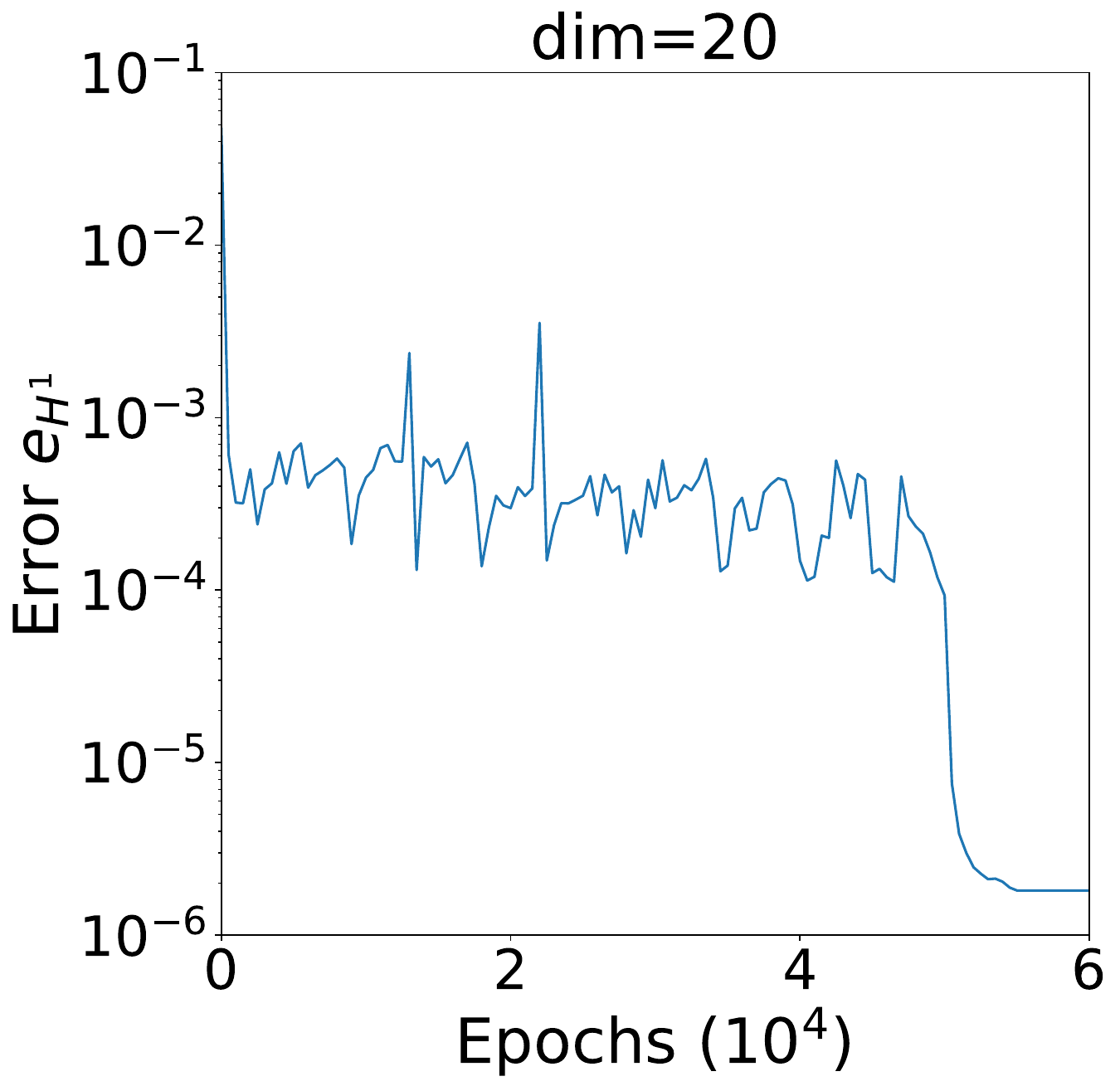}\\
\caption{Relative errors during the training process for the homogeneous 
Dirichlet boundary problem: $d=5$, $10$, and $20$. The upper row shows 
the relative $L^2(\Omega)$ errors, and the down row shows the 
relative $H^1(\Omega)$ errors of solution approximations.}\label{fig_dirichlet}
\end{figure}

\subsection{High-dimensional elliptic PDEs with non-homogeneous Dirichlet boundary condition}
We consider the following boundary value problem with the non-homogeneous 
Dirichlet boundary condition: Find $u$ such that
\begin{eqnarray}\label{ex_nonhomo_Dirichlet}
\left\{
\begin{aligned}
-\Delta u&=\frac{\pi^2}{4}\sum_{i=1}^d\sin\Big(\frac{\pi}{2}x_i\Big),\ \ \ &x\in&(0,1)^d,\\
u&=\sum_{i=1}^d\sin\Big(\frac{\pi}{2}x_i\Big),\ \ \ &x\in&\partial[0,1]^d.
\end{aligned}
\right.
\end{eqnarray}
It is easy to know the problem (\ref{ex_nonhomo_Dirichlet}) has the exact solution
$u(x)=\sum_{i=1}^d\sin\Big(\frac{\pi}{2}x_i\Big)$ in $\Omega=(0,1)^d$.

We set up two TNNs $\Psi_1(x;c_1,\theta_1)$ and $\Psi_2(x;c_2,\theta_2)$, 
where $\Psi_1$ is used to fit the non-homogeneous Dirichlet boundary conditions, 
and $\Psi_2$ is used to solve the deduced homogeneous Dirichlet boundary problem.
For the cases $d=5,10,20$, we choose the subnetworks of $\Psi_1$ and $\Psi_2$ 
for each dimension to be FNNs. Each FNN has three hidden layers and each hidden layer has $50$ neurons.
Then, the activation function is chosen as the sine function in each hidden layer.
The rank parameter for the TNN is set to be $p=20$ in each case.

In computing the integrations of the loss function, 
we divide each $\Omega_i=(0,1)$ into $10$ subintervals and choose 16 Gauss points in each subinterval.
During the training process of TNNs $\Psi_1$ and TNN $\Psi_2$,
the Adam optimizer is employed with a learning rate $0.003$ for the first $20000$ epochs
and then the L-BFGS with a learning rate $0.1$ for the subsequent $5000$ steps 
to produce the final result. 
The corresponding numerical results are shown in Table \ref{table_nonhomo_Dirichlet}, 
where 
\begin{eqnarray*}
e_{\rm bd}=\frac{\|\Psi_1-g\|_{0,\partial\Omega}}{\|g\|_{0,\partial\Omega}},\ \ \ e_{L^2}=\frac{\|\Psi_2+\Psi_1-u\|_{0,\Omega}}{\|f\|_{0,\Omega}},\ \ \ e_{H^1}=\frac{|\Psi_2+\Psi_1-u|_{1,\Omega}}{\|f\|_{0,\Omega}}.
\end{eqnarray*}

\begin{table}[htb!]
\caption{Errors of non-homogeneous Dirichlet boundary value problem for $d=5,10,20$.}\label{table_nonhomo_Dirichlet}
\begin{center}
\begin{tabular}{ccccc}
\hline
$d$&   $e_{\rm bd}$&   $e_{L^2}$&   $e_{H^1}$\\
\hline
5&     1.414e-06&    3.154e-07&   5.862e-06\\
10&    1.201e-06&    3.532e-07&   7.286e-06\\
20&    1.824e-06&    6.340e-07&   1.351e-05\\
\hline
\end{tabular}
\end{center}
\end{table}

\subsection{High-dimensional elliptic PDEs with non-homogeneous Neumann boundary condition}\label{section_test_nonhomoNeumann}
Here, we consider the following boundary value problem with the non-homogeneous Neumann boundary 
value problem: Find $u$ such that
\begin{eqnarray*}
\left\{
\begin{aligned}
-\Delta u+\pi^2u&=2\pi^2\sum_{i=1}^d\sin(\pi x_i),\ \ \ &x\in&(0,1)^d,\\
\frac{\partial u}{\partial \mathbf n}&=-\pi\cos(\pi x_i),\ \ \ &x\in&\Gamma_{i,a},\ \ \ i=1,\cdots,d,\\
\frac{\partial u}{\partial \mathbf n}&=\pi\cos(\pi x_i),\ \ \ &x\in&\Gamma_{i,b},\ \ \ i=1,\cdots,d,
\end{aligned}
\right.
\end{eqnarray*}
where $\Gamma_{i,a}$ and $\Gamma_{i,b}$ are defined as (\ref{eq_Gamma_ab}) with $a_i=0$ and $b_i=1$. 
The exact solution is $u(x)=\sum_{i=1}^d\sin(\pi x_i)$.  
We test high-dimensional cases  $d=5,10,20$, respectively.
Each subnetwork of TNN is a FNN with $3$ hidden layers and each hidden 
layer has $100$ neurons. The activation function is chosen as $\sin(x)$ and the rank $p$ 
is set to be $100$ for each case.
To calculate the integration in the loss function, 
we decompose each $\Omega_i$ ($i=1, \cdots, d$) into $100$ equal 
subintervals and choose $16$ Gauss points on each subinterval.
The a posteriori error estimator on the subspace $V_p$ is chosen 
as the loss function during the training process.
The Adam optimizer is employed with a learning rate of $0.003$ for the first $50000$ training steps.  
Then, the L-BFGS method with a learning rate of $1$ is used for the following $10000$ steps. 
Table \ref{table_nonhomoNeumann} lists the corresponding final errors for different dimensional cases. 
Figure \ref{fig_nonhomoNeumann} shows the change of relative 
errors $e_{L^2}$ and $e_{H^1}$ during the training process. 

\begin{table}[htb!]
\caption{Errors of non-homogeneous Neumann boundary value problem for $d=5,10,20$.}
\label{table_nonhomoNeumann}
\begin{center}
\begin{tabular}{ccccc}
\hline
$d$&   $e_{L^2}$&   $e_{H^1}$\\
\hline
5 &      1.903e-07&   2.395e-06\\
10&     3.554e-07&   3.864e-06\\
20&     2.413e-07&   2.731e-06\\
\hline
\end{tabular}
\end{center}
\end{table}
\begin{figure}[htb]
\centering
\includegraphics[width=4.25cm,height=4cm]{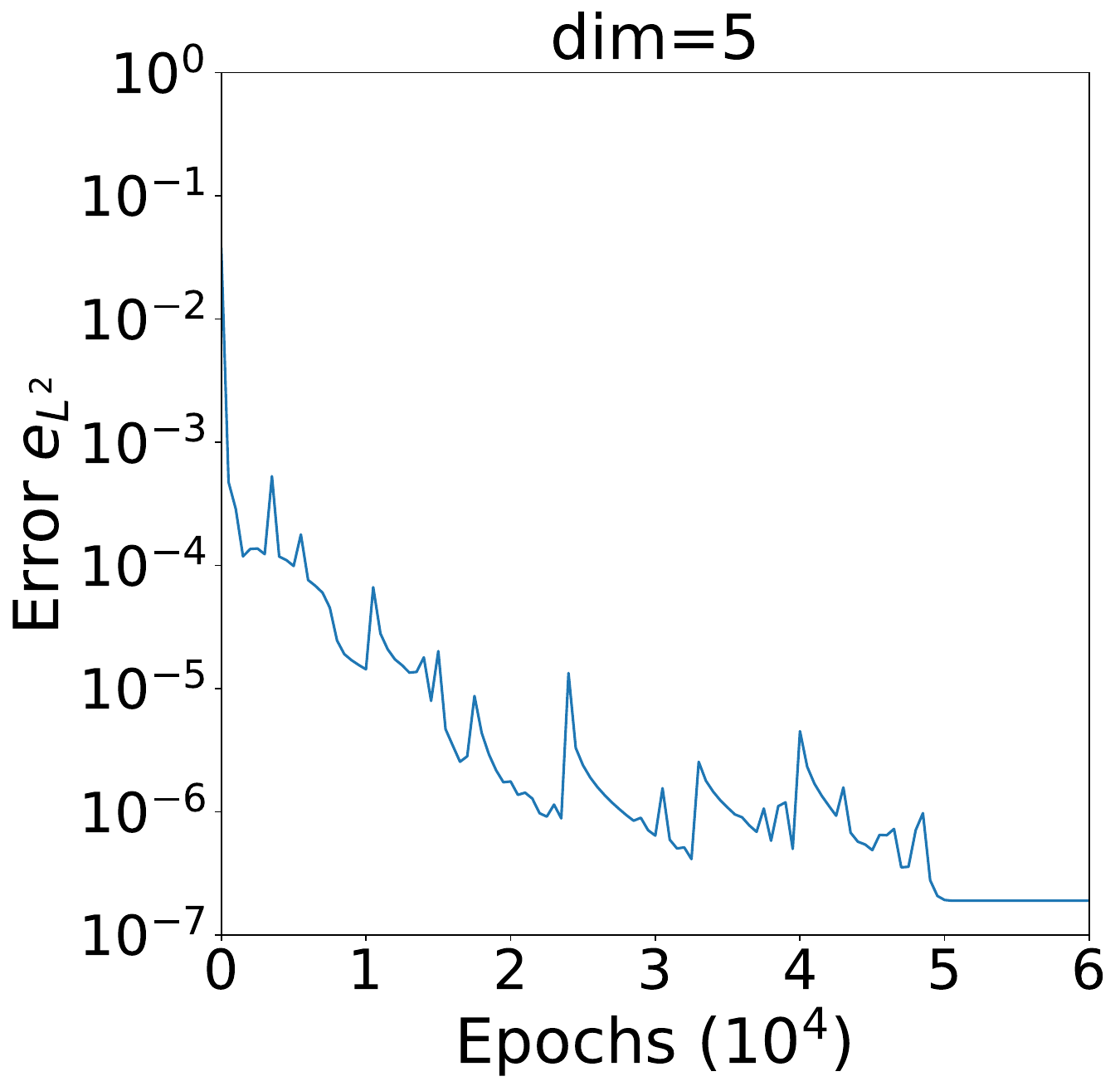}
\includegraphics[width=4.25cm,height=4cm]{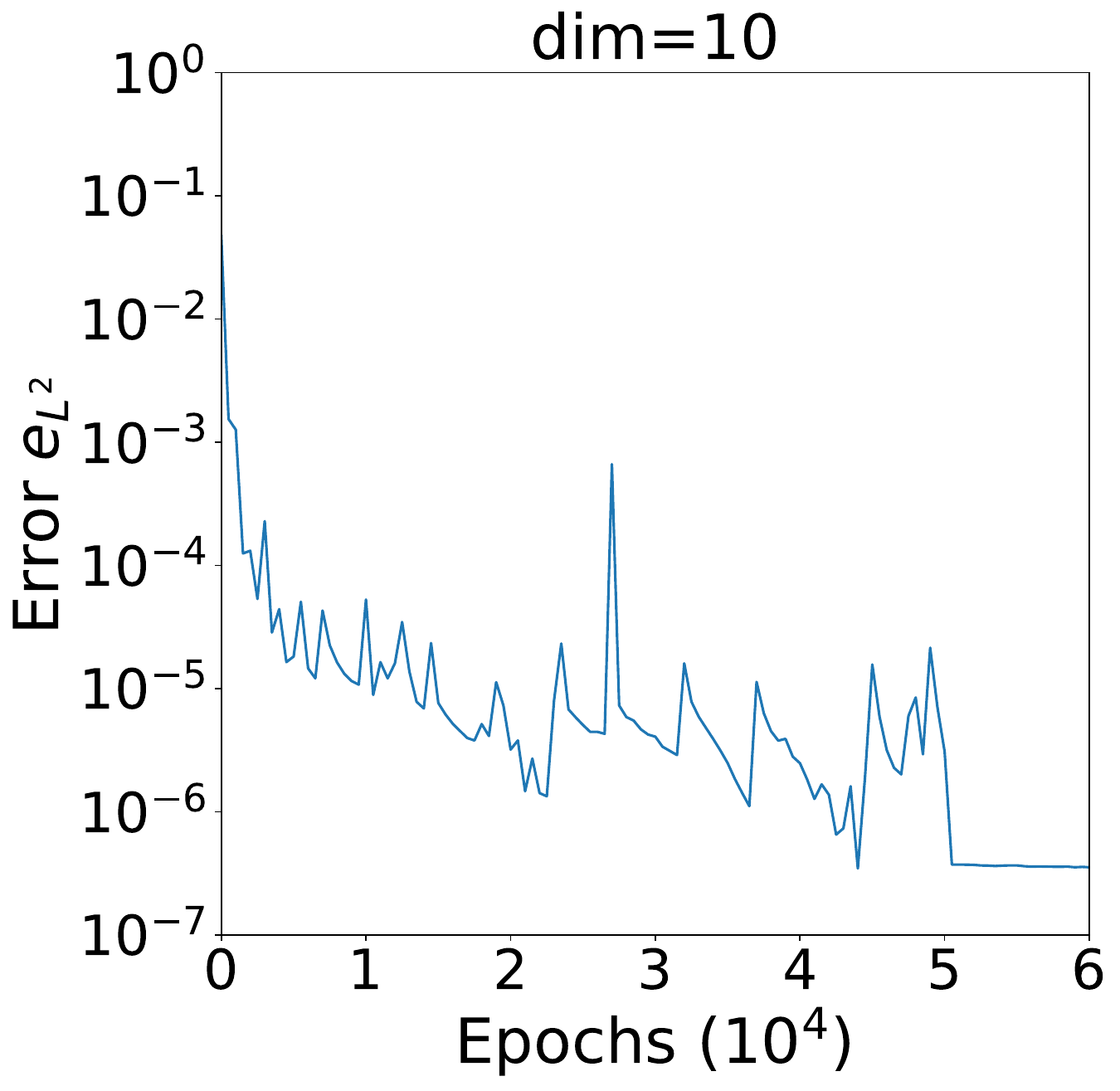}
\includegraphics[width=4.25cm,height=4cm]{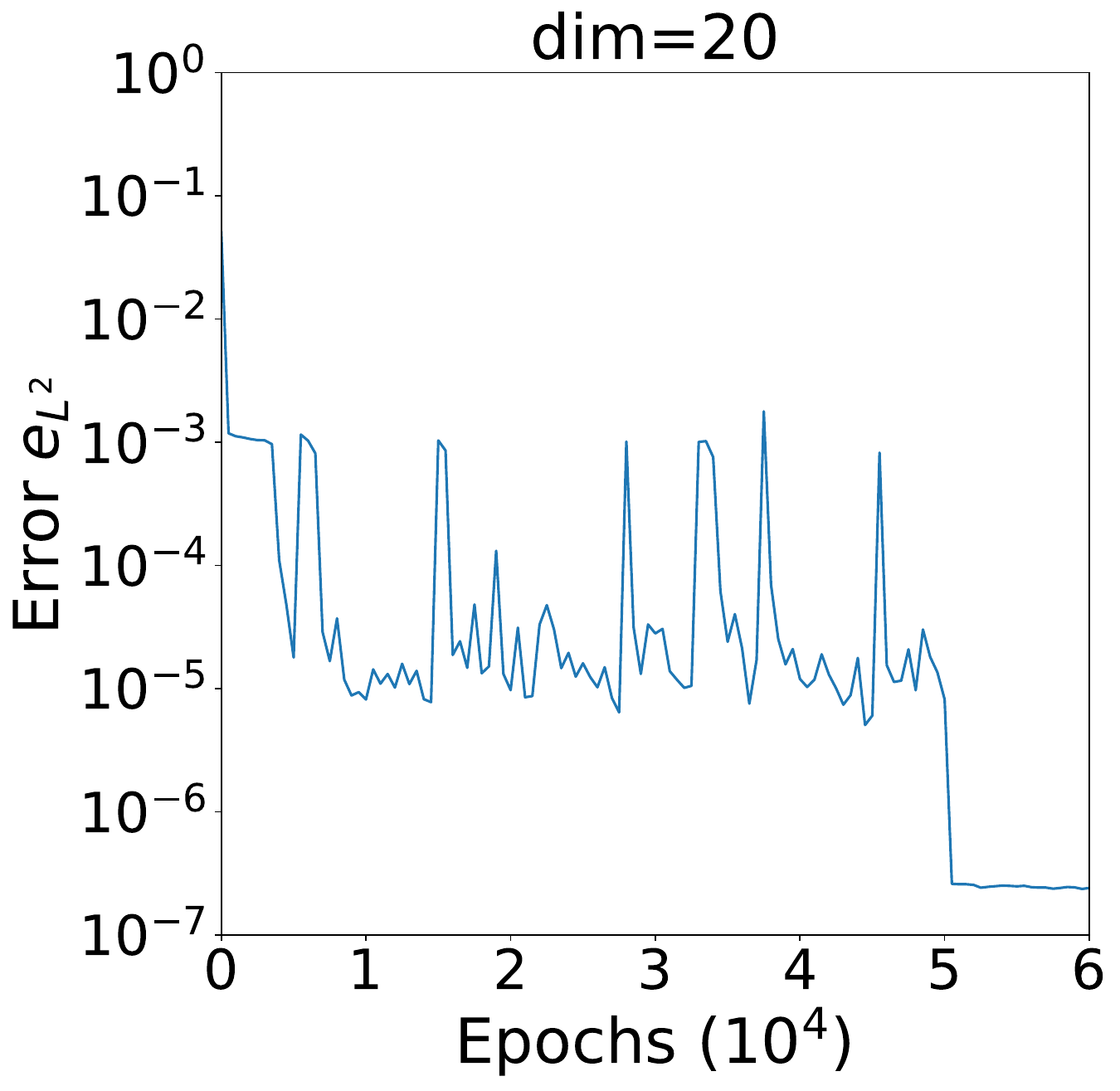}\\
\includegraphics[width=4.25cm,height=4cm]{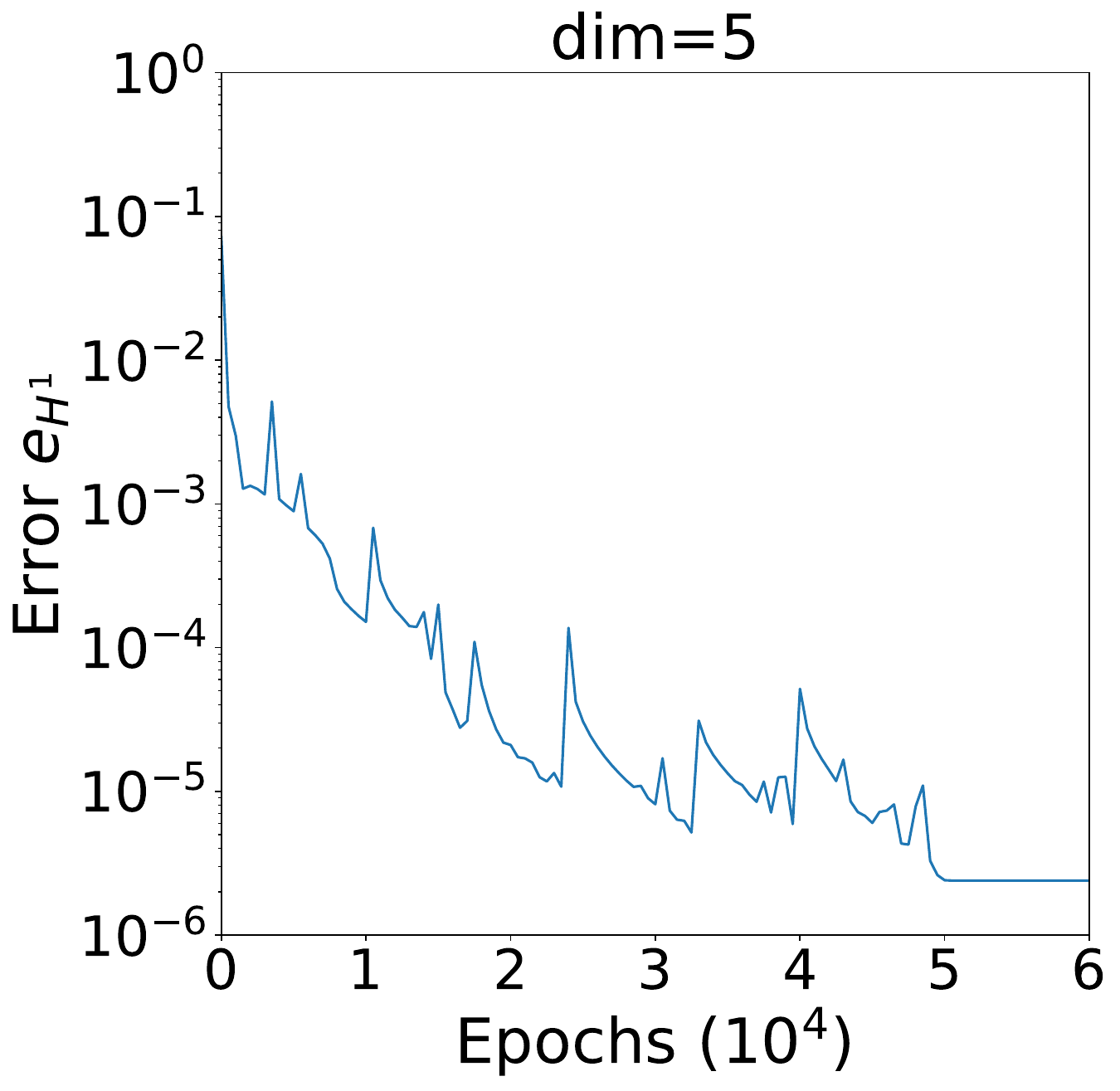}
\includegraphics[width=4.25cm,height=4cm]{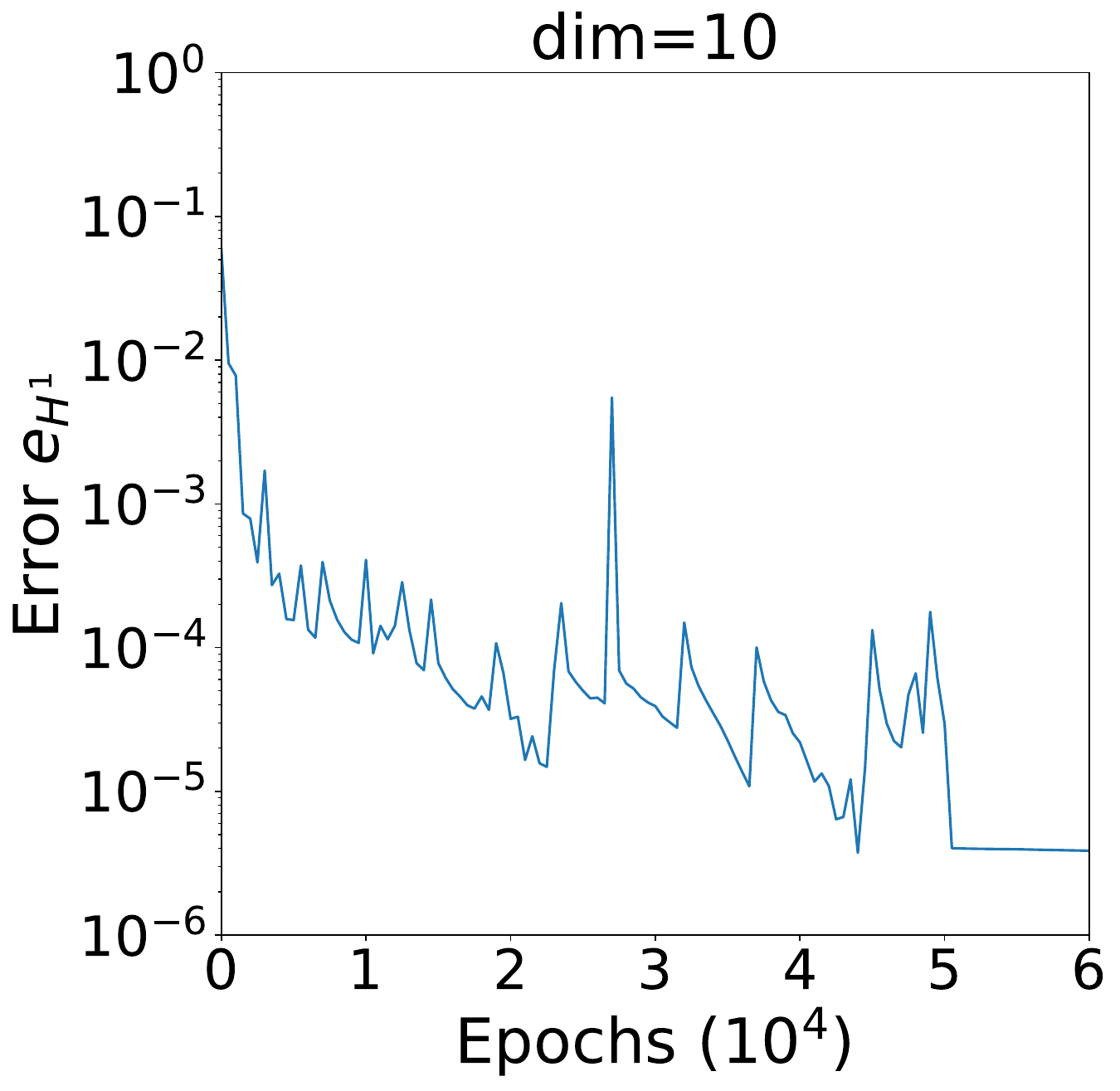}
\includegraphics[width=4.25cm,height=4cm]{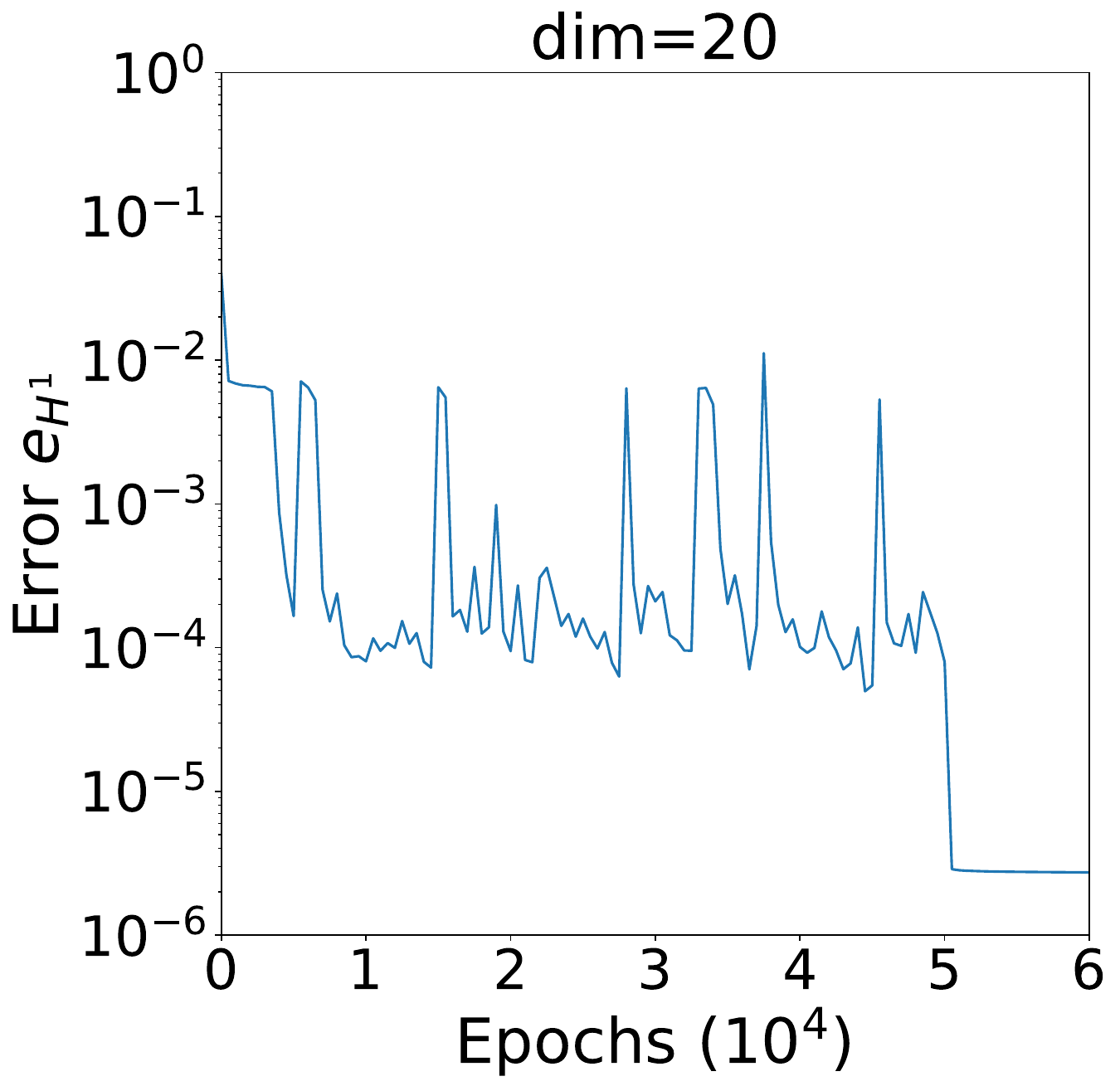}\\
\caption{Relative errors during the training process for non-homogeneous Neumann 
boundary problems: $d=5$, $10$, and $20$. The upper row shows 
the relative $L^2(\Omega)$ errors and the down row shows the 
relative $H^1(\Omega)$ errors of solution approximations.}\label{fig_nonhomoNeumann}
\end{figure}

\subsection{Eigenvalue problem of Laplace operator}
Now, we come to consider solving the following eigenvalue problem: 
Find $(\lambda,u)\in \mathbb R\times H_0^1(\Omega)$ such that
\begin{eqnarray*}
\left\{
\begin{aligned}
-\Delta u&=\lambda u,\ \ \ &x\in&(0,1)^d,\\
u&=0,\ \ \ &x\in&\partial[0,1]^d.
\end{aligned}
\right.
\end{eqnarray*}
The exact smallest eigenpair is
$\lambda=d\pi^2$ and $u(x)=\prod_{i=1}^d\sin(\pi x_i)$.
Here, we test high-dimensional cases with $d=5,10,20$, respectively. 
In this example, each subnetwork of TNN is an FNN with three hidden layers and 
each hidden layer has $100$ neurons. Here, $\sin(x)$ is selected as the 
activation function and the rank $p$ 
is set to be $50$. Quadrature scheme for the integrations in the loss function 
is obtained by decomposing each $\Omega_i$ ($i=1, \cdots, d$) into $100$ equal 
subintervals and choosing $16$ Gauss points on each subinterval. 
During the training process,  a posterior error estimator on the subspace $V_p$ 
is chosen as the loss function. 
The Adam optimizer is employed with a learning rate of $0.003$ for the first $50000$ 
training steps.  
Then, the L-BFGS method with a learning rate of $1$ is used for the following $10000$ steps. 
Table \ref{table_laplace} lists the corresponding final errors for different dimensional cases.  
We can find that the proposed method in this paper has better accuracy than that in \cite{WangJinXie}. 
\begin{table}[htb!]
\caption{Errors of Laplace eigenvalue problem for $d=5,10,20$.}\label{table_laplace}
\begin{center}
\begin{tabular}{ccccc}
\hline
$d$&   $e_{\lambda}$&   $e_{L^2}$&   $e_{H^1}$\\
\hline
5&   1.339e-14&   2.980e-08&   1.215e-07\\
10&   1.656e-14&   9.541e-08&   1.588e-07\\
20&   3.744e-14&   1.626e-07&   2.509e-07\\
\hline
\end{tabular}
\end{center}
\end{table}
Figure \ref{fig_laplace_high} shows the relative errors $e_\lambda$, $e_{L^2}$ and $e_{H^1}$ 
versus the number of epochs. 
We can find that the TNN method has almost the same convergence behaviors for different dimensions.
\begin{figure}[htb]
\centering
\includegraphics[width=4.25cm,height=4cm]{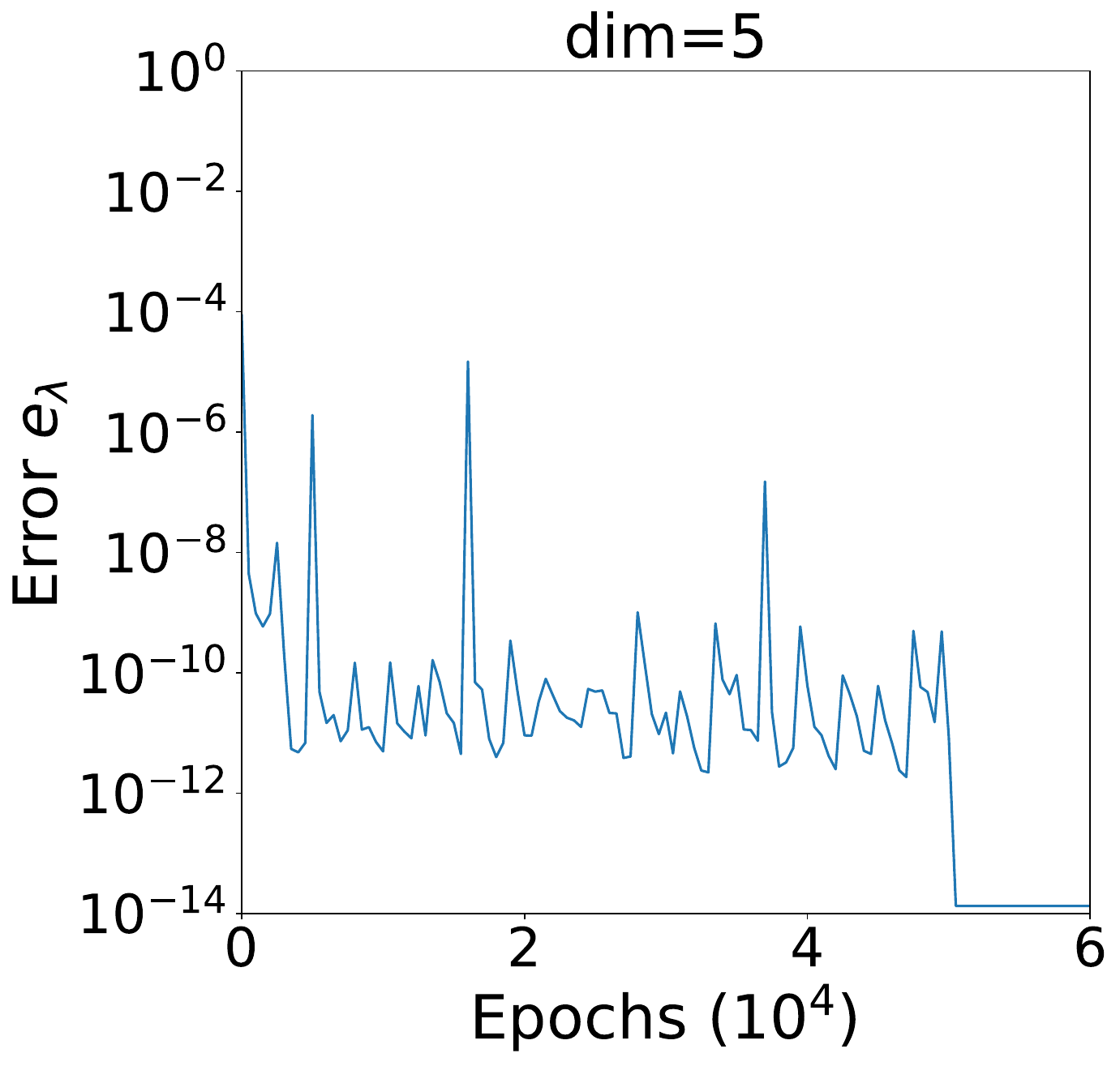}
\includegraphics[width=4.25cm,height=4cm]{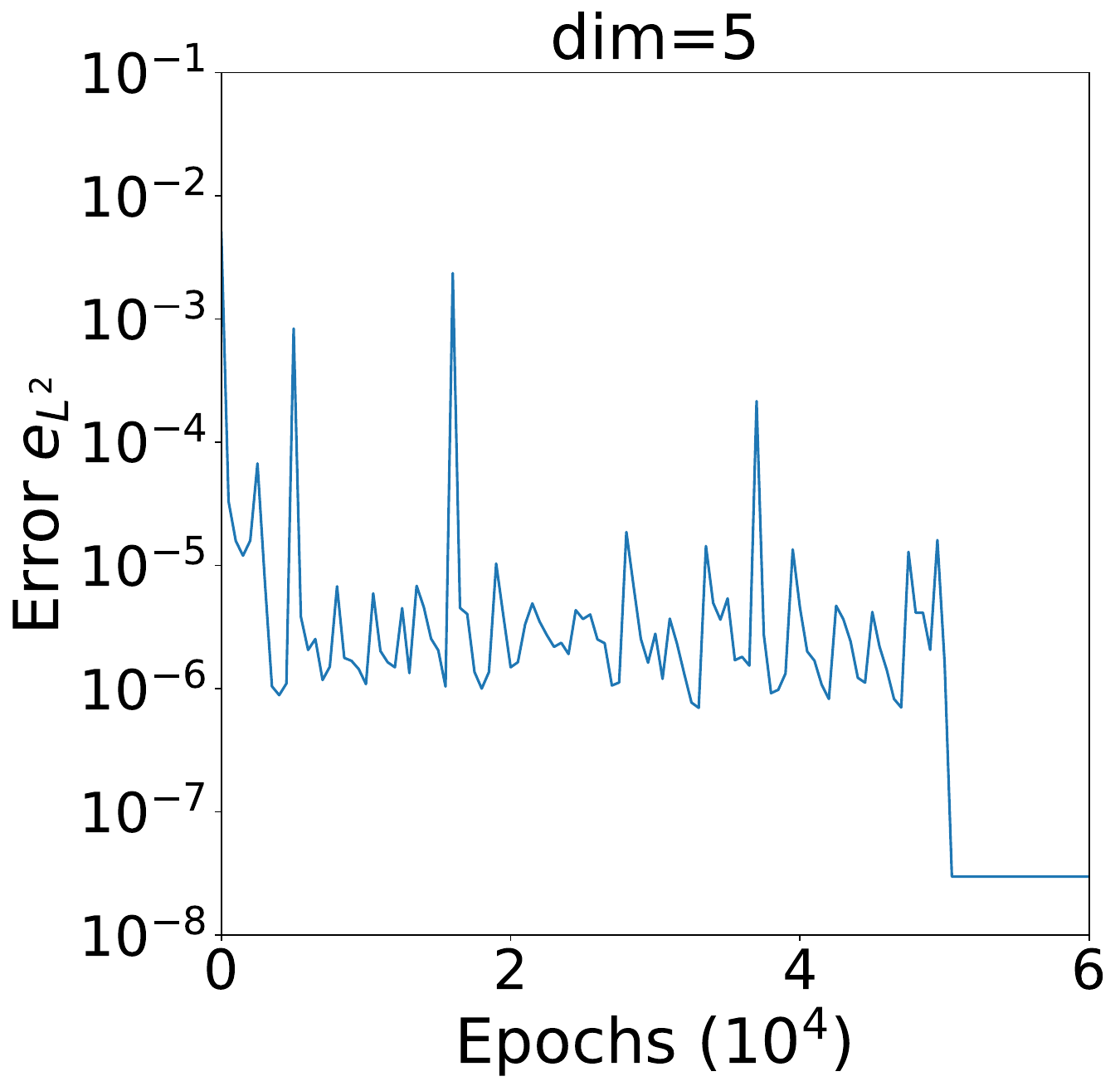}
\includegraphics[width=4.25cm,height=4cm]{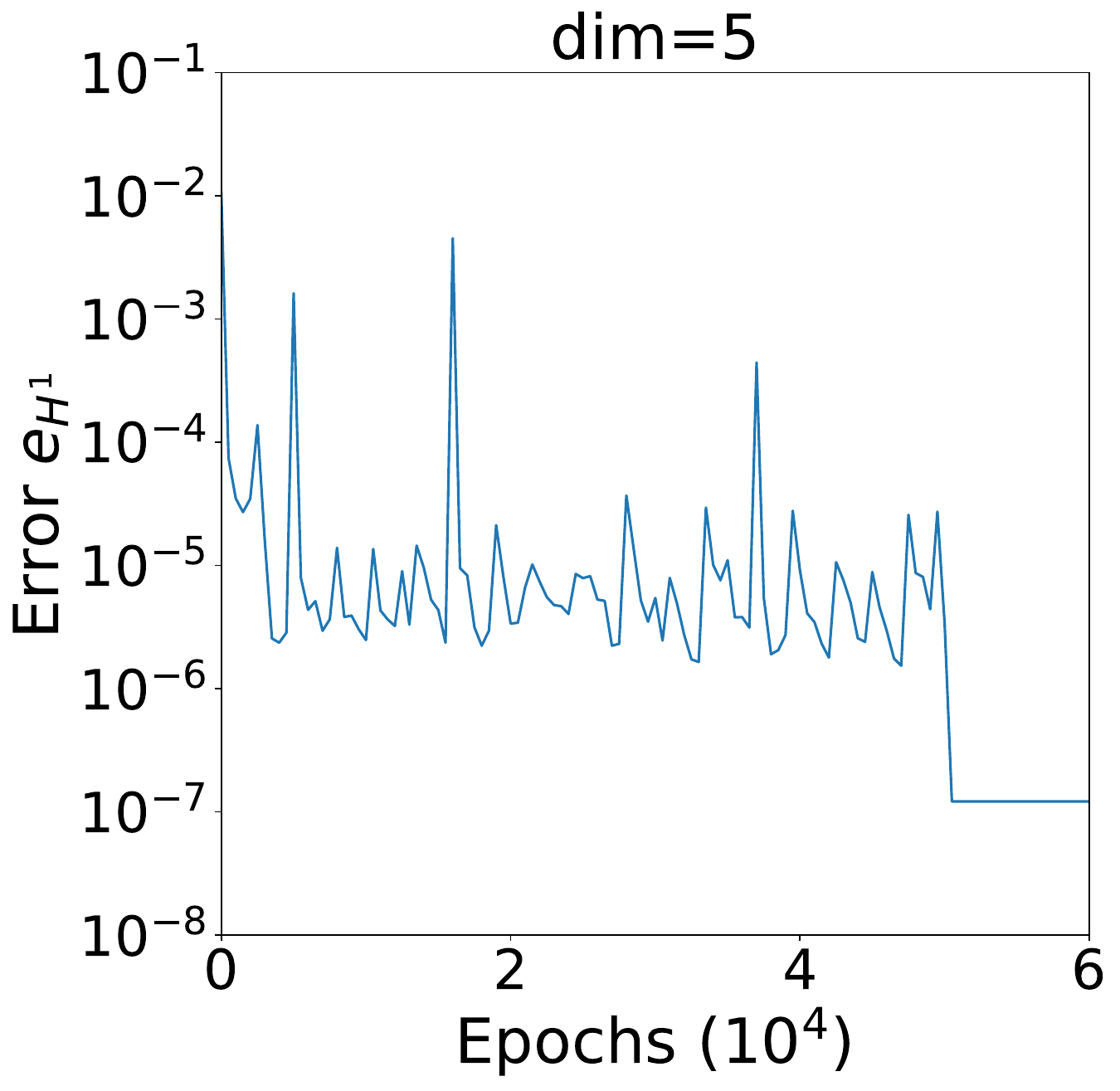}\\
\includegraphics[width=4.25cm,height=4cm]{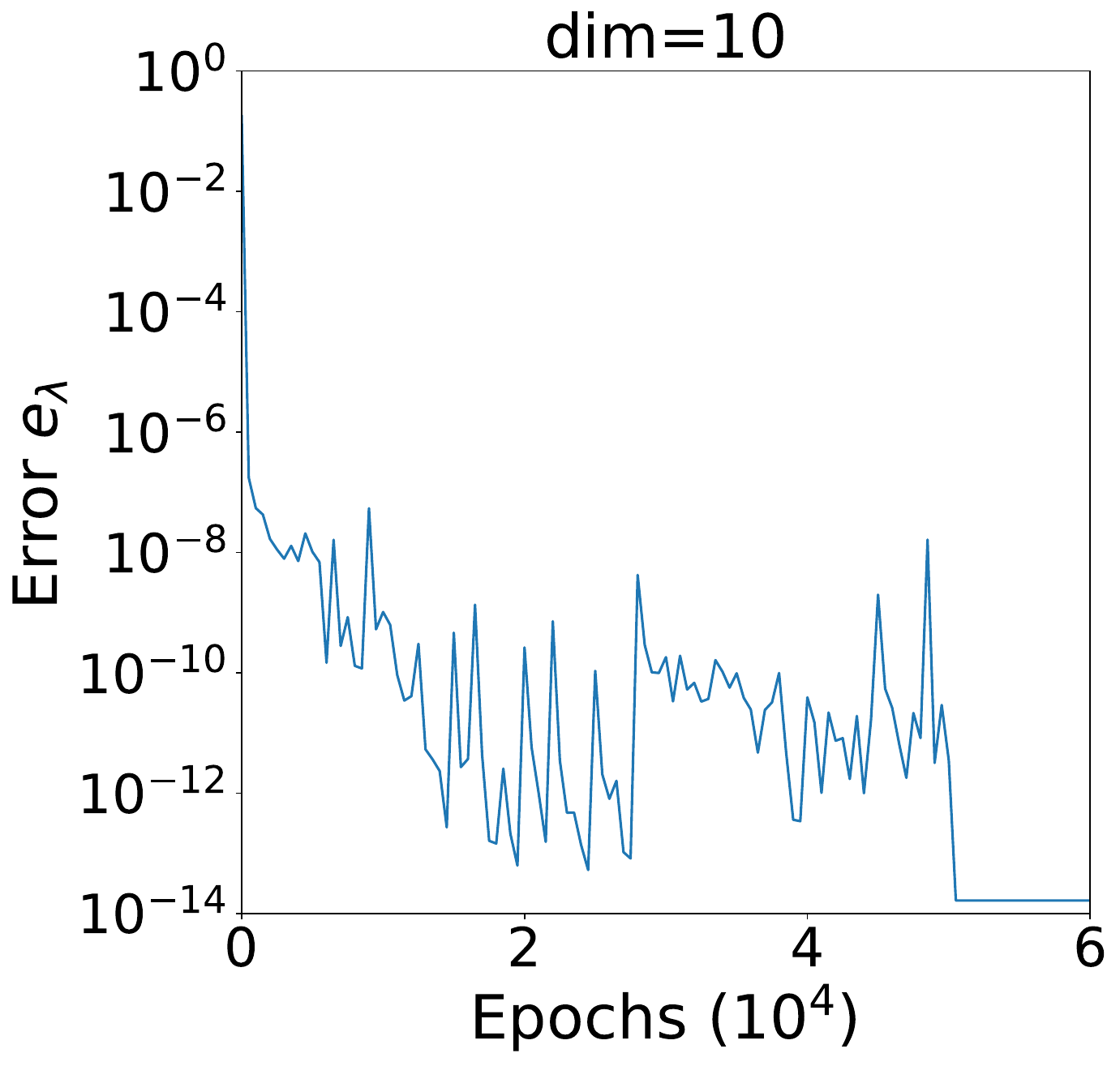}
\includegraphics[width=4.25cm,height=4cm]{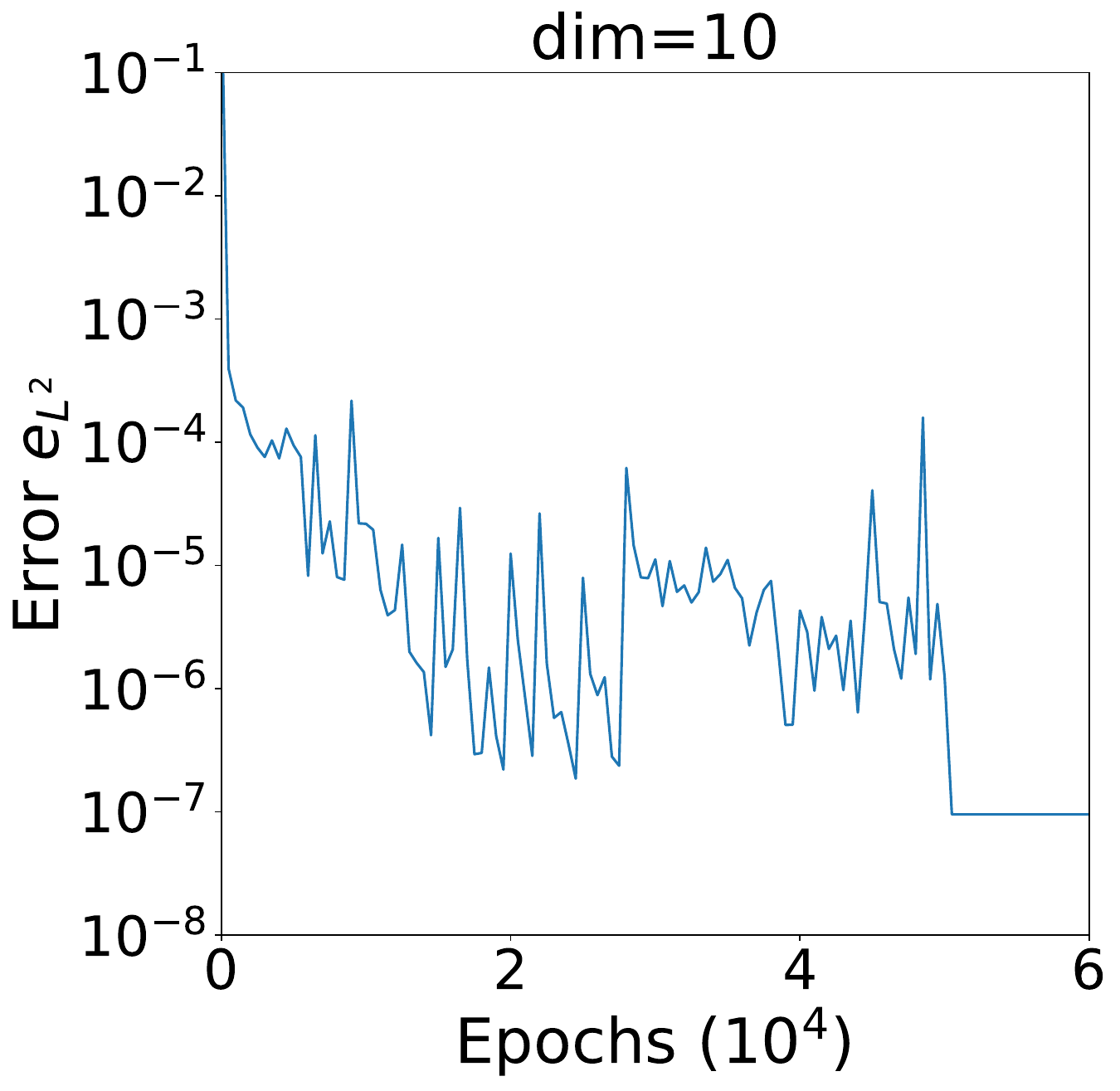}
\includegraphics[width=4.25cm,height=4cm]{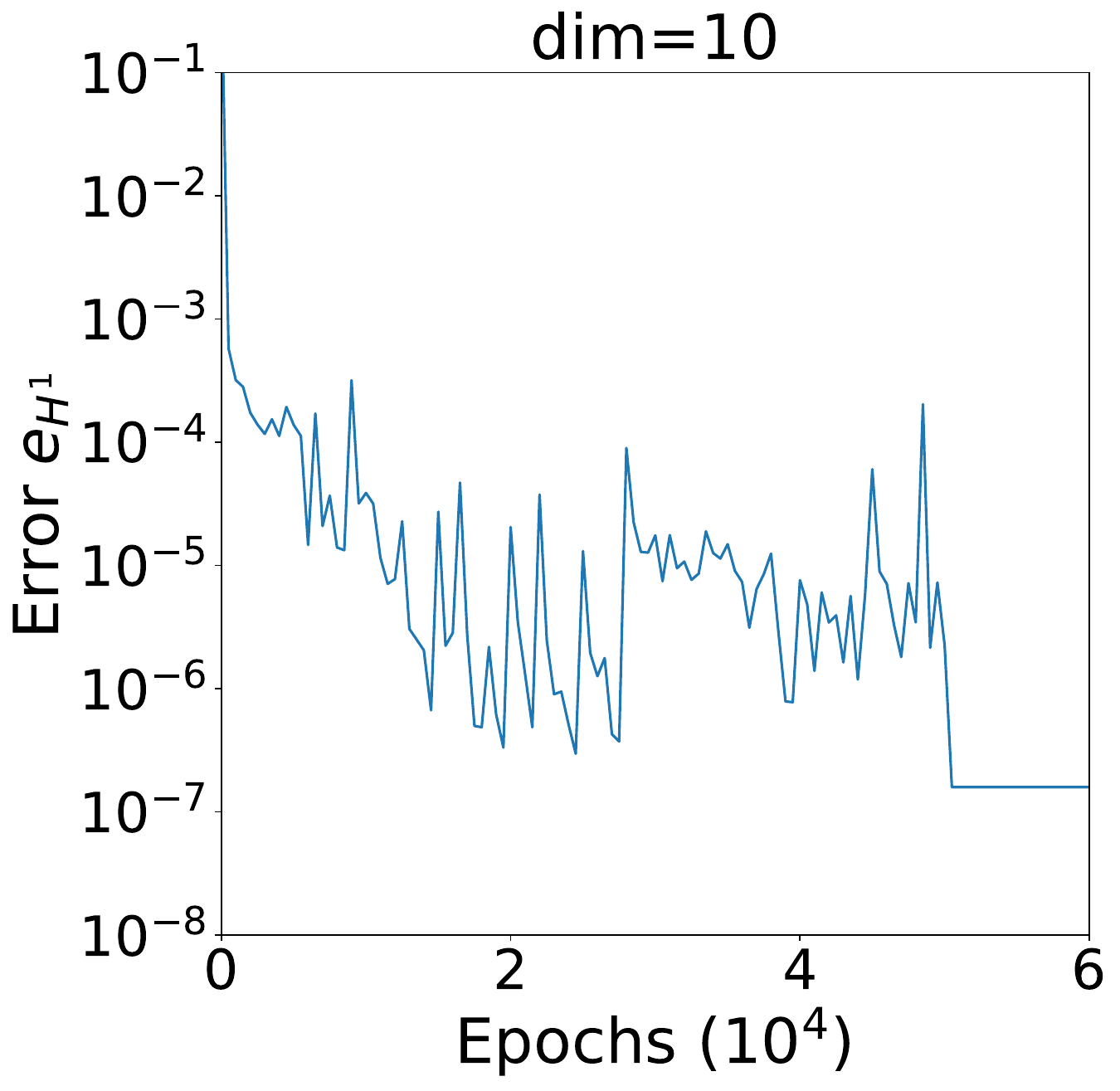}\\
\includegraphics[width=4.25cm,height=4cm]{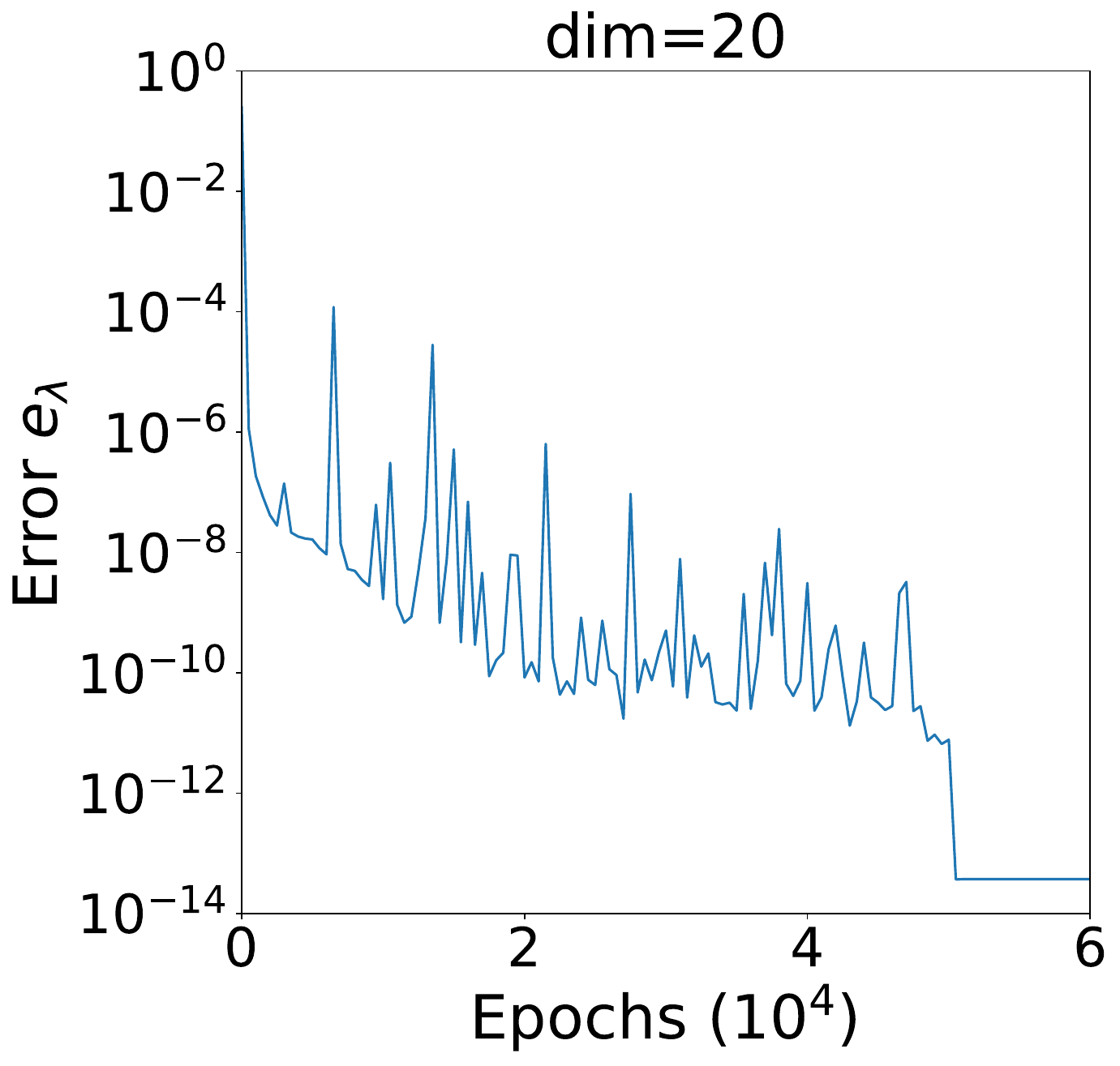}
\includegraphics[width=4.25cm,height=4cm]{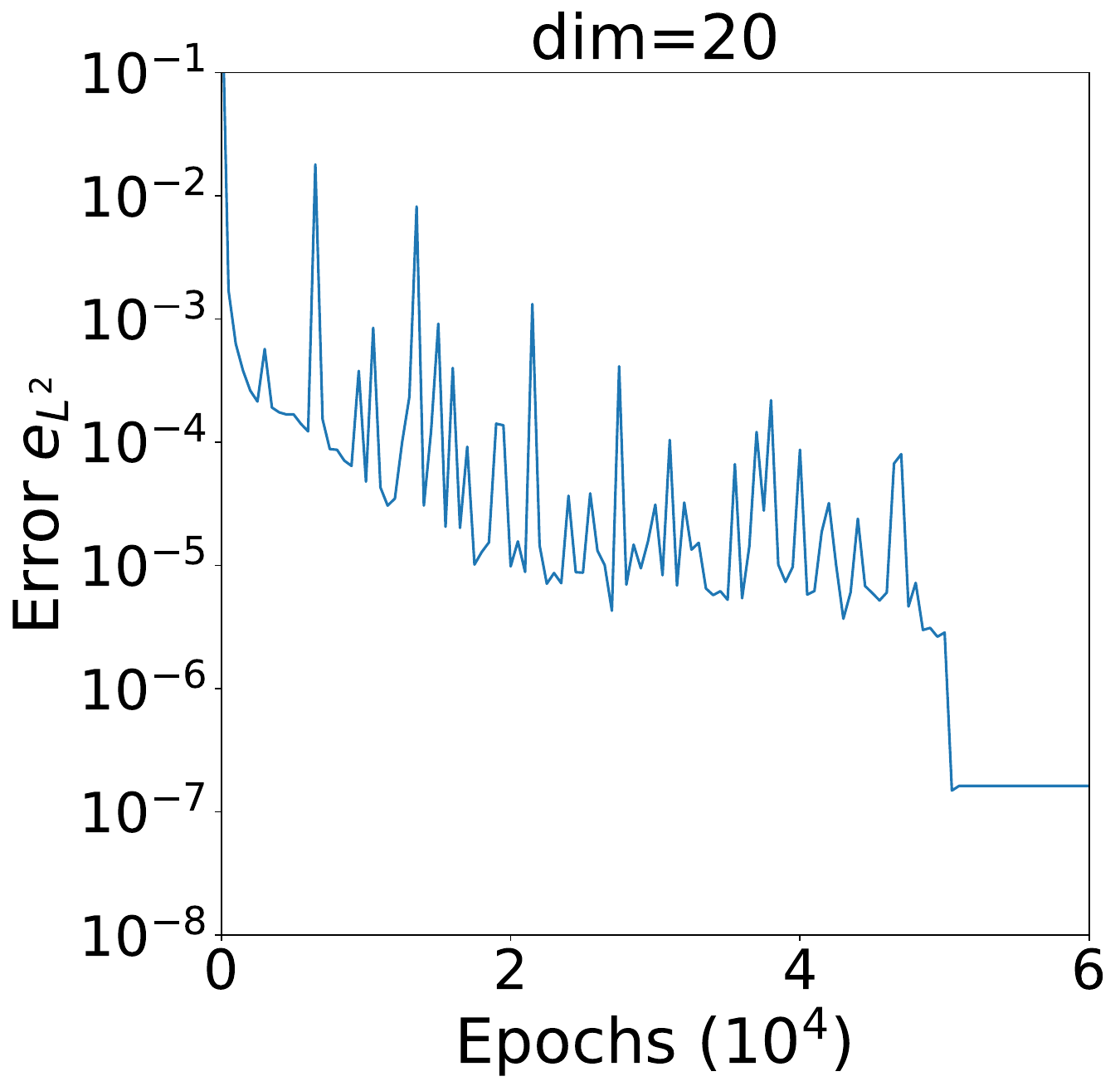}
\includegraphics[width=4.25cm,height=4cm]{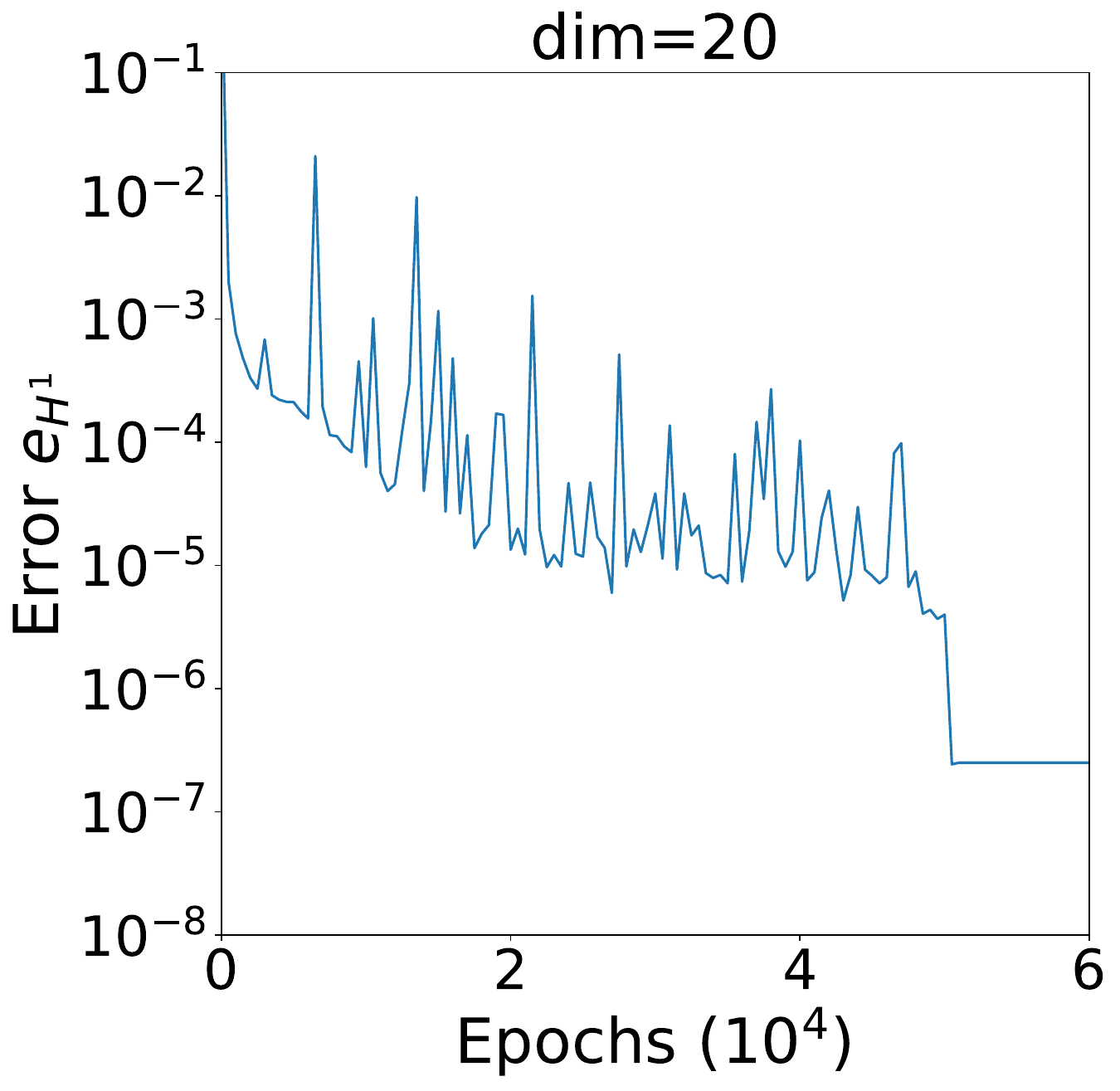}\\
\caption{Relative errors during the training process for Laplace eigenvalue problem: 
$d=5$, $10$, and $20$. The left column shows the relative errors of eigenvalue approximations, 
the middle column shows the relative $L^2(\Omega)$ errors and the 
right column shows the relative $H^1(\Omega)$ errors of eigenfunction 
approximations.}\label{fig_laplace_high}
\end{figure}
In order to show the computing efficiency of 
the proposed TNN-based method, we report 
the average wall time of each 1000 epochs on Tesla V100 and A800, respectively.  
Since the number of inner iterations of each L-BFGS step is different, 
we only report the wall time results 
for the ADAM steps, shown in Table \ref{table_laplace_time}. 
\begin{table}[htb!]
\caption{Average wall time of each 1000 ADAM steps for 
eigenvalue problem on GPUs.}\label{table_laplace_time}
\begin{center}
\begin{tabular}{ccccc}
\hline
$d$&   V100&  A800\\
\hline
5&    $24.503s/1000\ {\rm epochs}$&   $9.346s/1000\ {\rm epochs}$\\
10&   $27.508s/1000\ {\rm epochs}$&   $11.735s/1000\ {\rm epochs}$\\
20&   $29.005s/1000\ {\rm epochs}$&   $30.194s/1000\ {\rm epochs}$\\
\hline
\end{tabular}
\end{center}
\end{table}

\subsection{Eigenvalue problem with harmonic oscillator}
Then, we solve the following eigenvalue problem with harmonic oscillator potential: 
Find $(\lambda, u)$ such that
\begin{equation*}
-\Delta u+\sum_{i=1}^dx_i^2u=\lambda u,\ \ \ x\in\mathbb R^d.
\end{equation*}
The exact smallest eigenpair is $\lambda=d$ and 
$u(x)=\prod_{i=1}^d\exp\left(-\frac{x_i^2}{2}\right)$. 
Notice that the harmonic oscillator problem is defined in the whole space $\mathbb{R}^d$.
Since the potential term $\sum_{i=1}^d x_i^2$ tends to infinity as $x \to \infty$, 
it follows that the eigenfunctions of the problem tend to zero as $x \to \infty$.
We use $200$ Hermite-Gauss quadrature points in practical computations in each dimension.

For the tests with $d=5,10,20$, we choose the subnetworks of 
the TNN for each dimension to be FNNs. Each of the FNNs consists of 3 hidden 
layers with 100 neurons, and the activation function  
for each hidden layer is chosen as the sine function $\sin(x)$. 
The rank parameter for the TNN is set to $p=50$.

It should be noted that during the training process of the harmonic oscillator 
eigenvalue problem, if we directly use the a posteriori error estimator 
in the space $V_p$ as the loss function, it is prone to get stuck 
in local minima, meaning that the obtained eigenpairs may not be the smallest ones.
To overcome this issue, we first use the smallest eigenvalue 
in the space $V_p$ as the loss function, and perform a learning process 
using the ADAM optimizer for 10000 steps to obtain a good initial value. 
Then, we use the a posteriori error estimator in the space $V_p$ 
as the loss function and perform the next 10000 steps of optimization 
using the L-BFGS optimizer. The learning rate for the ADAM optimizer 
is set to $\eta_{\rm adam}=0.01$, and the learning rate for 
the L-BFGS optimizer is set to $\eta_{\rm lbfgs}=1$.
The final errors are shown in Table \ref{table_harmonic}.
The error variations during the L-BFGS training process are shown 
in Figure \ref{fig_harmonic}.

\begin{table}[!htb]
\caption{Errors of the harmonic oscillator eigenvalue problem for $d=5,10,20$.}\label{table_harmonic}
\begin{center}
\begin{tabular}{ccccc}
\hline
$d$&   $e_{\lambda}$&   $e_{L^2}$&   $e_{H^1}$\\
\hline
5&    5.370e-13&   4.131e-07&   8.664e-07\\
10&   1.245e-13&   3.213e-07&   4.733e-07\\
20&   3.523e-13&   7.768e-07&   9.746e-07\\
\hline
\end{tabular}
\end{center}
\end{table}

\begin{figure}[htb]
\centering
\includegraphics[width=4.25cm,height=4cm]{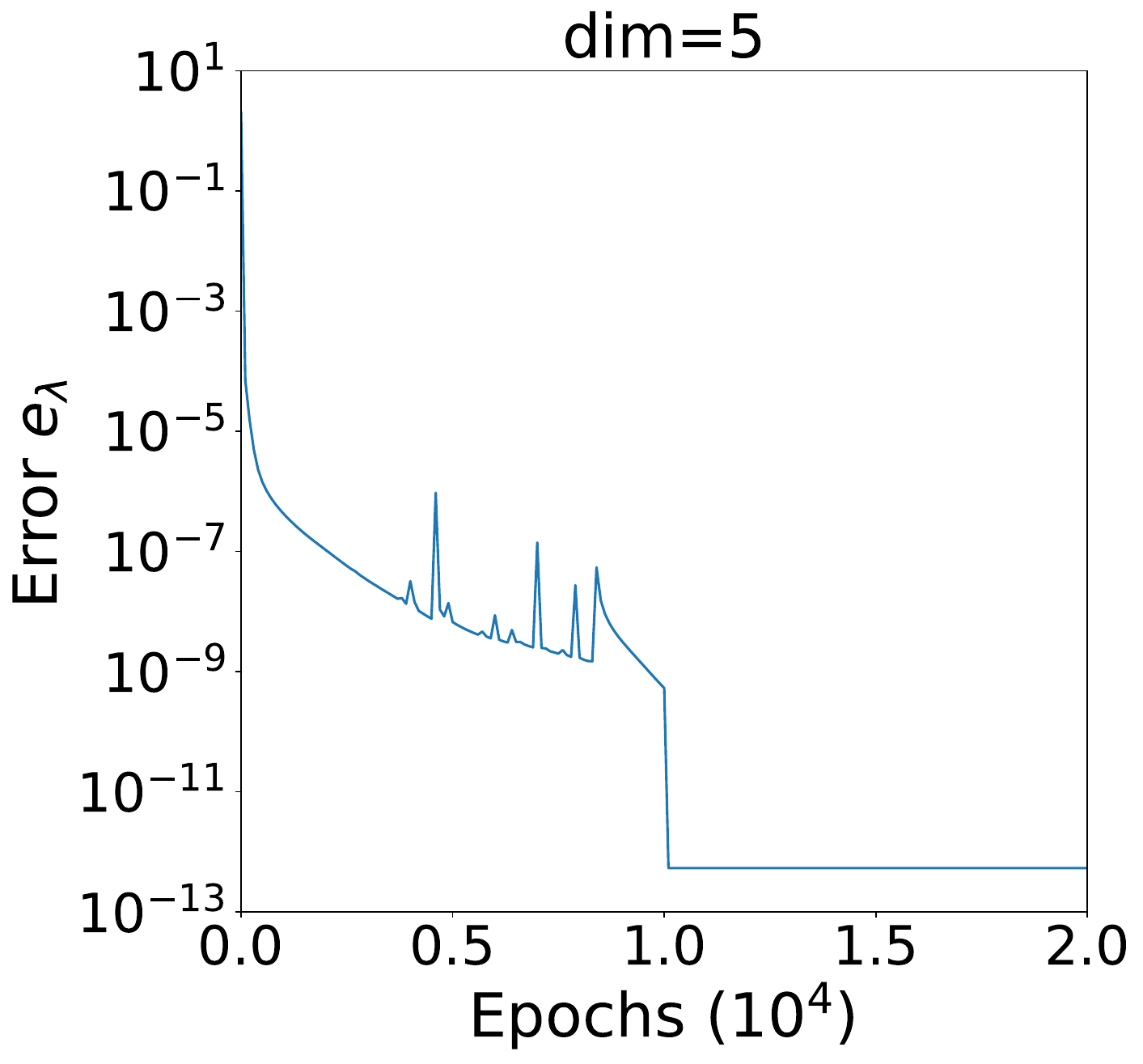}
\includegraphics[width=4.25cm,height=4cm]{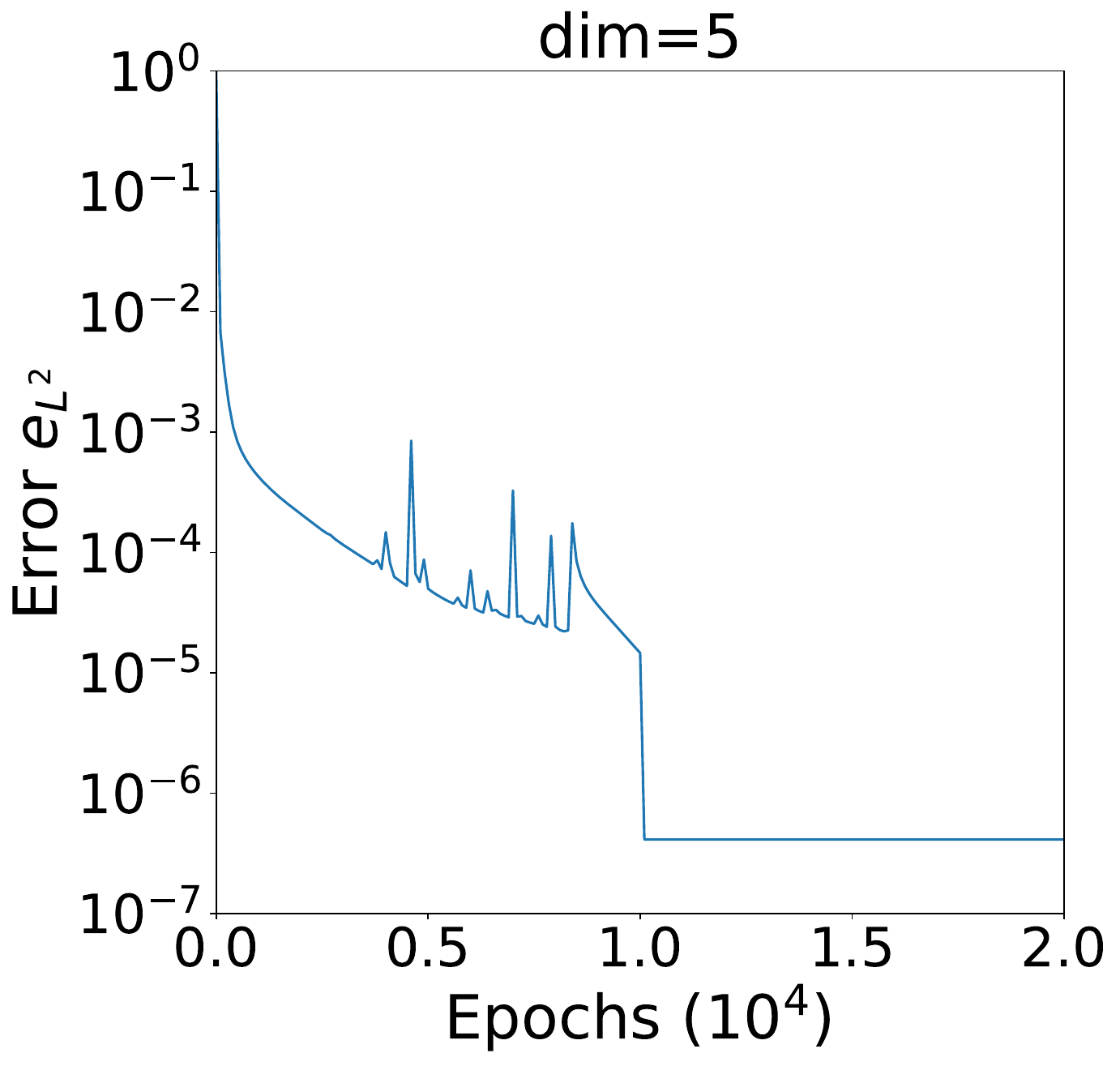}
\includegraphics[width=4.25cm,height=4cm]{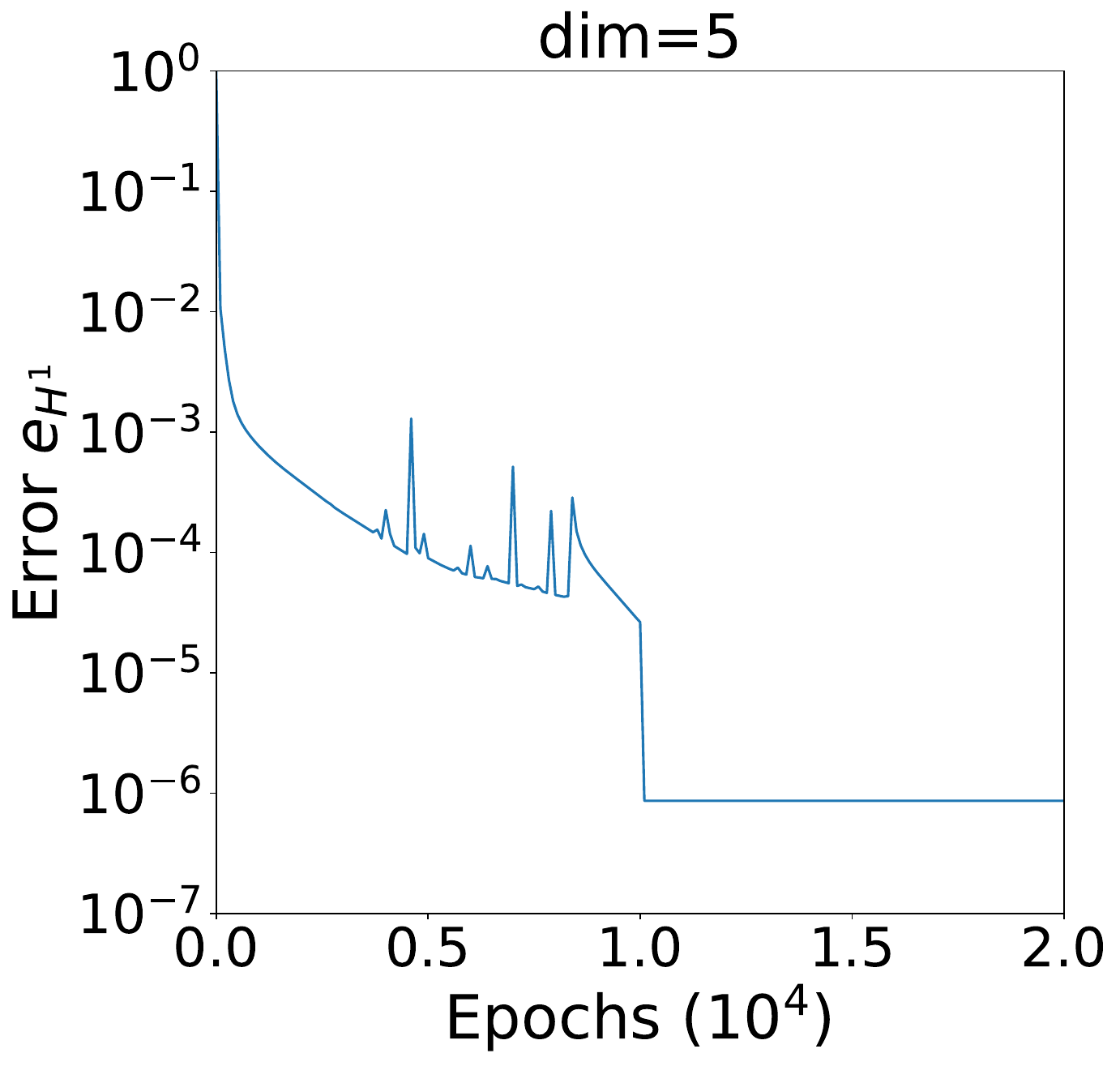}\\
\includegraphics[width=4.25cm,height=4cm]{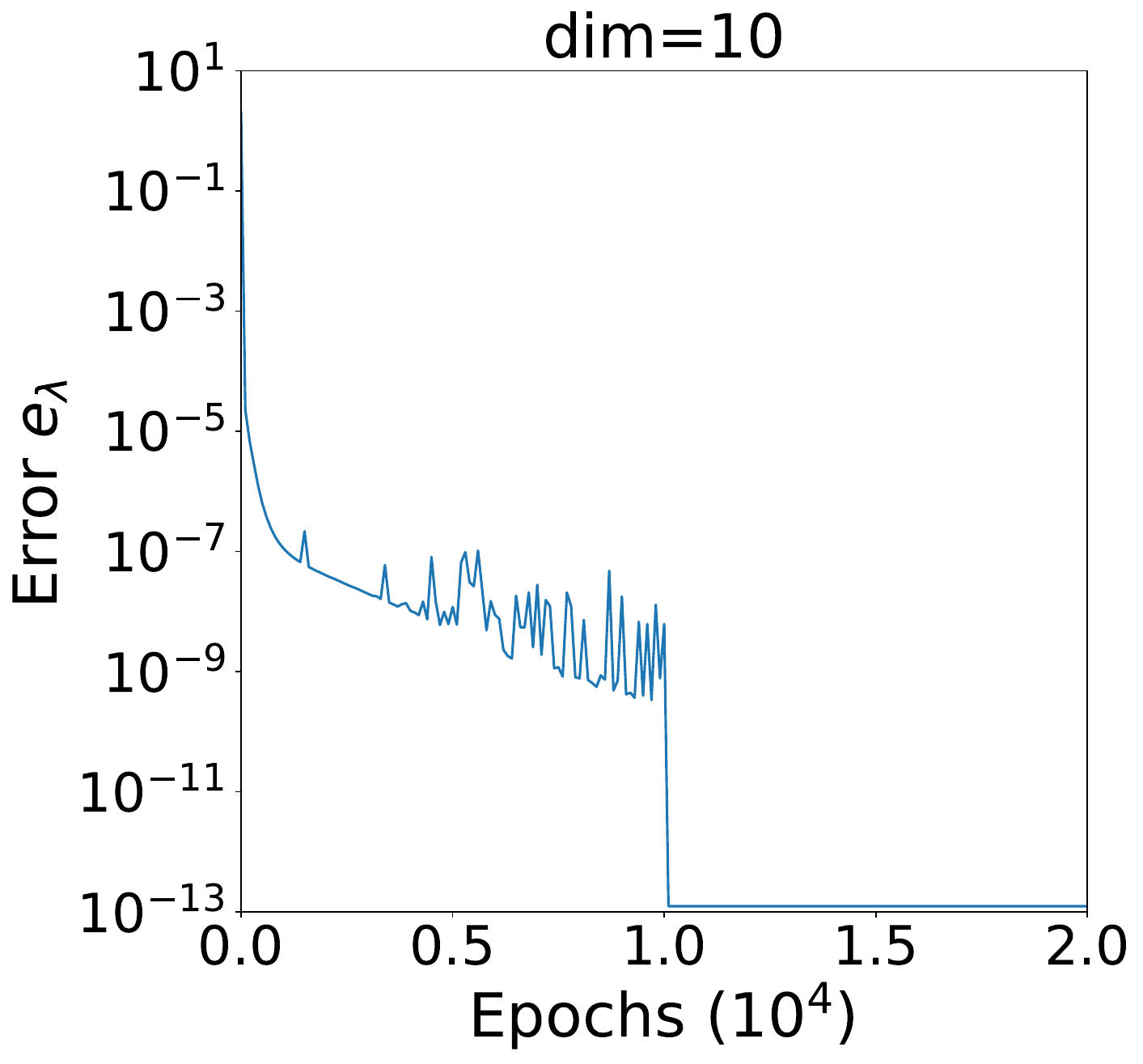}
\includegraphics[width=4.25cm,height=4cm]{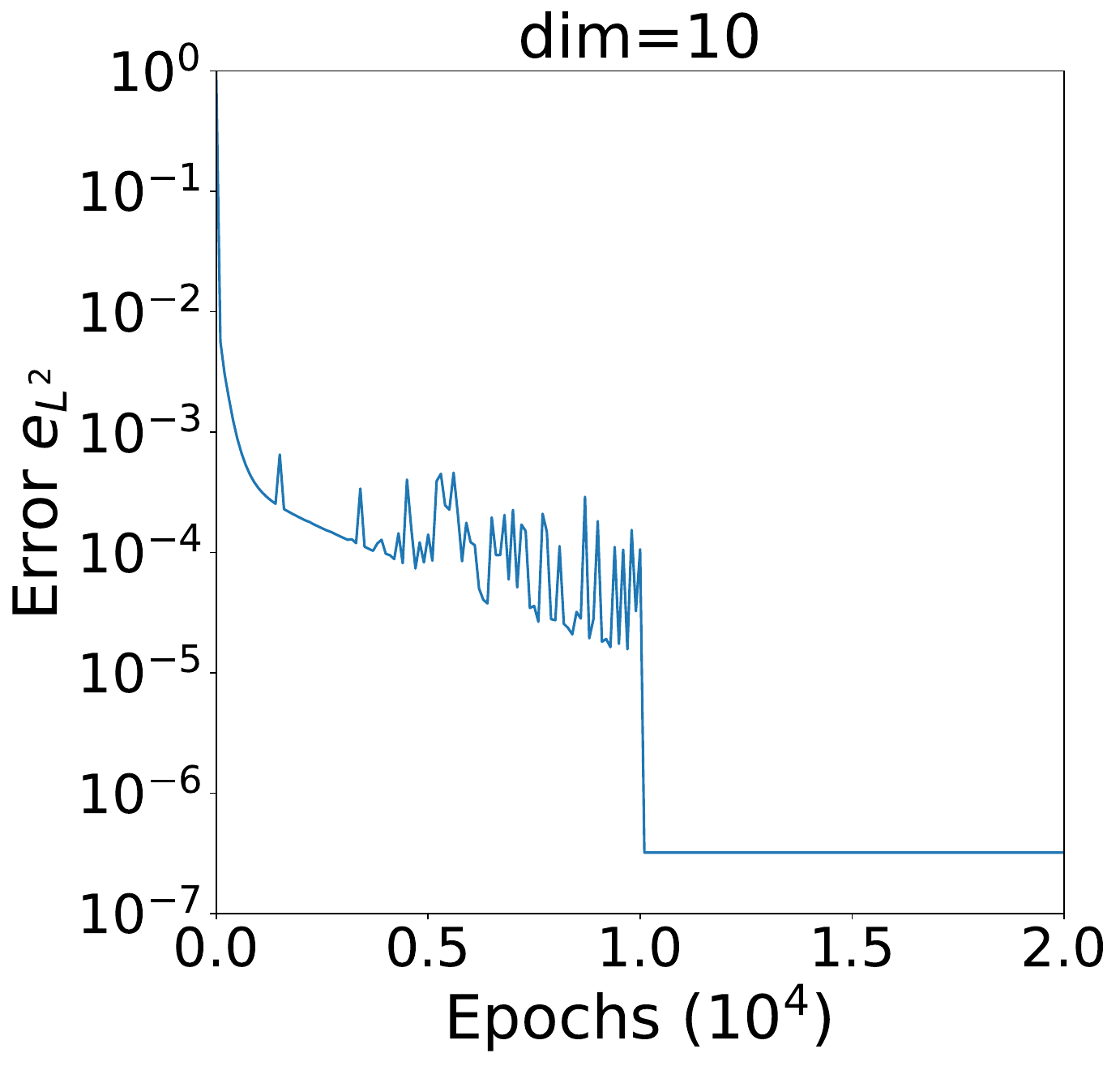}
\includegraphics[width=4.25cm,height=4cm]{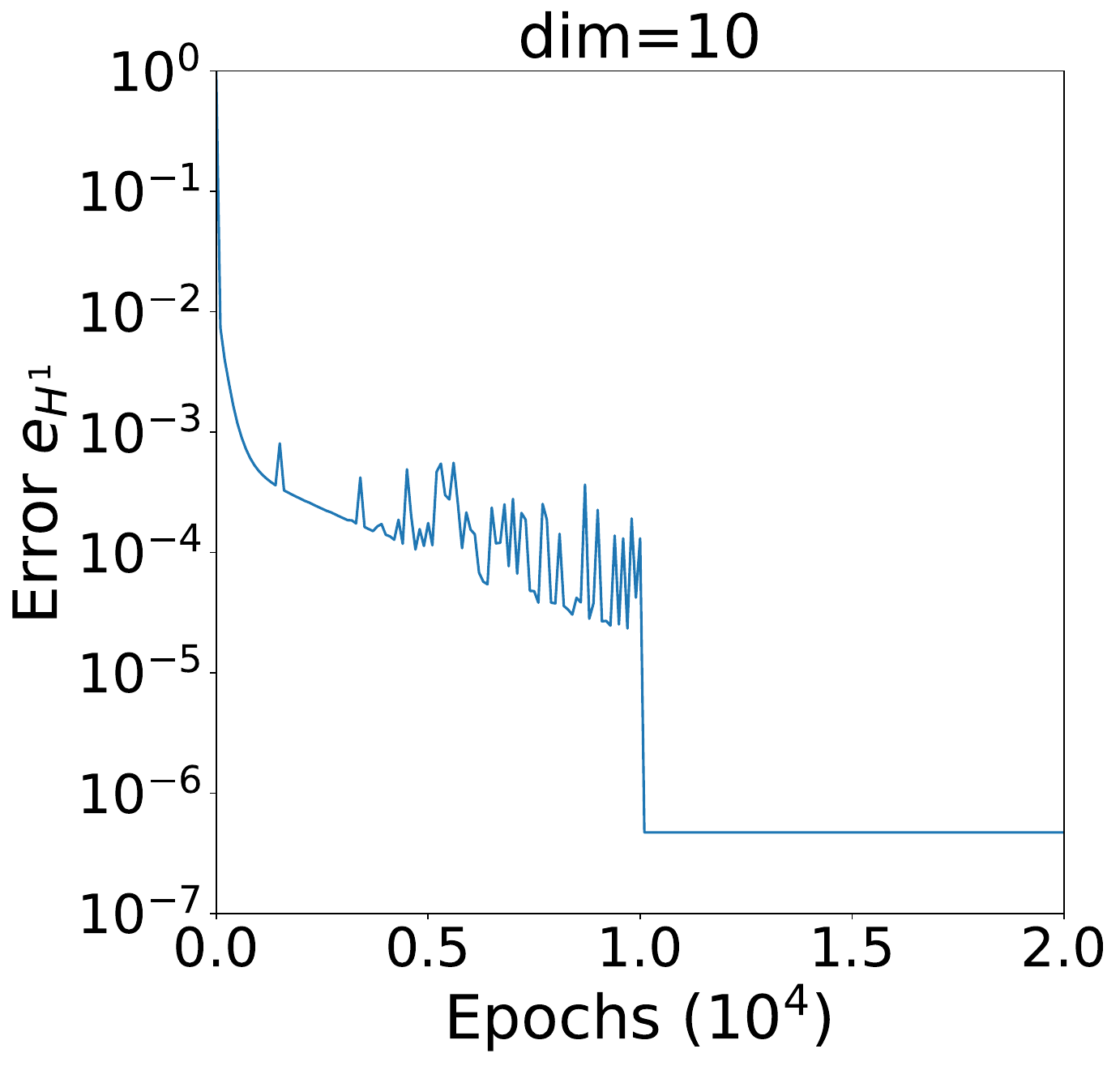}\\
\includegraphics[width=4.25cm,height=4cm]{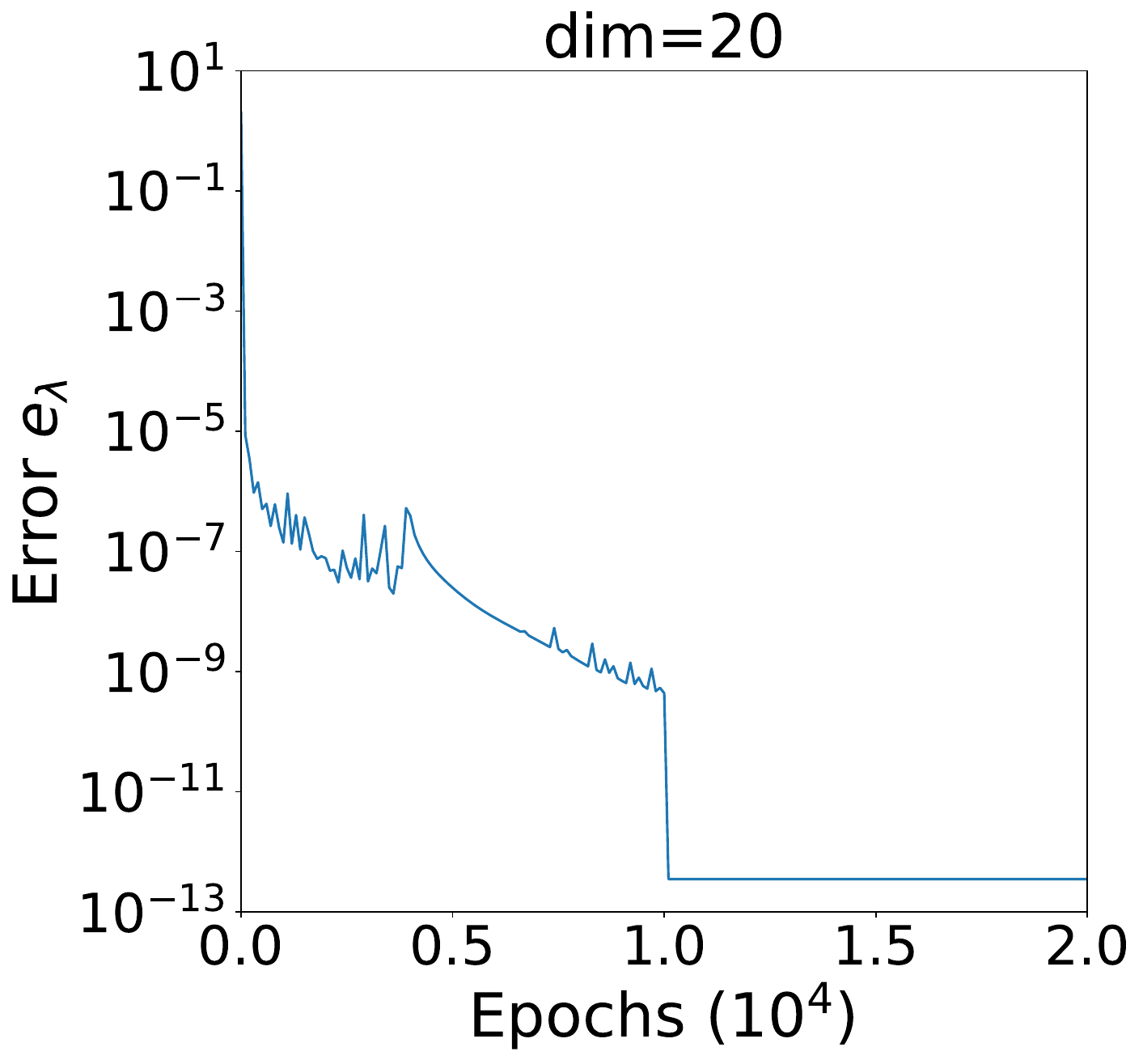}
\includegraphics[width=4.25cm,height=4cm]{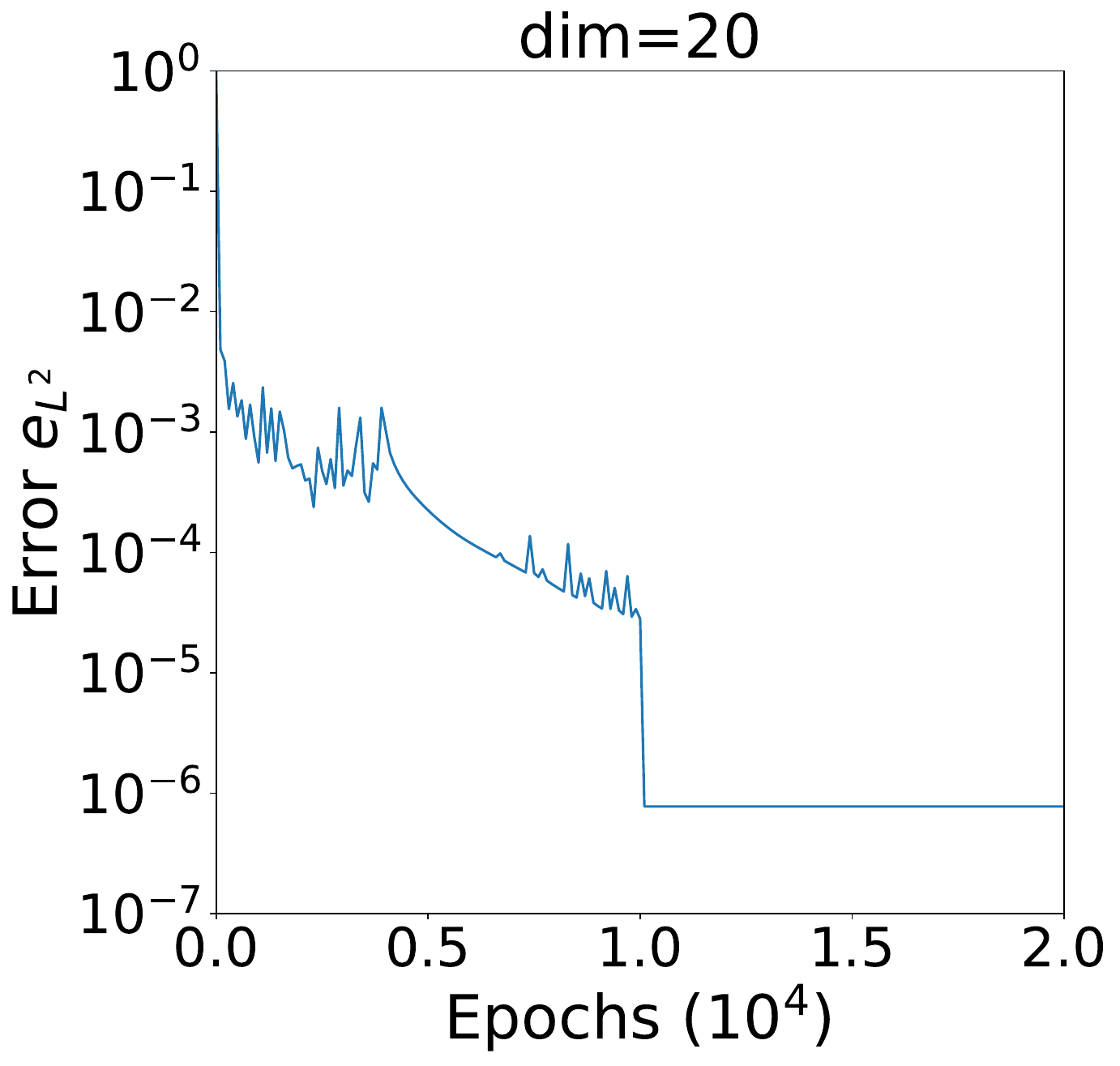}
\includegraphics[width=4.25cm,height=4cm]{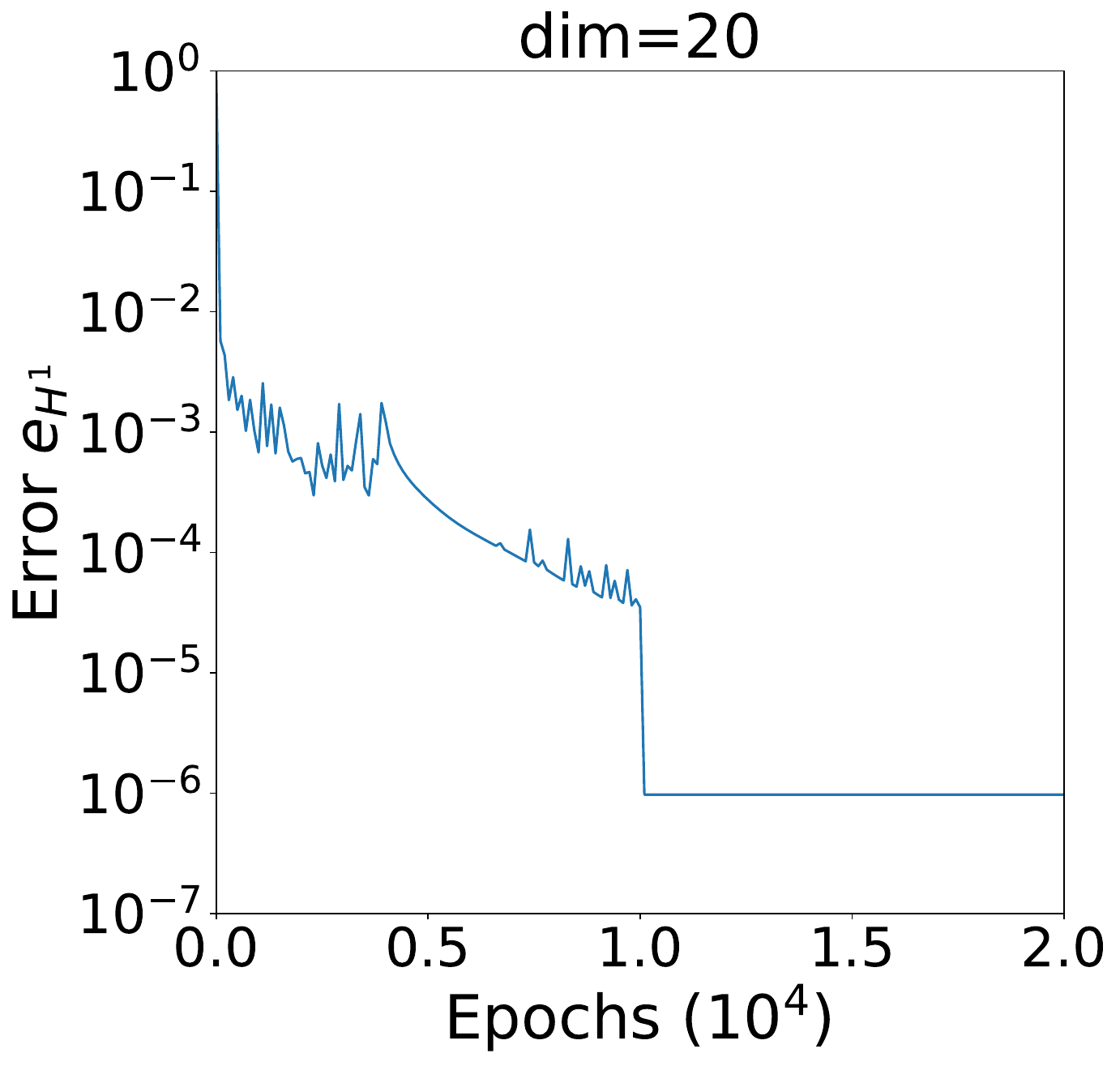}\\
\caption{Relative errors during the L-BFGS training process for harmonic 
oscillator eigenvalue problem: $d=5$, $10$, and $20$. The left column 
shows the relative errors of eigenvalue approximations, the middle 
column shows the relative $L^2(\Omega)$ errors and the right column 
shows the relative $H^1(\Omega)$ errors of eigenfunction approximations.}\label{fig_harmonic}
\end{figure}
To demonstrate the computational efficiency of the proposed method, 
we report the computation time on two different GPU models. 
Since the number of inner iterations in each L-BFGS iteration may vary significantly, 
resulting in different computation time for different tests, 
we only report the computation time for the ADAM process.
The average computation time for every $1000$ steps of the ADAM process is shown in Table \ref{table_harmonic_time}.

\begin{table}[htb!]
\caption{Average wall time of each 1000 ADAM steps for harmonic oscillator eigenvalue problem on GPUs.}\label{table_harmonic_time}
\begin{center}
\begin{tabular}{ccccc}
\hline
$d$&    V100&   A800\\
\hline
5&    $30.216s/1000\ {\rm epochs}$&   $6.510s/1000\ {\rm epochs}$\\
10&   $30.079s/1000\ {\rm epochs}$&   $8.610s/1000\ {\rm epochs}$\\
20&   $39.598s/1000\ {\rm epochs}$&   $13.372s/1000\ {\rm epochs}$\\
\hline
\end{tabular}
\end{center}
\end{table}

\section{Concluding remarks}
In this paper, we present a TNN-based machine-learning methods 
to solve the Dirichlet boundary value problem, Neumann boundary value problem, 
and eigenvalue problem of second-order elliptic operators in high dimensional space.  
With the help of high accuracy of high dimensional integrations of TNN functions, 
the a posteriori error estimator can be adopted to act as the loss function for 
the machine learning method, which is the main contribution of this paper. 
The neural network parameters are updated by minimizing 
this a posteriori error estimator which serves as an upper bound of the exact energy 
error for the concerned problems. Numerical results demonstrate that the selection of this new 
loss function improves accuracy.

Furthermore, high accuracy 
is also obtained for the homogeneous and non-homogeneous boundary value problems 
by the TNN-based machine learning methods in this paper. Compared with other 
NN-based machine learning methods, the proposed methods here do not need 
hyperparameters to balance the losses from boundary and domain, 
which saves the trouble of pruning the hyperparameters and therefore 
contributes to the optimization process.

Finally, it should also be noted that the method is currently unsuitable 
for non-self-adjoint operators. 
This topic will be addressed in our future work.

\bibliographystyle{siamplain}

\end{document}